\documentclass[a4paper,10pt]{amsart}
\usepackage{amssymb}
\usepackage{xcolor}
\usepackage [latin1]{inputenc}

\numberwithin{equation}{section}
\newtheorem{theorem}{Theorem}[section]
\newtheorem{prop}[theorem]{Proposition}
\newtheorem{lem}[theorem]{Lemma}
\newtheorem{cor}[theorem]{Corollary}

\newtheorem{rem}[theorem]{Remark}

\theoremstyle{remark}

\newcommand{\R}{\mathbb{R}}

\newcommand{\N}{\mathbb{N}}

\newcommand{\F}{\mathbb{F}}
\newcommand{\Q}{\mathbb{Q}}
\newcommand{\B}{\mathcal{B}}

\newcommand{\G}{\mathbb{G}}

\renewcommand{\S}{\mathcal{S}}

\renewcommand{\bar}{\overline}
\renewcommand{\hat}{\widehat}

\author[]{Ricardo Alonso}
\address{$^1$Texas A\&M, Science Department, Education City, Doha, Qatar}
\address{$^2$PUC-Rio, Departamento de Matem\'atica, Rio de Janeiro, Brazil}
\email{\sl ricardo.alonso@qatar.tamu.edu }

\author[]{Hajer Orf}
\address{Texas A\&M, Science Department, Education City, Doha, Qatar}
\email{\sl hajer.orf@qatar.tamu.edu }

\title[Moments and integrability of monatomic gas mixtures with long range]
{Statistical moments and integrability properties of monatomic gas mixtures with long range interactions}

\begin{document}

\begin{abstract}
This document presents \textit{a priori} estimates related to statistical moments and integrability properties for solutions of systems of monatomic gas mixtures modelled with the homogeneous Boltzmann equation with long range interactions for hard potentials.  We detail the conditions for the generation and propagation of polynomial and exponential moments, and the integrability in Lebesgue spaces.
\end{abstract}


\maketitle


\section{Introduction}
In this paper we study \textit{a priori} estimates for the homogeneous Boltzmann system describing multi-component monatomic gas mixtures for binary interactions in three dimension.  We focus in long range and hard interactions using techniques that have been developed recently for the scalar equation combining recent result and notation presented in the works \cite{GP,CGP} for the case of cutoff angular kernels.  The Cauchy problem will be addressed elsewhere.\\

The method of proof follows a classical approach where the generation and propagation of polynomial and exponential moments use, at its core, a Povzner inequality associated to the binary interaction of different monatomic species.  The Povzner inequality for monatomic species is similar to the scalar version, yet less symmetric.  This loss of symmetry poses important difficulties that modify the classical mathematical treatment.   The mathematical analysis at the level of moments, polynomial and exponential, follows ideas from refined versions for the scalar case given in \cite{GP,TAGP,NF}, however, it is worthwhile to mention that the moment analysis requires considering the system as a whole and not component by component independently.  In other words, a global moment inequality has to be made for the system involving all species at once.  This is particularly challenging for the case of exponential moment analysis which turns out to be subtle.\\

In a second stage, once the statistical moments are understood, a theory of generation and propagation of higher integrability norms is performed.  We implement an approach given in \cite{RA} for the treatment of $L^{p}$-norm generation in the range $p\in(1,\infty]$.  Interestingly, for higher integrability generation and propagation, each species of the gas mixture can be considered independently by interpreting the rest of the components as given.  This is related to the fact that the type of scattering considered in this document is only forward; in such a case only statistical moments are needed to prove higher Lebesgue integrability generation and propagation.\\   

Let us mention here that regarding the homogeneous scalar equation, the study of moments have been addressed for the long range regime with hard and Maxwell potentials interactions in references such as \cite{Wennberg,LuMouhot, TAGP, NF,PaTa}.  The generation and propagation of higher integrability for the scalar problem have been addressed in \cite{AMUXY,RA,RA1,DM} using methods initially presented in \cite{ADVW,AMUXY}.  Boundedness of solutions for the homogeneous equation has been studied in \cite{RA,Silvestre,Taskovic} using different approaches and including polynomial and exponential weights, a probabilistic numerical method can be found in \cite{FM}.  In addition, the gas mixture for monatomic components has been addressed in the so-called cutoff with hard interactions regime in references such as \cite{BDTP,GP,CGP}, the spatially inhomogeneous setting near thermal equilibrium is addressed in \cite{BriantDaus}, the BGK approximation for multi-species can be found in \cite{BC}, hydrodynamics expansions can be referred to \cite{BBBD}, chemical reactions of multi-components is addressed in \cite{DMS}, fine mathematical properties of the linear multi-component model are addressed in \cite{DJMZ}.\\

The paper is organised in the following way: In section 2, we give some notations and preliminaries and state the main results  of this manuscript. Then, in section 3, we prove Theorem \ref{Momlem}, \ref{Mom} and give some preliminary Lemma including Povzner inequality that we need to prove our first result. Section 4, is devoted to prove the uniform coercive estimate and our two last Theorem \ref{theo2} and \ref{CorLp}. Finally, in section 5, we give some general Lemma which help us to prove ours results.\\

\noindent
\textbf{Acknowledgments.} R. Alonso gratefully acknowledges the support from Conselho Nacional de Desenvolvimento Cient\'ifico e Tecnol\'ogico (CNPq), grant Bolsa de Produtividade em Pesquisa (303325/2019-4).
     
\section{Preliminaries and Main Results}
\subsection{Preliminaries}
Let us consider a gas mixture of $I$ species where each component of the mixture, namely $A_i$ with $1\leq i\leq I$, is described by its distribution density function denoted with $f_i:= f_i(t,v)\geq 0$ depending on time $t>0$ and particle velocity $v\in \R^3$.  The evolution of the gas mixture $\F = [f_i]_{1\leq i\leq I}$ satisfies the Boltzmann system
\begin{equation}\label{Bolt1}
\partial_t \F(t,v)=\Q(\F,\F),
\end{equation}
where $\Q(\F,\F)$ is the vector of collision operator defined for each component as
\begin{equation}\label{ColOp1}
[\Q(\F,\F)]_i :=\sum_{j=1}^I Q_{ij}(f_i,f_j)\,.
\end{equation}
Here the collision Boltzmann operators $Q_{ij}(f_i,f_j)$ measure the influence that the species $f_j$ exerts on the species $f_i$.  Consequently, for any $i$ fixed the distribution function $f_i$ solves a Boltzmann equation where the collision operator takes into account the influence of all the species $A_i$ as
\begin{equation}\label{Bolt-i}
\partial_t f_i(t,v)=[\Q(\F,\F)]_i=\sum_{j=1}^I Q_{ij}(f_i,f_j)\,,\qquad 1\leq i \leq I\,.
\end{equation}

\subsubsection{Collision process}
Consider two interacting particles from the species $A_i$ with associated particle mass $m_i>0$ and $A_j$ with associated particle mass $m_{j}>0$.   The shorthand $('\!v_{ij},\, '\!v_{*,ij})$ denotes the particles velocities before collision.  The dependence of such pre-collisional velocities with respect to the post-collisional velocities $(v,v_*)$ is given by 
\begin{align}
\begin{split}\label{collision-laws}
'\!v_{ij}&=\frac{m_iv+m_jv_*}{m_i+m_j}+\frac{m_j}{m_i+m_j}|v-v_*|\sigma = v - \frac{2\,m_j\,u^{-}}{m_{i}+m_j}\\
'\!v_{*,ij}&=\frac{m_iv+m_jv_*}{m_i+m_j}-\frac{m_i}{m_i+m_j}|v-v_*|\sigma =v_* + \frac{2\,m_i\,u^{-}}{m_{i}+m_j}\,.
\end{split}
\end{align}
Here $\sigma\in\mathcal{S}^{2}$ is the scattering direction and $u^\pm = \frac{u\pm|u|\sigma}{2}$. After collision, the species remain in the  same species with the same associated mass.  The collision laws \eqref{collision-laws} conserve momentum and kinetic energy, that is, the interactions are elastic.  In this way one has that pre and post collisional velocities are interchangeable $('\!v_{ij},\, '\!v_{*,ij})= (v'_{ij},\, v'_{*,ij})$ and the conservation laws
\begin{align}
\begin{split}\label{conservation-laws}
m_i v'_{ij}+m_j v'_{*,ij}&=m_iv+m_jv_* \qquad \text{conservation of momentum}\,,\\
m_i|v'_{ij}|^2+m_j|v'_{*,ij}|^2&=m_i|v|^2+m_j|v_*|^2 \qquad \text{conservation of energy}\,,
\end{split}
\end{align}
are valid.  Now, note that in the collision laws \eqref{collision-laws} (or in the conservation laws \eqref{conservation-laws}) one can divide the laws by $\sum_{j=1}^I m_{j}$ giving normalised component masses $\tilde m_{i} = \frac{m_{i}}{\sum_{j=1}^I m_{j}}\in(0,1)$, for each species $1\leq i\leq I$.  For such reason, one can assume without loss of generality that the component masses add up to unity $\sum_{j=1}^I m_j=1.$

\begin{rem}
To simplify notation we eliminate the subindices from the shorthand $(v'_{ij},v'_{*,ij})$ and simply write $(v',v'_*)$.  Just keep in mind that the post and pre collisional relations will differ for each collision operators $Q_{ij}$.
 \end{rem}

\subsubsection{Collision operators} The collision operator $Q_{ij}$ describing the binary interaction between the particles of species  $A_i$, denoted by $f$, and the particles of species $A_j$, denoted by $g$, is explicitly defined by
\begin{equation}\label{Q1}
Q_{ij}(f,g)(v)=\int\int_{\R^3\times \S^2} \big(f(v')g(v'_*)-f(v)g(v_*)\big)B_{ij}(|u|, \widehat{u}\cdot\sigma)d\sigma dv,
\end{equation}
where $u=v-v_*$ is the relative velocity between interacting particles, $\widehat{u}=\frac{u}{|u|}$ its direction, $\sigma=\frac{u'}{|u'|}$ the scattering direction, and $B_{ij}$ the collisional cross section.  The $B_{ij}$ are positive functions satisfying the \textit{symmetric condition} $B_{ij}=B_{ji}$ and the following micro-reversibility assumption
\begin{equation}\label{B0}
B_{ij}(v,v_*,\sigma)=B_{ij}(v',v'_*, \sigma')=B_{ij}(v_*,v,-\sigma),
\end{equation}
with $\sigma'=\frac{u}{|u|}$.

In this document we concentrate our efforts in hard potential interactions without angular cutoff.  More precisely, the collision cross sections $B_{ij}$, for $1\leq i,j \leq I$, are assumed to satisfy the explicit expression
\begin{equation}\label{B1}
B_{ij}(|u|, \widehat{u}\cdot\sigma)=|u|^{\lambda_{ij}}b_{ij}(\widehat{u}\cdot\sigma),\qquad \lambda_{ij}\in(0,2]\,,
\end{equation}
where $b_{ij}$ is the angular scattering kernel.  Writing $\widehat{u}\cdot\sigma=\cos\theta$ with $\theta\in(0,\frac\pi2]$, the angular scattering satisfies the condition 
\begin{equation}\label{B2}
\exists\, \kappa^1_{ij},\kappa^2_{ij}\in (0,\infty) \qquad \kappa^1_{ij}\,\theta^{-s_{ij}-1}\leq \beta_{ij}(\theta)=\sin\theta\, b_{ij}(\cos\theta)\leq\kappa^2_{ij}\,\theta^{-s_{ij}-1},\quad s_{ij}\in(0,2).
\end{equation}
The condition on the support of $b$ implies that the regime we treat in this document is the so-called peaked forward scattering.  Note that the symmetric condition implies that
\begin{equation*}
\lambda_{ij}=\lambda_{ji} \qquad \text{and} \qquad b_{ij}=b_{ji} \quad (\text{that is} \;\, s_{ij}=s_{ji})\,.
\end{equation*}
The following parameters, related to the minimum and the maximum value of the potential rates $\lambda_{ij}$, will be important when considering the different properties of solutions,
$$\overline{\lambda}_i:=\min_{1\leq j\leq I}\lambda_{ij},\qquad \overline{\overline{\lambda}}_i:=\max_{1\leq j\leq I}\lambda_{ij}, $$
$$\overline{\lambda}:=\min_{1\leq i,j\leq I}\lambda_{ij},\qquad \overline{\overline{\lambda}}:=\max_{1\leq i,j\leq I}\lambda_{ij},$$ and $$\lambda^\natural:=\min_{1\leq i\leq I}\max_{1\leq j\leq I}\lambda_{ij}\,.$$
Additionally, we will consider the following parameters related to the angular singularities,
$$\overline{s}_i:=\min_{1\leq j\leq I}s_{ij},\qquad \overline{\overline{s}}_i:=\max_{1\leq j\leq I}s_{ij},$$ 
$$\overline{s}:=\min_{1\leq i,j\leq I}s_{ij},\qquad \overline{\overline{s}}:=\max_{1\leq i,j\leq I}s_{ij}\,,$$
and $$s^\natural:=\min_{1\leq i\leq I}\max_{1\leq j\leq I} s_{ij}.$$
Before continuing, let us mention that condition \eqref{B2} is only needed in a vicinity of $\theta=0$.  Away from this point $b$ can be assumed just integrable.  Additionally, let us remark that for the scalar Boltzmann equation the support of the angular scattering $b(\cos(\theta))$ can be assumed in the upper scattering sphere $\theta\in(0,\pi/2]$ with no loss of generality since $v'\leftrightarrow v'_{*}$ under the change $\sigma\leftrightarrow-\sigma$.  This is not the case for monatomic gas mixtures due to lack of symmetry in the collision laws.  We point out, however, that the arguments and results of Section 3 are all valid, without modification, with a kernel $b(\cos(\theta))$ defined in $(0,\pi]$.   Morevover, from a physical perspective, the most generic assumption for the backscattering is $b(\hat{e}\cdot \sigma)1_{ \{ \hat{e}\cdot \sigma\leq0 \}} \in L^{1}_{\sigma}(\mathcal{S}^{2})$.  In other words, terms associated to the backscattering are of lower order in comparison to the forward and can be included in the analysis of Section 4 at a technical price, which may slightly impact some of the arguments in this section.  For the interested reader, we refer to \cite{CGP} for the nuances of the cutoff scattering analysis. 
\subsubsection{Functional spaces}
First, we recall some notations and definitions that will be important along the manuscript.
\begin{itemize}
\item[$\cdot$] The mixture's bracket form is defined as
\begin{equation*}
\langle v\rangle_i:=\sqrt{1+\frac{m_i}{\sum_{j=1}^I m_{j}}|v|^2}= \sqrt{1+m_i\,|v|^2},\qquad\qquad v\in\R^3\,,\quad 1\leq i\leq I\,.
\end{equation*}
\item[$\cdot$] The scalar $k^{th}$-polynomial moment, with $k\geq 0$, is defined for any $0\leq f \in L^{1}_{k}$ as
\begin{equation*}
{\bf m}_{k,i}[f](t):= \int_{\R^3} f(t,v)\langle v\rangle_i^k dv\,,\qquad 1\leq i \leq I\,.
\end{equation*} 
For the gas mixture $\F=[f_i]_{1\leq i\leq I}$ we use the shorthand
\begin{equation*}
{\bf m}_{k,i}(t)={\bf m}_{k,i}[f_i](t):= \int_{\R^3} f_{i}(t,v)\langle v\rangle_i^k dv\,,
\end{equation*} 
that is, the subindex $i$ in the shorthand ${\bf m}_{k,i}(t)$ prescribes the bracket and the associate component in $\F$.   The cumulative $k^{th}$ moment of the gas mixture is defined as
 \begin{equation*}
 {\bf m}_k[\F](t):= \sum_{i=1}^I {\bf m}_{k,i}[f_i](t)=  \sum_{i=1}^I\int_{\R^3} f_{i}(t,v)\langle v\rangle_i^k dv.
\end{equation*}
\end{itemize}
The Banach $L^p_n$ spaces associated to the mixture are defined as follow, refer to \cite{CGP},
$$L^p_n=\Big\{\F=[f_i]_{1\leq i\leq I}\,\Big|\, \sum_{i=1}^I \int_{\R^3}\Big(\langle v\rangle_i^n |f_i(v)|\Big)^pdv<\infty,\; n\geq 0,\; 1\leq p <\infty \Big\}$$
with associated norm
$$\|\F\|_{L^p_n}:= \bigg(\sum_{i=1}^I\int_{\R^3}\Big(\langle v\rangle_i^n |f_i(v)|\Big)^p dv\bigg)^{\frac{1}{p}}.$$
For the special case $p=+\infty$, the space is defined as
$$L^\infty_n=\Big\{\F=[f_i]_{1\leq i\leq I} \, \Big| \, \sum_{i=1}^I {\rm ess}\,\sup \langle  v\rangle_i^n |f_i(v)| <\infty,\; n\geq 0\Big\}$$
with associated norm
$$\|\F\|_{L^\infty_n}:=\sum_{i=1}^I {\rm ess}\sup\,\langle v\rangle_i^n |f_i(v)|\,.$$
The $L\log L$ space is defined as 
$$L\log L=\Big\{\F=[f_i]_{1\leq i\leq I} \, \Big| \, \sum_{i=1}^I \int_{\R^3} |f_i(v)|\log\big(1 + |f_i(v)|\big)dv<\infty \Big\}$$
with associated norm
$$\|\F\|_{L\log L}:= \sum_{i=1}^I\int_{\R^3} |f_i(v)|\log\big(1 + |f_i(v)| \big)dv.$$
We also work with the Sobolev spaces 
$$H^k:=\Big\{\F=[f_i]_{1\leq i\leq I} \,\Big| \, \sum_{i=1}^I \int_{\R^3}\Big|\big(1+(-\Delta)\big)^\frac{k}{2} f_i(v)\Big|^2dv<\infty,\; k\geq 0\Big\}\,,$$
endowed with the norm 
$$\|\F\|_{H^k}=\Big(\sum_{i=1}^I \int_{\R^3}\Big|\big(1+(-\Delta)\big)^\frac{k}{2} f_i(v)\Big|^2dv\Big)^\frac{1}{2}=\Big(\sum_{i=1}^I \int_{\R^3}\big(1+|\xi|^2\big)^k |\mathcal{F}(f_i)(\xi)|^2d\xi\Big)^\frac{1}{2}\,.$$
\subsubsection{Weak form of the Boltzmann collision operator} Testing the collision operator against a suitable test vector function $[\psi_i]_{1\leq i\leq I}$ it holds that
\begin{equation*}
\int_{\R^3}Q_{ij}(f_i,f_j)(v)\psi_i(v)dv=\int\int\int_{\R^3\times\R^3\times \S^2}f_i(v)f_j(v_*) \big(\psi_i(v')-\psi_i(v)\big)B_{ij}(|u|,\widehat{u}\cdot\sigma)d\sigma dv_*dv\,,
\end{equation*}
and
\begin{equation*}
\int_{\R^3}Q_{ji}(f_j,f_i)(v)\psi_j(v)dv=\int\int\int_{\R^3\times\R^3\times \S^2}f_j(v)f_i(v_*) \big(\psi_j(v')-\psi_j(v)\big)B_{ji}(|u|,\widehat{u}\cdot\sigma)d\sigma dv_*dv.
\end{equation*}
Thus, splitting the sum in the sets $\{i\geq j\}$ and $\{i<j\}$ we can interchange indexes $i\leftrightarrow j$ to obtain that
\begin{align*}
2\sum^{I}_{i=1}\int_{\R^3}\big[\Q(\F,\F)\big]_i \psi_i(v)\, dv &= 2\sum_{i=1}^I\sum_{j=1}^I \int_{\R^3}Q_{ij}(f_i,f_j)(v)\psi_i(v)dv \\
&= \sum_{i=1}^I\sum_{j=1}^I \int_{\R^3}\Big(Q_{ij}(f_i,f_j)(v)\psi_i(v) + Q_{ji}(f_{j},f_{i})\psi_{j}(v)\Big)dv\,.
\end{align*}
Recalling that $B_{ji}=B_{ij}$ and the micro reversibility condition \eqref{B0}, we can interchange variables $(v,v_*,\sigma)\rightarrow(v_*,v,-\sigma)$ in the term associated to $Q_{ji}(f_j,f_i)$ (it holds that $v'_{ji} = v'_{*,ij}$) to obtain that
\begin{equation}\label{Q4}
\begin{split}
2\sum^{I}_{i=1}\int_{\R^3}\big[\Q(\F,\F)\big]_i \psi_i(v)\,& dv = \sum_{i=1}^I\sum_{j=1}^I\int\int\int_{\R^3\times\R^3\times \S^2}f_i(v)f_j(v_*)\\
&\times \big(\psi_i(v')+\psi_j(v'_*)-\psi_i(v)-\psi_j(v_*)\big)B_{ij}(|u|,\widehat{u}\cdot\sigma)d\sigma dv_*dv\,.
\end{split}
\end{equation}
The weak form of Boltzmann collision operator imply the conservation properties of the system. In particular, $$\int_{\R^3}Q_{ij}(f_i, f_j)(v)dv=0\,,\qquad 1\leq i,\,j\leq I\,,$$
which leads to conservation of mass for each single component and for the system as a whole
$$\int_{\R^3}\big[\Q(\F,\F)\big]_i dv = 0 \qquad\text{and}\qquad \sum_{i=1}^I\int_{\R^3}\big[\Q(\F,\F)\big]_i dv=0\,.$$
Moreover, when we choose the test functions $\psi_i=m_i\, v$ and $\psi_i=m_i\, |v|^2$ in the weak form \eqref{Q4} we obtain the conservation of momentum and energy for the system as a whole (but not for each single species)
 $$\sum_{i=1}^I \int_{\R^3}\big[\Q(\F,\F)\big]_i \,m_i \,v\,dv=0\qquad\text{and}\qquad\sum_{i=1}^I \int_{\R^3}\big[\Q(\F,\F)\big]_i\,m_i\, |v|^2\,dv=0\,.$$
More precisely, if $\F$ is a solution of Boltzmann gas mixture system then 
$$ \int_{\R^3}f_i(v)dv= \int_{\R^3}f_{0,i}(v)dv\,,\qquad 1\leq i \leq I\,,\qquad \text{conservation of mass per species}\,,$$
and
$$\sum_{i=1}^I\int_{\R^3}m_i \,v \,f_i(v)dv=\sum_{i=1}^I\int_{\R^3}m_i \,v\,f_{0,i}(v)dv\,,\qquad\text{conservation of total momentum}\,,$$
and
$$\sum_{i=1}^I\int_{\R^3}m_i |v|^2f_i(v)dv=\sum_{i=1}^I\int_{\R^3}m_i |v|^2f_{0,i}(v)dv\,,\qquad \text{conservation of total energy}\,.$$
This latter yields
$$\int_{\R^3}m_i |v|^2f_i(v)dv\leq \sum_{i=1}^I\int_{\R^3}m_i |v|^2f_{0,i}(v)dv\,, \qquad 1\leq i \leq I\,.$$
Also, we recall the gas mixture entropy
$$\mathcal{H}[\F]:=\sum_{i=1}^I\int_{\R^3}f_i(v)\log(f_i)(v)dv\,,$$
and the gas mixture entropy production, given by 
$$D[\F]:=\sum_{i=1}^I \int_{\R^3}\big[\Q(\F,\F)\big]_i(v) \log(f_i)(v)dv.$$
It is possible to deduce from the weak form \eqref{Q4} that $D[\F]\leq 0$, refer to \cite{DMS}, which yields that 
\begin{equation}\label{log}
\mathcal{H}[\F](t)\leq \mathcal{H}[\F](0)=\mathcal{H}_0, \qquad \forall\, t>0.
\end{equation}
Consequently,
\begin{align*}
\| \F \|_{L\log L} &= \sum_{i=1}^I\int_{\R^3}f_{i}(v)\log(1 +  f_{i}(v))dv=\sum_{i=1}^I\int_{\{f_i\leq 1\}} f_{i}(v)\log(1+f_{i}(v))dv\\
&\hspace{6cm} +\sum_{i=1}^I\int_{\{f_i\geq 1\}} f_{i}(v)\log(1+f_{i}(v))dv\\
&\leq \log(2)\sum^{I}_{i=1}\int_{\R^3} f_{i}(v)dv  + \sum^{I}_{i=1}\int_{\R^3} f_{i}(v)\log(f_i(v))dv - \sum^{I}_{i=1}\int_{\{f_i\leq1\}} f_{i}(v)\log(f_i(v))dv\,.
\end{align*}
Note that
\begin{equation*}
-\sum^{I}_{i=1}\int_{\{f_i\leq1\}} f_{i}(v)\log(f_i(v))dv \leq \tfrac32\sum^{I}_{i=1}\int_{\R^3} |f_{i}(v)|^{\frac34}dv \leq C\bigg(\sum^{I}_{i=1}\int_{\R^3} f_{i}(v)\langle v \rangle^\frac43_idv\bigg)^{\frac34}
\end{equation*}
for $C=\frac32\big(\sum^{I}_{i=1}\int \langle v \rangle^{-4}_{i}\big)^\frac14$.  Therefore, by conservation of total mass, energy, and dissipation of entropy, it holds that
\begin{align}
\begin{split}\label{LlogL}
\| \F \|_{L\log L}\leq \log(2)\sum^{I}_{i=1}&\int_{\R^3} f_{0,i}(v)dv \\
&+ \sum^{I}_{i=1}\int_{\R^3} f_{0,i}(v)\log(f_{0,i}(v))dv +C\bigg(\sum^{I}_{i=1}\int_{\R^3} f_{0,i}\langle v \rangle^{2}_{i}\,dv\bigg)^\frac34\,.
\end{split}
\end{align}
Thus, we consider initial data satisfying that
\begin{equation}
\begin{split}\label{Cond1}
{\bf m}_0 &= \sum_{i=1}^I\int_{\R^3} f_{0,i}(v) dv<\infty\,,\qquad {\bf m}_{1} = \sum_{i=1}^I\int_{\R^3} m_i \,v \, f_{0,i}(v) dv= 0\,,\\
{\bf m}_{2} &= \sum_{i=1}^I\int_{\R^3}f_{0,i}(v) \langle v \rangle^{2}_{i}dv<\infty\,,\qquad \mathcal{H}_0=\sum_{i=1}^I\int_{\R^3} f_{0,i}(v)\,\log(f_{0,i})(v)dv<\infty\,,
\end{split}
\end{equation}
so that $\| \F \|_{L\log L} \leq \log(2){\bf m}_0 + \mathcal{H}_0 + C\,{\bf m}^{\frac34}_{2}<\infty$ as argued in \eqref{LlogL}.  In this way, solutions to the gas mixture will formally lie in the space 
$$U(D_0, E_0)=\big\{ \G\in (L^1_2\cap L\log L)\,,\, \G\geq 0 \, \big| \, \|g_i\|_{L^1}\geq D_0,\; \|\G\|_{L^1_2}+\|\G\|_{L\log L}\leq E_0 \big\},$$
where $D_0,\,E_0>0$. In particular, we take initial data $\F_0\in U(D_0,E_0)$.
\subsection{Main Results}
Recall that we address \textit{a priori} estimates related to statistical moments and higher integrability for a system of spatially homogeneous Boltzmann equations describing a gas mixture of monatomic components $A_1, . . . , A_I$, where the component $A_{i}$ has associated particle mass $m_{i}>0$
\begin{equation}
\label{B-E1}
\begin{cases}
\partial_t \F(t,v)=\Q(\F,\F)(t,v),\qquad t>0,\;\;v\in\R^3,\\
\F(0,v)=\F_0(v).
\end{cases}
\end{equation}
The $i^{th}$ entry of the vector $\F:=\F(t,v) = [f_i(t,v)]_{1\leq i\leq I}$ is the velocity distribution function $f_i:=f_{i}(t,v)\geq0$ associated to each component $A_i$ of the mixture.  Recall that
$$\Q(\F,\F)=\Big[\sum_{j=1}^I Q_{ij}\big(f_i,f_j\big)\Big]_{1\leq i\leq I}$$ 
was defined in \eqref{ColOp1}.  Also recall in the following that $\lambda_{ij}\in (0,2]$ and $s_{ij}\in (0,2)$ with $1\leq i, j \leq I$.  Keep in mind that the following are \textit{a priori} estimates, so we consider solutions $\F(t)$ to \eqref{E1} with sufficient regularity so that all computations performed are allowed.  Finally, in the results is explicit that constants depend of the parameters of the model, that is, on the $\lambda_{ij}$, $s_{ij}$, $\kappa^1_{ij}$, $\kappa^2_{ij}$, $m_{i}$, and $\| \theta^2 b_{ij}\|_{L^{1}(\mathcal{S}^2)}$.   Such constant will depend on the mass, enegy and entropy $D_0$ and $E_0$ of the gas mixture initial configuration as well.
\begin{prop}\label{Momlem}
Take $\lambda^\natural>0$ and $\F_0\in U(D_0,E_0)$.  For any integer $n\geq 2,$ the statistical moments satisfy the equation
$${\bf m}'_{2n}[\F(t)]=\sum^{I}_{i=1}\sum^{I}_{j=1} \partial_t{\bf m}^{ij}_{2n}(t)\,,\qquad t>0\,,$$
where, for each $1\leq i,j\leq I$,
\begin{align*}
\begin{split}
\partial_t{\bf m}^{ij}_{2n}:=\frac{1}{2}\int_{\mathbb{R}^3}\int_{\mathbb{R}^3}\int_0^\pi\int_0^{2\pi}\Big(\langle v'\rangle_i^{2n}+ \langle v_*'\rangle_j^{2n} - \langle v&\rangle_i^{2n} - \langle v_*\rangle_j^{2n}\Big)d\varphi\, \beta_{ij}(\theta)\,d\theta\\
&\times\,|v-v_*|^{\lambda_{ij}}\,f_i(v)\,f_j(v_*)\,dv_* \,dv\,.
\end{split}
\end{align*}
Furthermore, the following estimates hold for all $t>0$.

\medskip
\begin{itemize}
\item[(i)] For any fixed $1\leq i,j\leq I$ there exist positive constants $C^1_{ij}$ and $C^2_{ij}$ depending on the parameters of the model, $D_0$ and $E_0$, such that the pair $(f_i(t),f_j(t))$ satisfies
$$\partial_t{\bf m}^{ij}_{2n}(t)\leq -C^1_{ij}\,n^{s_{ij}/2}\big({\bf m}_{2n+\lambda_{ij},i}(t)+{\bf m}_{2n+\lambda_{ij},j}(t)\big)+C_{ij}^2\,S^{ij}_{n}(t)\,,$$
where
\begin{align*}
\mathcal{S}^{ij}_{n}(t) := \sum^{\lfloor n/2 \rfloor}_{a=1}\Big(\begin{matrix}n\\a\end{matrix}\Big)\frac{n^{s_{ij}/2}}{a^{s_{ij}/2+1}}\Big({\bf m}_{2a,j}(t){\bf m}_{2(n-a)+\lambda_{ij},i}(t) +{\bf m}_{2a,i}(t){\bf m}_{2(n-a)+\lambda_{ij},j}\Big)\,.
\end{align*} 
\item[(ii)] In addition, the system $\F(t)$ satisfies the inequality
$${\bf m}'_{2n}[\F(t)]\leq - C_1\;n^{\frac{\bar s}{2}}\;  {\bf m}_{2n+{\lambda}^\natural}[\F(t)]+\Big(C_2\,2^{\frac n2}\,n^{\frac{\overline{\overline{s}}}{2}-\frac{\theta\overline{s}}{2}}\Big)^{\frac{1}{1-\theta}}\,,\qquad \theta=\frac{2n+\overline{\overline{\lambda}}-4}{2n+\lambda^\natural-2}<1.
$$
\end{itemize}
The positive constants $C_1$ and $C_2$ have the aforementioned dependence. 
\end{prop}
Based on Proposition \ref{Momlem} item (ii) one can prove the following results about higher algebraic/exponential tail integrability.
\begin{theorem}[Polynomial Moments]\label{Mom}
Take $\lambda^\natural>0$ and $\F_0\in U(D_0,E_0)$.  Then, for any $r>2$  there exist positive constants $C_1$ and $C_2$ depending on $r$, the parameters of the model, $D_0$ and $E_0$, such that
$${\bf m}_r[\F](t)\leq C_1\Big(1+t^{-\frac{r-2}{{\lambda}^\natural}}\Big)\,,\qquad t>0\,.$$
Furthermore, in the particular case of even integers $r=2n\geq 4$ if ${\bf m}_{2n}(0)<\infty$ then
$$\sup_{t\geq 0}{\bf m}_{2n}[\F](t)\leq \max \big\{{\bf m}_{2n}(0), C_2 \big\}\,.$$
\end{theorem}
Using the two aforementioned results, we deduce the following theorem.
\begin{theorem}[Exponential Moments]\label{theo1} 
Let  $\lambda^\natural>0$, $s^\natural>0$, and $\F_0\in U(D_0,E_0)$.
\begin{itemize}
\item[(i)] Let $\rho=\min\big\{\frac{2\lambda^\natural}{2-s^\natural}, 2\big\}$.  There exist constants $T>0$ and $\sigma>0$ depending only on the parameters of the model, $D_0$ and $E_0$, such that
$$\sup_{[0,T]}\sum_{i=1}^I\int_{\R^3}\exp\Big[\frac{\sigma t^{\rho/\lambda^\natural}\langle v\rangle_i^{\rho}}{2}\Big]f_i(v)dv\leq 4\, {\bf m}_0\qquad \text{(Generation of Moments)}\,.$$
\item[(ii)] For any $A>0$, $\sigma_0>0$ and $\rho\in(0,2]$, there exists a (computable) constant $\sigma>0$ depending on the parameters of the model, $D_0$, $E_0$, $\rho$, $\sigma_0$, and $A$, such that if
$$\sum_{i=1}^I\int_{\R^3}\exp\big[ \sigma_0 \langle v\rangle_i^\rho \big] \, f_i(0,v)dv\leq A,$$
then
$$\sup_{t\geq 0}\sum_{i=1}^I\int_{\R^3}\exp[\sigma \langle v\rangle_i^{\rho}]f_i(t,v)dv\leq 6\,I\,e\,\big({\bf m}_{0}+1\big)\qquad \text{(Propagation of Moments)}\,.$$
\end{itemize}

\end{theorem}
Once statistical moments are understood for the system, it is possible to obtain higher integrability generation and propagation as stated in the following results.  We remark that the results now refer to individualised components.  

\begin{theorem}[Higher integrability]\label{theo2}
Fix $1\leq i \leq I$, and let $\lambda^\natural>0$ and $\bar{\bar{s}}_{i}>0$.   Assume also that  $\F_0\in U(D_0,E_0)$. Then, it holds for any $p\in(1,\infty)$ and $t_0>0$ that
\begin{equation}\label{theo2-e1}
\|f_i(t)\|_{L^p}\leq C_{t_0}, \qquad t\geq t_0\,.
\end{equation}
The constant $C_{t_0}\lesssim t^{-\beta}_0+1$ (for some $\beta>0$) depends only on $t_0$, the parameters of the model, $D_0$ and $E_0$.  Furthermore, if $f_{i,0}\in L^1_{2n} \cap L^p$ with sufficiently large number of moments $n\geq2$, then
\begin{equation}\label{theo2-e2}
\sup_{t\geq  0}\,\|f_i(t)\|_{L^p}\leq \max\Big\{\|f_{i,0}\|_{L^p}\,,\, C\Big\}\,.
\end{equation}
The positive constant $C$ depends, in addition, on the $L^{1}_{2n}$ norm of $f_{i,0}$.  The number of moments can be taken as $2n\geq r_s:=\frac{3\bar{\bar{\lambda}}_i}{\bar{\bar s}_i}$.
\end{theorem}
In light of Theorem \ref{theo2}, one has the following result as a consequence.
\begin{cor}[Baseline regularity]\label{CorLp}
Fix $1\leq i \leq I$.  Assume $\bar{\lambda}_i$, $\bar{\bar{s}}_{i}>0$ and $\F_0\in U(D_0,E_0)$.  Then, 
$$
\int^{t}_{t_0}\|\langle v\rangle_i^{\frac{\overline{\lambda}_i}{2}}f_i(\tau)\|_{H^{\frac{\overline{\overline{s}}_{i}}{2}}}^2\,d\tau \leq C_{t_0}(1+t)\,,\qquad  t\geq t_0>0\,,
$$
where the constant $C_{t_0}\lesssim t^{-\beta}_0+1$ (for some $\beta>0$) depends only on $t_0$, the parameters of the model, $D_0$ and $E_0$.

\smallskip
\noindent
Furthermore, if $f_{i,0}\in L^1_{2n} \cap L^2$ with sufficiently large number of moments $2n\geq \frac{3\bar{\bar{\lambda}}_i}{\bar{\bar s}_i}$, then
$$
\int^{t}_{0}\|\langle v\rangle_i^{\frac{\overline{\lambda}_i}{2}}f_i(\tau)\|_{H^{\frac{\overline{\overline{s}}_{i}}{2}}}^2\,d\tau \leq C\big(\| f_{i,0} \|_{L^{2}} + t\big)\,,\qquad  t\geq0\,,
$$
where $C>0$ depends, in addition, on the $L^{1}_{2n}$ norm of $f_{i,0}$.
\end{cor}
The case $p=+\infty$ is treated in our last result.
\begin{theorem}[Boundedness]\label{theo3}
Fix $1\leq i \leq I$, and let $\lambda^\natural>0$ and $\bar{\bar{s}}_{i}>0$.  Assume also that $\F_0\in U(D_0,E_0)$. Then, given $t_0>0$ there exists a constant $C_{t_0}>0$ depending on the parameters of the model, $D_0$, $E_0$, and $t_0$, such that
$$
\|f_i(t)\|_{L^\infty}\leq C_{t_0},\qquad t\geq t_0\,.
$$
The constant is controlled as $C_{t_0}\lesssim(t^{-\beta}_0+1)$ (for some $\beta>0$).  Furthermore, if additionally $f_{i,0}\in L^{1}_{2n}\cap L^\infty$ with sufficiently large number of moments $2n\geq \frac{18\bar{\bar{\lambda}}_i}{\bar{\bar s}_i}$, then
$$
\sup_{t\geq 0}\|f_i(t)\|_{L^\infty} \leq \max\Big\{2\,\|f_{i,0}\|_{L^\infty}, C\Big\}\,,
$$
for a positive constant $C$ depending additionally on the $L^{1}_{2n}$ norm of $f_{i,0}$.
\end{theorem}
Let us remark that the fact that we can infer properties of higher integrability/regularity emergence and propagation for individualised components in Theorem \ref{theo2}, Corollary \ref{CorLp}, and Theorem \ref{theo3} is a consequence of the forward scattering hypothesis in the $b^{ij}$.   As previously mentioned, including the backscattering in the interactions adds lower order terms that can be managed at a technical price.  We refer to \cite{CGP} for a detailed analysis on the handling of ``cutoff" terms.

\section{Generation and propagation of statistical moments}
In this section, the proof are inspired from classical approaches, in particular we extend the Povzner lemma given in \cite{NF} to this setting of system of equations.
\subsection{Povzner Lemma for long range interactions}
\begin{lem}\label{Povlem}
Let the cross section $\B_{ij}$ satisfied \eqref{B1} and \eqref{B2}. Let $v'$ and $v'_*$ the collision laws defined in \eqref{collision-laws} with $m_i, m_j>0$.   Then, for all integers $n\geq 2$ the following estimate holds \begin{align*}
\int_0^\pi\int_0^{2\pi}\Big(\langle v'\rangle_i^{2n} + &\langle v_*'\rangle_j^{2n}-\langle v\rangle_i^{2n}- \langle v_*\rangle_j^{2n}\Big)d\varphi \, \beta_{ij}(\theta) \, d\theta\leq -\lambda^1_{ij}n^{s_{ij}/2}\Big(\langle v\rangle_i^{2n}+\langle v_*\rangle_j^{2n}\Big)\\
&+ \lambda^2_{ij}\sum^{n-1}_{a=1}\Big(\begin{matrix}n\\a\end{matrix}\Big)\Big[\frac{n^{s_{ij}/2}}{(n-a)^{s_{ij}/2+1}}+\frac{1}{a}\Big] \times \Big(\langle v\rangle_i^{2a}\langle v_*\rangle_j^{2(n-a)}+\langle v\rangle_i^{2(n-a)}\langle v_*\rangle_j^{2a}\Big),
\end{align*}
where 
$$\lambda^1_{ij}=\frac{\kappa^1_{ij}}{2-s_{ij}}\frac{m_im_j(m_i^2+m_j^2+m_im_j)}{(m_i+m_j)^4}\,,\qquad \lambda^2_{ij}=\zeta \frac{(m_i+m_j)^2}{m_im_j}.$$
\end{lem}
\begin{proof}
In this proof we need to parametrize the collision laws \eqref{collision-laws}, see for example \cite{NF} or \cite{FM}. For every $X\in \mathbb{R}^3\setminus\{0\}$, we introduce $I(X),\; J(X)\in \mathbb{R}^3$ such that $\big(\frac{X}{|X|}, \frac{I(X)}{|X|}, \frac{J(X)}{|X|}\big)$ is an orthonormal basis of $\mathbb{R}^3.$ We also put $I(0)=J(0)=0$.  For $X, v, v_*\in \mathbb{R}^3,\; \theta \in [0, 2\pi)$ and $\varphi\in [0, 2\pi)$, we set
$$
\begin{cases}
\Gamma (X,\varphi)= \cos\varphi I(X)+ \sin \varphi J(X)\,,\\
\sigma= \cos\theta\cdot\frac{(v-v_*)}{|v-v_*|}+\sin\theta \frac{\Gamma(v-v_*,\varphi)}{|v-v_*|},\\
v'= v- \frac{m_j}{m_i+m_j}(1-\cos \theta)(v-v_*) + \frac{m_j}{m_i+m_j}\sin \theta\Gamma(v-v_*, \varphi),\\
v'_*= v_*+ \frac{m_i}{m_i+m_j}(1-\cos \theta)(v-v_*) - \frac{m_i}{m_i+m_j}\sin \theta\Gamma(v-v_*, \varphi)\,,
\end{cases}
$$
where $|\Gamma(v-v_*,\varphi)|=|v-v_*|$ and $(v-v_*)\cdot\Gamma(v-v_*, \varphi)=0$.  Then, we find that 
\begin{align*}
|v'|^2&= |v|^2+ \Big(\frac{m_j}{m_i+m_j}\Big)^2(1-\cos\theta)^2|v-v_*|^2+\Big(\frac{m_j}{m_i+m_j}\Big)^2\sin^2\theta |v-v_*|^2\\
&\qquad-2\frac{m_j}{m_i+m_j}(1-\cos\theta)v\cdot(v-v_*)+2\frac{m_j}{m_i+m_j}\sin\theta\, v\cdot\Gamma(v-v_*,\varphi)\\
&= |v|^2\Big[1+2\Big(\frac{m_j}{m_i+m_j}\Big)^2(1-\cos\theta)-2\frac{m_j}{m_i+m_j}(1-\cos\theta)\Big]+|v_*|^2\Big[2\Big(\frac{m_j}{m_i+m_j}\Big)^2(1-\cos\theta)\Big]\\
&\qquad+v\cdot v_*\Big[-4\Big(\frac{m_j}{m_i+m_j}\Big)^2(1-\cos\theta)+2\frac{m_j}{m_i+m_j}(1-\cos\theta)\Big]+2\frac{m_j}{m_i+m_j}\sin\theta v\cdot \Gamma(v-v_*,\varphi)\\
&=\Big[\frac{m_i^2+m_j^2+2m_im_j\cos\theta}{(m_i+m_j)^2}\Big] |v|^2 +\Big[2\Big(\frac{m_j}{m_i+m_j}\Big)^2(1-\cos\theta)\Big]|v_*|^2\\
&\qquad+\Big[\frac{2m_im_j-2m_j^2}{(m_i+m_j)^2}(1-\cos\theta)\Big]v\cdot v_*+2\frac{m_j}{m_i+m_j}\sin\theta\, v\cdot \Gamma(v-v_*,\varphi).
\end{align*}
Multiplying the above equality by $\frac{m_i}{\sum_{l'=1}^Im_{l'}}$ we get that
\begin{align*}
\frac{m_i}{\sum m_{l'}}|v'|^2&=\Big[\frac{m_i^2+m_j^2+2m_im_j\cos\theta}{(m_i+m_j)^2}\Big]\frac{m_i}{\sum m_{l'}} |v|^2 +\frac{m_i}{\sum m_{l'}}\Big[2\Big(\frac{m_j}{m_i+m_j}\Big)^2(1-\cos\theta)\Big]|v_*|^2\\
&+\frac{m_i}{\sum m_{l'}}\Big[\frac{2m_im_j-2m_j^2}{(m_i+m_j)^2}(1-\cos\theta)\Big]v\cdot v_*+2\frac{m_i}{\sum m_{l'}}\frac{m_j}{m_i+m_j}\sin\theta v\cdot \Gamma(v-v_*,\varphi)\\
&=\Big[\frac{m_i^2+m_j^2+2m_im_j\cos\theta}{(m_i+m_j)^2}\Big]\frac{m_i}{\sum m_{l'}} |v|^2 +\Big[2\frac{m_im_j}{(m_i+m_j)^2}(1-\cos\theta)\Big]\frac{m_j}{\sum m_{l'}}|v_*|^2\\
&+\Big[\frac{2m_i^2m_j-2m_im_j^2}{\sum m_{l'}(m_i+m_j)^2}(1-\cos\theta)\Big]v\cdot v_*+2\frac{m_im_j}{\sum m_{l'}(m_i+m_j)}\sin\theta v\cdot \Gamma(v-v_*,\varphi),
\end{align*}
which yields,
\begin{align*}
\langle v'\rangle_i^2&= \Big[\frac{m_i^2+m_j^2+2m_im_j\cos\theta}{(m_i+m_j)^2}\Big]\langle v\rangle_i^2 +\Big[2\frac{m_im_j}{(m_i+m_j)^2}(1-\cos\theta)\Big]\langle v_*\rangle_j^2\\
&+\Big[\frac{2m_i^2m_j-2m_im_j^2}{\sum m_{l'}(m_i+m_j)^2}(1-\cos\theta)\Big]v\cdot v_*+2\frac{m_im_j}{\sum m_{l'}(m_i+m_j)}\sin\theta v\cdot \Gamma(v-v_*,\varphi)\\
&+\underset{=0}{\underbrace{1-\frac{m_i^2+m_j^2+2m_im_j\cos\theta}{(m_i+m_j)^2}-2\frac{m_im_j}{(m_i+m_j)^2}(1-\cos\theta)}}\\
&= \Big[\frac{m_i^2+m_j^2+2m_im_j\cos\theta}{(m_i+m_j)^2}\Big]\langle v\rangle_i^2 +\Big[2\frac{m_im_j}{(m_i+m_j)^2}(1-\cos\theta)\Big]\langle v_*\rangle_j^2\\
&+\Big[\frac{2m_i^2m_j-2m_im_j^2}{\sum m_{l'}(m_i+m_j)^2}(1-\cos\theta)\Big]v\cdot v_*+2\frac{m_im_j}{\sum m_{l'}(m_i+m_j)}\sin\theta v\cdot \Gamma(v-v_*,\varphi)\,.
\end{align*}
Fixing $n\geq 2$ and letting $B_n=\{(p,q,k,l)\in\N^4: p+q+k+l=n\}$, we obtain using Newton's expansion that
\begin{equation}\label{Vn1}
\begin{split}
\langle v'\rangle_i^{2n}&= \sum_{p,q,k,l}\frac{n!}{p!q!k!l!}\Big[\frac{m_i^2+m_j^2+2m_im_j\cos\theta}{(m_i+m_j)^2}\Big]^p\\
&\qquad\times\Big[2\frac{m_im_j}{(m_i+m_j)^2}(1-\cos\theta)\Big]^q\Big[\frac{2m_i^2m_j-2m_im_j^2}{\sum m_{l'}(m_i+m_j)^2}(1-\cos\theta)\Big]^k\\
&\qquad\quad\times\Big[2\frac{m_im_j}{\sum m_{l'}(m_i+m_j)}\sin\theta\Big]^l\langle v\rangle_i^{2p}\langle v_*\rangle_j^{2q}(v\cdot v_*)^k\big[v\cdot \Gamma(v-v_*,\varphi)\big]^l.
\end{split}
\end{equation}
Also, we have that
\begin{align*}
\langle v_*'\rangle_j^2&=\Big[\frac{m_i^2+m_j^2+2m_im_j\cos\theta}{(m_i+m_j)^2}\Big]\langle v_*\rangle_j^2 +\Big[2\frac{m_im_j}{(m_i+m_j)^2}(1-\cos\theta)\Big]\langle v\rangle_i^2\\
&\quad -\Big[\frac{2m_i^2m_j-2m_im_j^2}{\sum m_{l'}(m_i+m_j)^2}(1-\cos\theta)\Big]v\cdot v_*-2\frac{m_im_j}{\sum m_{l'}(m_i+m_j)}\sin\theta v\cdot \Gamma(v-v_*,\varphi),
\end{align*}
which implies that
\begin{equation}\label{Vn2}
\begin{split}
\langle v'_*\rangle_j^{2n}&= \sum_{p,q,k,l}\frac{n!}{p!q!k!l!}\Big[2\frac{m_im_j}{(m_i+m_j)^2}(1-\cos\theta)\Big]^p\\
&\qquad\times\Big[\frac{m_i^2+m_j^2+2m_im_j\cos\theta}{(m_i+m_j)^2}\Big]^q\Big[-\frac{2m_i^2m_j-2m_im_j^2}{\sum m_{l'}(m_i+m_j)^2}(1-\cos\theta)\Big]^k\\
&\qquad\quad\times\Big[-2\frac{m_im_j}{\sum m_{l'}(m_i+m_j)}\sin\theta\Big]^l\langle v\rangle_i^{2q}\langle v_*\rangle_j^{2p}(v\cdot v_*)^k\big[v\cdot \Gamma(v-v_*,\varphi)\big]^l\,.
\end{split}.
\end{equation}
Then, combining \eqref{Vn1}) and \eqref{Vn2} we get the following equality 
$$\frac{1}{2\pi}\int_0^\pi\int_0^{2\pi} \big(\langle v'\rangle_i^{2n}+\langle v'_*\rangle_j^{2n}-\langle v\rangle_i^{2n}-\langle v_*\rangle_j^{2n}\big)d\varphi \, \beta_{ij}(\theta) \,d\theta= -a_n+b_n$$
where
$$a_n:=\int_0^\pi\Big(1-\Big[\frac{m_i^2+m_j^2+2m_im_j\cos\theta}{(m_i+m_j)^2}\Big]^n-\Big[2\frac{m_im_j}{(m_i+m_j)^2}(1-\cos\theta)\Big]^n\Big)\Big(\langle v\rangle_i^{2n}+\langle v_*\rangle_j^{2n}\Big)\beta_{ij}(\theta) d\theta\,,$$
and
$$
b_n:=\int_0^\pi\int^{2\pi}_0 \Big(\Theta(v,v_*,\theta)+ \Theta(v_*,v,-\theta) \Big)d\varphi\, \beta_{ij}(\theta)\, d\theta,
$$
with
\begin{align*}
\int^{2\pi}_0 \Theta(v,v_*,\theta)d\varphi&= \frac{1}{2\pi}\sum_{(p,q,k,l)\in A_n}\frac{n!}{p!q!k!l!}\Big[\frac{m_i^2+m_j^2+2m_im_j\cos\theta}{(m_i+m_j)^2}\Big]^p\Big[2\frac{m_im_j}{(m_i+m_j)^2}(1-\cos\theta)\Big]^q\\
&\qquad\times\Big[\frac{2m_i^2m_j-2m_im_j^2}{\sum m_{l'}(m_i+m_j)^2}(1-\cos\theta)\Big]^k\Big[2\frac{m_im_j}{\sum m_{l'}(m_i+m_j)}\sin\theta\Big]^l\\
&\qquad\quad\times\langle v\rangle_i^{2p}\langle v_*\rangle_j^{2q}(v\cdot v_*)^k\int^{2\pi}_0\big[v\cdot \Gamma(v-v_*,\varphi)\big]^ld\varphi\,,
\end{align*} 
such that $A_n=\big\{(p,q,k,l)\in \mathbb{N}^4; p+q+k+l=n/\{(n,0,0,0),(0,n,0,0)\}\big\}$.\\

We proceed in two steps.\\

\noindent
{\bf Step 1.} We show that there exists a constant $\lambda^1_{ij}>0$ such that
 $$a_n\geq \lambda^1_{ij}\,n^{\frac{s_{ij}}{2}}\Big(\langle v\rangle_i^{2n}+\langle v_*\rangle_j^{2n}\Big).$$
We remark that the integral in $a_n$ is nonnegative because for $X=\frac{m_i^2+m_j^2+2m_im_j\cos\theta}{(m_i+m_j)^2}\in [0,1]$ we have that $$X^n+(1-X)^n\leq 1\,.$$
Then, using \eqref{B2} we get that
$$a_n\geq \kappa^1_{ij} \int_0^{n^{-\frac{1}{2}}}\bigg(1-\Big[\frac{m_i^2+m_j^2+2m_im_j\cos\theta}{(m_i+m_j)^2}\Big]^n-\Big[2\frac{m_im_j}{(m_i+m_j)^2}(1-\cos\theta)\Big]^n\Big)\theta^{-s_{ij}-1}d\theta\Big(\langle v\rangle_i^{2n}+\langle v_*\rangle_j^{2n}\bigg). $$
Furthermore, for all $\theta\in [0,n^{-\frac{1}{2}}]$ we have that
$$\cos\theta\leq 1-\frac{\theta^2}{2}\,,$$
which yields
$$\frac{m_i^2+m_j^2+2m_im_j\cos\theta}{(m_i+m_j)^2}\leq 1-\frac{m_im_j}{(m_i+m_j)^2}\theta^2\,,$$
whence 
$$\Big[\frac{m_i^2+m_j^2+2m_im_j\cos\theta}{(m_i+m_j)^2}\Big]^n\leq \Big[1-\frac{m_im_j}{(m_i+m_j)^2}\theta^2\Big]^n\leq e^{-n\frac{m_im_j}{(m_i+m_j)^2}\theta^2}.$$
Then
$$1-\Big[\frac{m_i^2+m_j^2+2m_im_j\cos\theta}{(m_i+m_j)^2}\Big]^n\geq n\frac{m_im_j}{(m_i+m_j)^2}\theta^2.$$
Next, still for $\theta\in[0, n^{-\frac{1}{2}}]$ we have that
$$1-\cos\theta \leq \frac{\theta^2}{2}\,,$$
which implies that
$$\Big[2\frac{m_im_j}{(m_i+m_j)^2}(1-\cos\theta)\Big]^n\leq \Big[\frac{m_im_j}{(m_i+m_j)^2}\theta^2\Big]^n,$$
whence, recalling that $n\geq 2$,
\begin{align*}
1-&\Big[\frac{m_i^2+m_j^2+2m_im_j\cos\theta}{(m_i+m_j)^2}\Big]^n-\Big[2\frac{m_im_j}{(m_i+m_j)^2}(1-\cos\theta)\Big]^n \\
&\quad\geq n\frac{m_im_j}{(m_i+m_j)^2}\theta^2-\Big[\frac{m_im_j}{(m_i+m_j)^2}\theta^2\Big]^n\geq  n\frac{m_im_j}{(m_i+m_j)^2}\theta^2-\Big(\frac{m_im_j}{(m_i+m_j)^2}\Big)^2\theta^2\\
&\quad\geq \frac{m_im_j}{(m_i+m_j)^2}\Big[\frac{n}{2}-\frac{n m_im_j}{2(m_i+m_j)^2}\Big]\theta^2\geq \frac{m_im_j(m_i^2+m_j^2+m_im_j)}{2(m_i+m_j)^4}n\theta^2\geq c\, n\, \theta^2.
\end{align*}
Consequently, 
\begin{align*}
a_n&\geq \kappa^1_{ij}\,c\, n\int^{n^{-\frac{1}{2}}}_0 \theta^{1-s_{ij}}d\theta\Big(\langle v\rangle_i^{2n}+\langle v_*\rangle_j^{2n}\Big) \geq \lambda^1_{ij}n^{\frac{s_{ij}}{2}}\Big(\langle v\rangle_i^{2n}+\langle v_*\rangle_j^{2n}\Big)
\end{align*}
with $\lambda^1_{ij}=\frac{\kappa^1_{ij}}{2-s_{ij}}\frac{m_im_j(m_i^2+m_j^2+m_im_j)}{(m_i+m_j)^4}$.\\

\noindent
{\bf Step 2.} We estimate $b_n$. We start by recalling that 
\begin{align*}
\int^{2\pi}_0 \Theta(v,v_*,\theta)d\varphi&= \frac{1}{2\pi}\sum_{(p,q,k,l)\in A_n}\frac{n!}{p!q!k!l!}\Big[\frac{m_i^2+m_j^2+2m_im_j\cos\theta}{(m_i+m_j)^2}\Big]^p\Big[2\frac{m_im_j}{(m_i+m_j)^2}(1-\cos\theta)\Big]^q\\
&\qquad\times\Big[\frac{2m_i^2m_j-2m_im_j^2}{\sum m_{l'}(m_i+m_j)^2}(1-\cos\theta)\Big]^k\Big[2\frac{m_im_j}{\sum m_{l'}(m_i+m_j)}\sin\theta\Big]^l\\
&\qquad\quad\times\langle v\rangle_i^{2p}\langle v_*\rangle_j^{2q}(v\cdot v_*)^k\int^{2\pi}_0\big[v\cdot \Gamma(v-v_*,\varphi)\big]^ld\varphi
\end{align*} 
such that $A_n=\big\{(p,q,k,l)\in \mathbb{N}^4; p+q+k+l=n/\{(n,0,0,0),(0,n,0,0)\}\big\}$. For the last term we have that 
$$\frac{1}{2\pi}\int^{2\pi}_0\big[v\cdot \Gamma(v-v_*,\varphi)\big]^ld\varphi= {\bf 1}_{\{l\in2\mathbb{N}\}}\frac{l!}{2^l[(\frac{l}{2})!]^2}\Big(|v|^2|v_*|^2-(v\cdot v_*)^2\Big)^\frac{l}{2}.$$
Let $C_n=\{(p,q,k,l)\in A_n; l\in 2\mathbb{N}\}$, we compute that
\begin{align*}
\int^{2\pi}_0 \Theta_1(v , & \,v_*,\theta)d\varphi= \!\!\!\!\sum_{(p,q,k,l)\in C_n}\frac{n!}{p!q!k![(\frac{l}{2})!]^2}\Big[\frac{m_i^2+m_j^2+2m_im_j\cos\theta}{(m_i+m_j)^2}\Big]^p\Big[2\frac{m_im_j}{(m_i+m_j)^2}(1-\cos\theta)\Big]^q\\
&\qquad\times\Big[\frac{2m_i^2m_j-2m_im_j^2}{\sum m_{l'}(m_i+m_j)^2}(1-\cos\theta)\Big]^k
\langle v\rangle_i^{2p}\langle v_*\rangle_j^{2q}(v\cdot v_*)^k\\
&\qquad\quad\times\Big[\frac{m_im_j}{\sum m_{l'}(m_i+m_j)}\sin\theta\Big]^l\big[|v|^2|v_*|^2-(v\cdot v_*)^2\big]^\frac{l}{2}\\
&\leq \sum_{(p,q,k,l)\in C_n}\frac{n!}{p!q!k![(\frac{l}{2})!]^2}\Big[\frac{m_i^2+m_j^2+2m_im_j\cos\theta}{(m_i+m_j)^2}\Big]^p\Big[2\frac{m_im_j}{(m_i+m_j)^2}(1-\cos\theta)\Big]^q\\
&\qquad\times\langle v\rangle_i^{2p}\langle v_*\rangle_j^{2q}
\Big[\frac{2m_i^\frac{3}{2}m_j^\frac{1}{2}-2m_i^\frac{1}{2}m_j^\frac{3}{2}}{(m_i+m_j)^2}(1-\cos\theta)\Big]^k
\Big[\Big(\frac{m_i^\frac{1}{2}}{(\sum m_{l'})^\frac{1}{2}}|v|\frac{m_j^\frac{1}{2}}{(\sum m_{l'})^\frac{1}{2}}|v_*|\Big)^2\Big]^\frac{k}{2}\\
&\qquad\quad\times\Big[\frac{m^{\frac{1}{2}}_im^\frac{1}{2}_j}{(m_i+m_j)}\sin\theta\Big]^l\Big[\frac{m_i}{\sum m_{l'}}|v|^2\frac{m_j}{\sum m_{l'}}|v_*|^2-\frac{m_im_j}{\sum m_{l'}}(v\cdot v_*)^2\Big]^\frac{l}{2}\\
&\leq \sum_{(p,q,k,l)\in C_n}\frac{n!}{p!q!k![(\frac{l}{2})!]^2}\Big[\frac{m_i^2+m_j^2+2m_im_j\cos\theta}{(m_i+m_j)^2}\Big]^p\Big[2\frac{m_im_j}{(m_i+m_j)^2}(1-\cos\theta)\Big]^q\\
&\qquad\times \Big[\frac{m_i^\frac{3}{2}m_j^\frac{1}{2}-m_i^\frac{1}{2}m_j^\frac{3}{2}}{(m_i+m_j)^2}(1-\cos\theta)\Big]^k\Big[\frac{m^{\frac{1}{2}}_im^\frac{1}{2}_j}{(m_i+m_j)}\sin\theta\Big]^l \langle v\rangle_i^{2p+l+k}\langle v_*\rangle_j^{2q+l+k}.
\end{align*}
Then,
\begin{align*}
\int_0^\pi &\int_0^{2\pi}\Big(\Theta_1(v,v_*,\theta)+\Theta_2(v_*,v,-\theta)\Big)d\varphi\, \beta_{ij}(\theta)\,d\theta\\
&\leq  \sum_{(p,q,k,l)\in C_n}\frac{n!}{p!q!k![(\frac{l}{2})!]^2}\int_0^\pi\Big[\frac{m_i^2+m_j^2+2m_im_j\cos\theta}{(m_i+m_j)^2}\Big]^p\Big[2\frac{m_im_j}{(m_i+m_j)^2}(1-\cos\theta)\Big]^q\\
&\qquad\times \Big[\frac{m_i^\frac{3}{2}m_j^\frac{1}{2}-m_i^\frac{1}{2}m_j^\frac{3}{2}}{(m_i+m_j)^2}(1-\cos\theta)\Big]^k
\Big[\frac{m^{\frac{1}{2}}_im^\frac{1}{2}_j}{(m_i+m_j)}\sin\theta\Big]^l\,\beta_{ij}(\theta)\,d\theta\\
&\qquad\quad\times\Big(\langle v\rangle_i^{2p+l+k}\langle v_*\rangle_j^{2q+l+k}+\langle v\rangle_i^{2q+l+k}\langle v_*\rangle_j^{2p+l+k}\Big)
\\
&\leq  \sum_{(p,q,k,l)\in C_n}\frac{n!}{p!q!k![(\frac{l}{2})!]^2} I_{p,q,k,l}\times\Big(\langle v\rangle_i^{2p+l+k}\langle v_*\rangle_j^{2q+l+k}+\langle v\rangle_i^{2q+l+k}\langle v_*\rangle_j^{2p+l+k}\Big)
\end{align*}
where 
\begin{align*}
I_{p,q,k,l}&=\int_0^\pi\Big[\frac{m_i^2+m_j^2+2m_im_j\cos\theta}{(m_i+m_j)^2}\Big]^p\Big[2\frac{m_im_j}{(m_i+m_j)^2}(1-\cos\theta)\Big]^q \\
&\qquad\times\Big[\frac{m_i^\frac{3}{2}m_j^\frac{1}{2}-m_i^\frac{1}{2}m_j^\frac{3}{2}}{(m_i+m_j)^2}(1-\cos\theta)\Big]^k\times\Big[\frac{m^{\frac{1}{2}}_im^\frac{1}{2}_j}{(m_i+m_j)}\sin\theta\Big]^l\,\beta_{ij}(\theta)\, d\theta \,.
\end{align*}
This can be rewritten as
$$\int_0^\pi \int_0^{2\pi}\Big(\Theta_1(v,v_*,\theta)+\Theta_2(v_*,v,-\theta)\Big)d\varphi\, b_{ij}(\cos\theta)\,d\theta \leq \sum_{a=0}^n K_{n,a} \Big(\langle v\rangle_i^{2a}\langle v_*\rangle_j^{2(n-a)}+\langle v\rangle_i^{2(n-a)}\langle v_*\rangle_j^{2a}\Big), $$
where, setting $A_{n,a}=\big\{ (p,q,k,l)\in C_n; p+\frac{k}{2}+\frac{l}{2}=a\;({\rm whence\;} q+\frac{k}{2}+\frac{l}{2}=n-a )\big\}$, one has that
$$K_{n,a}= \sum_{(p,q,k,l)\in A_{n,a}}\frac{n!\;I_{p,q,k,l}}{p!q!k![(\frac{l}{2})!]^2}\,.$$
Fix $a\in [1,n-1]$, then $(p,q,k,l)\in A_{n,a}$ if and only if there exist $r,s\in \mathbb{N}$ such that $k=2r$, $l=2s$, $p=a-r-s$, and $q=n-a-r-s$.  Consequently, 
$$K_{n,a}=\sum_{r,s}\frac{n!\;I_{a-r-s,n-a-r-s,2r,2s}}{(a-r-s)!(n-a-r-s)!(2r)!s!}\,.$$
Using the fact that
$$\Big[\frac{m_i^\frac{3}{2}m_j^\frac{1}{2}-m_i^\frac{1}{2}m_j^\frac{3}{2}}{(m_i+m_j)^2}(1-\cos\theta)\Big]^{2}\leq \Big[\frac{m_i^2+m_j^2+2m_im_j\cos\theta}{(m_i+m_j)^2}\Big]\Big[2\frac{m_im_j}{(m_i+m_j)^2}(1-\cos\theta)\Big]$$
and
$$\Big[\frac{m^{\frac{1}{2}}_im^\frac{1}{2}_j}{(m_i+m_j)}\sin\theta\Big]^{2}\leq \Big[\frac{m_i^2+m_j^2+2m_im_j\cos\theta}{(m_i+m_j)^2}\Big]\Big[2\frac{m_im_j}{(m_i+m_j)^2}(1-\cos\theta)\Big]$$
we deduce that
 \begin{align*}
I_{a-r-s,n-a-r-s,2r,2s}&=\int_0^\pi\Big[\frac{m_i^2+m_j^2+2m_im_j\cos\theta}{(m_i+m_j)^2}\Big]^{a-r-s}\Big[2\frac{m_im_j}{(m_i+m_j)^2}(1-\cos\theta)\Big]^{n-a-r-s} \\
&\qquad\times\Big[\frac{m_i^\frac{3}{2}m_j^\frac{1}{2}-m_i^\frac{1}{2}m_j^\frac{3}{2}}{(m_i+m_j)^2}(1-\cos\theta)\Big]^{2r}\times\Big[\frac{m^{\frac{1}{2}}_im^\frac{1}{2}_j}{(m_i+m_j)}\sin\theta\Big]^{2s}\beta_{ij}(\theta) \,d\theta \\
&\hspace{-1cm}\leq \int^\pi_0\Big[\frac{m_i^2+m_j^2+2m_im_j\cos\theta}{(m_i+m_j)^2}\Big]^{a}\Big[2\frac{m_im_j}{(m_i+m_j)^2}(1-\cos\theta)\Big]^{n-a}\,\beta_{ij}(\theta)\,d\theta = :J_{n,a}\,.
\end{align*}
Thus,
$$K_{n,a}=J_{n,a}\sum_{r,s}\frac{n!}{(a-r-s)!(n-a-r-s)!(2r)!s!}\leq C_n\,J_{n,a}\,.$$
From the definition of $J_{n,a}$ one has that
$$J_{n,a}\leq \kappa^1_{ij} K_{s_{ij},n,a}+\kappa^2_{ij} L_{s_{ij},n,a},$$
where
$$ K_{s_{ij},n,a}=\int^{\pi/2}_0 \Big[\frac{m_i^2+m_j^2+2m_im_j\cos\theta}{(m_i+m_j)^2}\Big]^{a}\Big[2\frac{m_im_j}{(m_i+m_j)^2}(1-\cos\theta)\Big]^{n-a} \theta^{-s_{ij}-1}d\theta\,,$$
and $$L_{s_{ij},n,a}= \Big(\frac{2}{\pi}\Big)^{s_{ij}+1}\int^\pi_{\pi/2}\Big[\frac{m_i^2+m_j^2+2m_im_j\cos\theta}{(m_i+m_j)^2}\Big]^{a}\Big[2\frac{m_im_j}{(m_i+m_j)^2}(1-\cos\theta)\Big]^{n-a}d\theta.$$
Let us show that $L_{s_{ij},n,a}$ is estimated in terms of $K_{0,n,n-a}$.  Indeed, using the substitution $\theta\rightarrow \pi-\theta$ we see that
$$\frac{m_i^2+m_j^2-2m_im_j\cos\theta}{(m_i+m_j)^2}=2\frac{m_im_j}{(m_i+m_j)^2}(1-\cos\theta)\,,$$
and 
$$2\frac{m_im_j}{(m_i+m_j)^2}(1+\cos\theta)=\frac{m_i^2+m_j^2+2m_im_j\cos\theta}{(m_i+m_j)^2}\,.$$
Consequently, we deduce that
\begin{align*}
L_{s_{ij},n,a}&= \Big(\frac{2}{\pi}\Big)^{s_{ij}+1}\int^{\pi/2}_0\Big[\frac{m_i^2+m_j^2+2m_im_j\cos\theta}{(m_i+m_j)^2}\Big]^{n-a}\Big[2\frac{m_im_j}{(m_i+m_j)^2}(1-\cos\theta)\Big]^{a}d\theta\\
&\leq \Big(\frac{2}{\pi}\Big)^{s_{ij}}\int^{\pi/2}_0\Big[\frac{m_i^2+m_j^2+2m_im_j\cos\theta}{(m_i+m_j)^2}\Big]^{n-a}\Big[2\frac{m_im_j}{(m_i+m_j)^2}(1-\cos\theta)\Big]^{a}\theta^{-1}d\theta \\
&\leq\Big(\frac{2}{\pi}\Big)^{s_{ij}} K_{0,n,n-a}\,.
\end{align*} 
Now, we estimate $K_{s_{ij},n,a}$.   For $\theta\in(0,\pi/2]$ we have that 
$$\theta \leq 2\sin\theta\qquad\text{and}\qquad\theta^{-1}\leq \Big[2\frac{m_im_j}{(m_i+m_j)^2}(1-\cos\theta)\Big]^{-1/2}\,,$$
so that 
$$\theta^{-s_{ij}-1}\leq 2\theta^{-s_{ij}-2}\sin\theta \leq 2 \Big[2\frac{m_im_j}{(m_i+m_j)^2}(1-\cos\theta)\Big]^{-s_{ij}/2-1}\,,$$
and thus
\begin{align*}
K_{s_{ij},n,a}&\leq 2 \int^{\pi/2}_0 \Big[\frac{m_i^2+m_j^2+2m_im_j\cos\theta}{(m_i+m_j)^2}\Big]^{a}\Big[2\frac{m_im_j}{(m_i+m_j)^2}(1-\cos\theta)\Big]^{n-a-s_{ij}/2-1}\sin\theta \,d\theta\\
&= \frac{2(m_i+m_j)^2}{m_im_j}\int^1_{1/2} x^a(1-x)^{n-a-s_{ij}/2-1}dx\,, 
\end{align*}
where we used the change of variables $x=\frac{m_i^2+m_j^2+2m_im_j\cos\theta}{(m_i+m_j)^2}$ in the latter.  Hence
$$K_{s_{ij}, n, a}\leq \frac{2(m_i+m_j)^2}{m_im_j}\int^1_{1/2} x^a(1-x)^{n-a-s_{ij}/2-1}dx= \frac{(m_i+m_j)^2}{2m_im_j}\frac{\Gamma(a+1)\Gamma(n-a-s_{ij}/2)}{\Gamma(n+1-s_{ij}/2)}$$
where $\Gamma(\cdot)$ is Euler's Gamma function.  Consequently,
$$\Big(\begin{matrix}
n\\a
\end{matrix}
\Big) K_{s_{ij},n,a}\leq \frac{2(m_i+m_j)^2}{m_im_j} \frac{u_{s_{ij},n}}{(n-a-s_{ij}/2)u_{s_{ij},n-a}}$$
with $u_{s_{ij},k}=\frac{\Gamma(k+1)}{\Gamma(k+1-s_{ij}/2)}$.  By  Sirling's  Formula $\Gamma(x+1)\sim \sqrt{2\pi x}(\frac{x}{e})^x$ as $x\rightarrow \infty$; one can verify that $u_{s_{ij},k}\sim k^{s_{ij}/2}$ as $k\rightarrow \infty$ so that there exists $A_{s_{ij}}\in(0,\infty)$ such that
$$\Big(\begin{matrix}
n\\a
\end{matrix}\Big) K_{s_{ij},n,a}\leq a\,A_{s_{ij}}^2\frac{(m_i+m_j)^2}{m_im_j} \frac{n^{s_{ij}/2}}{(1-s_{ij}/2)(n-a)^{s_{ij}/2+1}}\,,$$
and 
$$\Big(\begin{matrix}
n\\a
\end{matrix}\Big) L_{s_{ij},n,a}\leq (2/\pi)^{s_{ij}}\Big(\begin{matrix}
n\\a
\end{matrix}\Big) K_{0,n,n-a}\leq (2/\pi)^{s_{ij}} a^{-1}A_0\frac{(m_i+m_j)^2}{m_im_j}.$$
Then 
$$K_{n,a}\leq \zeta \frac{(m_i+m_j)^2}{m_im_j}\Big[\frac{n^{s_{ij}/2}}{(n-a)^{s_{ij}/2+1}}+\frac{1}{a}\Big] = \lambda^2_{ij}\Big[\frac{n^{s_{ij}/2}}{(n-a)^{s_{ij}/2+1}}+\frac{1}{a}\Big], $$
where $\lambda^2_{ij}:=\zeta \frac{(m_i+m_j)^2}{m_im_j}$ and
$$\zeta= \sum_{r,s}\frac{(n-a)! a!}{(a-r-s)!(n-a-r-s)!(2r)!s!}\Big(\frac{2A^2_{s_{ij}}}{1-\frac{s_{ij}}{2}}+\big(\frac{2}{\pi}\big)^{s_{ij}}A_0\Big)\frac{(m_i+m_j)^2}{m_im_j}\,.$$
This concludes the proof.
\end{proof}
\subsection{Generation and Propagation of polynomial moments}
The purpose of Proposition \ref{Momlem} is to derive from the Boltzmann system \eqref{B-E1} an ordinary differential inequality for polynomial moment of order $2n$, with $n\geq2$, based on the Povzner estimate from Lemma \ref{Povlem} to implement statistical moment analysis.  In particular, the key idea in regard of \textit{exponential tails} generation and propagation is to realise that the mixture system's control inequality is obtained by individualising the $ij$-scalar components satisfying the Povzner lemma, Lemma \ref{Povlem}, specifically described in Proposition \ref{Momlem} item ${\rm (i)}$. %
\begin{proof}[Proof of Proposition \ref{Momlem}]
Let $n\geq 2$ and $1\leq i,j\leq I$, and recall that
\begin{align}
\begin{split}\label{componetij}
\partial_t{\bf m}^{ij}_{2n}=\frac{1}{2}\int_{\mathbb{R}^3}\int_{\mathbb{R}^3}\int_0^\pi\int_0^{2\pi}\Big(\langle v'\rangle_i^{2n}+ \langle v_*'\rangle_j^{2n} - \langle v&\rangle_i^{2n} - \langle v_*\rangle_j^{2n}\Big)d\varphi\, \beta_{ij}(\theta)\,d\theta\\
&\times\,|v-v_*|^{\lambda_{ij}}\,f_i(v)\,f_j(v_*)\,dv_* \,dv\,.
\end{split}
\end{align}
Consequently, a direct consequence of the weak formulation of the Boltzmann mixture equation \eqref{Q4} with $\psi_i=\langle v\rangle_i^{2n}$ is that
\begin{equation}\label{Diffmom}
{\bf m}'_{2n}(t)=\sum^{I}_{i=1}\sum^{I}_{j=1} \partial_t{\bf m}^{ij}_{2n}(t),\qquad t>0\,.
\end{equation} 
Using Povzner Lemma \ref{Povlem}, we get that $\partial_t{\bf m}^{ij}_{2n}(t)\leq -{\bf A}_{n}^{ij}(t)+{\bf B}_{n}^{ij}(t)$ where
$${\bf A}_{n}^{ij}(t)=\frac{1}{2}\lambda_{ij}^1n^{\frac{s_{ij}}{2}}\int_{\mathbb{R}^3}\int_{\mathbb{R}^3}\Big(\langle v\rangle_i^{2n}+\langle v_*\rangle_j^{2n}\Big)|v-v_*|^{\lambda_{ij}}f_i(v)f_j(v_*)dv_*dv,$$
and
\begin{align*}{\bf B}_{n}^{ij}(t)=\frac{1}{2}\lambda^2_{ij}\sum^{n-1}_{a=1}\Big(\begin{matrix}n\\a\end{matrix}\Big)\Big[\frac{n^{s_{ij}/2}}{(n-a)^{s_{ij}/2+1}} + \frac{1}{a}\Big]&\int_{\R^3}\int_{\R^3}\Big(\langle v\rangle_i^{2a}\langle v_*\rangle_j^{2(n-a)}+\langle v\rangle_i^{2(n-a)}\langle v_*\rangle_j^{2a}\Big)\\
&\times |v-v_*|^{\lambda_{ij}}\,f_i(v)\,f_j(v_*)\,dv_*\,dv\,.
\end{align*}
Let us begin proving the first result by proceeding in two steps.

\medskip
\noindent
{\bf Step 1.} Finding a lower bound for $A_n$.  Since $\F \in U(D_0, E_0)$, one can use a slight modification of the proof of \cite[Lemma 9]{AG} to obtain an explicit constant $c:=c(D_0,E_0)>0$ such that
\begin{equation*}
\int_{\mathbb{R}^3}|v-z|^{\lambda}f_i(z)dz\geq c\,\langle v \rangle^{\lambda}\,,\qquad\quad \lambda\in[0,2]\,,\quad 1\leq i \leq I\,.
\end{equation*} 
Consequently,
\begin{align*}
{\bf A}_{n}^{ij}(t)&= 2\lambda^1_{ij}n^{\frac{s_{ij}}{2}}\int_{\mathbb{R}^3}\int_{\mathbb{R}^3}\big(\langle v\rangle_i^{2n}+\langle v_*\rangle_j^{2n}\big)|v-v_*|^{\lambda_{ij}}f_i(v)f_j(v_*)dv_*dv\\
&\geq 2\lambda^1_{ij}n^{\frac{s_{ij}}{2}}\,c\,\Big( {\bf m}_{2n+\lambda_{ij},i}(t) + {\bf m}_{2n+\lambda_{ij},j}(t) \Big)\,.
\end{align*}
{\bf Step 2.} Finding a convenient upper bound for ${\bf B}_n^{ij}$.  Using that, see for example \cite{GP}, 
$$|v-v_*|^{\lambda_{ij}}\leq (m_im_j)^{\frac{\lambda_{ij}}{2}}\big(\langle v\rangle_i^{\lambda_{ij}}+\langle v_*\rangle_j^{\lambda_{ij}}\big),$$
we get that 
\begin{align*}
{\bf B}_{n}^{ij}(t)&\leq 2  (m_im_j)^{\frac{\lambda_{ij}}{2}}\lambda^2_{ij}\sum^{n-1}_{a=1}\Big(\begin{matrix}n\\a\end{matrix}\Big)\Big[\frac{n^{s_{ij}/2}}{(n-a)^{{s}_{ij}/2+1}}+\frac{1}{a}\Big]\Big({\bf m}_{2a+\lambda_{ij},i}(t){\bf m}_{2(n-a),j}(t) \\
&\qquad+ {\bf m}_{2a,j}(t){\bf m}_{2(n-a) +\lambda_{ij},i}(t) +{\bf m}_{2a+\lambda_{ij},j}(t){\bf m}_{2(n-a),i}(t)+{\bf m}_{2a,i}(t){\bf m}_{2(n-a)+\lambda_{ij},j}\Big)\\
 &= 2 (m_im_j)^{\frac{\lambda_{ij}}{2}}\lambda^2_{ij}\sum^{n-1}_{a=1}\Big(\begin{matrix}n\\a\end{matrix}\Big)\Big[\frac{n^{s_{ij}/2}}{a^{s_{ij}/2+1}}+\frac{1}{a}\Big]\Big({\bf m}_{2a+\lambda_{ij},i}(t){\bf m}_{2(n-a),j}(t)\\
&\qquad+{\bf m}_{2a,j}(t){\bf m}_{2(n-a)+\lambda_{ij},i}(t) +{\bf m}_{2a+\lambda_{ij},j}(t){\bf m}_{2(n-a),i}(t)+{\bf m}_{2a,i}(t){\bf m}_{2(n-a)+\lambda_{ij},j}\Big)\\
&\leq 4 (m_im_j)^{\frac{\lambda_{ij}}{2}}\lambda^2_{ij}\sum^{n-1}_{a=1}\Big(\begin{matrix}n\\a\end{matrix}\Big)\frac{n^{s_{ij}/2}}{a^{s_{ij}/2+1}}\Big({\bf m}_{2a+\lambda_{ij},i}(t){\bf m}_{2(n-a),j}(t)\\
&\qquad+{\bf m}_{2a,j}(t){\bf m}_{2(n-a)+\lambda_{ij},i}(t) + {\bf m}_{2a+\lambda_{ij},j}(t){\bf m}_{2(n-a),i}(t)+{\bf m}_{2a,i}(t){\bf m}_{2(n-a)+\lambda_{ij},j}\Big)\,,
 \end{align*}
where we performed the change of index $a\rightarrow n-a$ in the intermediate step, and the fact that $\frac{1}{a}\leq \frac{n^{s_{ij}/2}}{a^{s_{ij}/2+1}}$ for $a\leq n$.  Furthermore, the terms in parenthesis are symmetric with respect to $a$, that is
\begin{align*}
{\bf m}_{2a+\lambda_{ij},i}(t){\bf m}_{2(n-a),j}(t) &\quad \text{for}\quad a=1,2,3,\cdots \quad\text{equals}\\
{\bf m}_{2a,j}(t){\bf m}_{2(n-a)+\lambda_{ij},i}(t) &\quad \text{for}\quad a=n-1,n-2,n-3,\cdots\,.
\end{align*}
And similarly for the second group of moments.  Note, in addition, that the first half of weights controls the second half because $a\rightarrow \frac{n}{a}$ is decreasing.  Hence,
\begin{align*}
{\bf B}_{n}^{ij}(t)&\leq 8 (m_im_j)^{\frac{\lambda_{ij}}{2}}\lambda^2_{ij}\sum^{\lfloor n/2 \rfloor}_{a=1}\Big(\begin{matrix}n\\a\end{matrix}\Big)\frac{n^{s_{ij}/2}}{a^{s_{ij}/2+1}}\Big({\bf m}_{2a+\lambda_{ij},i}(t){\bf m}_{2(n-a),j}(t)\\
&\qquad+{\bf m}_{2a,j}(t){\bf m}_{2(n-a)+\lambda_{ij},i}(t) + {\bf m}_{2a+\lambda_{ij},j}(t){\bf m}_{2(n-a),i}(t)+{\bf m}_{2a,i}(t){\bf m}_{2(n-a)+\lambda_{ij},j}\Big)\,.
\end{align*}
Now, since $n-a\geq a$ for $1\leq a \leq\lfloor n/2 \rfloor$ we can invoke the Lemma \ref{App-interp} with $b=2(n-a)$ and $\beta=\lambda_{ij}$, which is a generalisation of the classical lemma \cite[Lemma 7]{NF}, to conclude that
\begin{equation*}
{\bf m}_{2a+\lambda_{ij},i}(t){\bf m}_{2(n-a),j}(t) \leq \theta_{ij}{\bf m}_{2(n-a)+\lambda_{ij},i}(t){\bf m}_{2a,j}(t)+ (1-\theta_{ij}){\bf m}_{2(n-a)+\lambda_{ij},j}(t){\bf m}_{2a,i}(t)\,,
\end{equation*}
resulting in the estimation
\begin{align*}
{\bf B}_{n}^{ij}(t)&\leq 16 (m_im_j)^{\frac{\lambda_{ij}}{2}}\lambda^2_{ij}\sum^{\lfloor n/2 \rfloor}_{a=1}\Big(\begin{matrix}n\\a\end{matrix}\Big)\frac{n^{s_{ij}/2}}{a^{s_{ij}/2+1}}\Big({\bf m}_{2a,j}(t){\bf m}_{2(n-a)+\lambda_{ij},i}(t) +{\bf m}_{2a,i}(t){\bf m}_{2(n-a)+\lambda_{ij},j}\Big)\,.
\end{align*}
Consequently,
\begin{equation}\label{diffmomt}
\partial_t{\bf m}^{ij}_{2n}(t)\leq -C^1_{ij} n^{\frac{s_{ij}}{2}} \big({\bf m}_{2n+\lambda_{ij},i}(t)+{\bf m}_{2n+\lambda_{ij},j}(t)\big)+C_{ij}^2 \,S^{ij}_{n}(t)\,. 
\end{equation}
where the constants can be taken as $$C^1_{ij}=2 \, \lambda^1_{ij}\, c \qquad \text{and} \qquad C^2_{ij}=16(m_{i}m_{j})^{\frac{\lambda_{ij}}{2}}\lambda^2_{ij}\,.$$
This is precisely statement (i).

\medskip
\noindent
Now, statement (ii) follow from (i) after some observations.  First, for the coercive term it readily follows that
\begin{align*}
\sum^{I}_{i=1}\sum^{I}_{j=1}C^1_{ij}n^{s_{ij}/2}\big({\bf m}_{2n+\lambda_{ij},i}(t)&+{\bf m}_{2n+\lambda_{ij},j}(t)\big)\geq 2\min_{ij}\{C^1_{ij}n^{s_{ij}/2}\}\sum^{I}_{i=1}\sum^{I}_{j=1}{\bf m}_{2n+\lambda_{ij},i}(t)\\
&\geq 2\,C_1\,n^{\frac{ \bar{s}}{2}}\sum^{I}_{i=1}{\bf m}_{2n+\overline{\bar{\lambda}}_{i},i}(t) \geq 2\,C_1\,n^{\frac{ \bar{s}}{2}}\,{\bf m}_{2n+\lambda^{\natural}}(t)\,.
\end{align*}
The upper estimate related to  $S^{ij}_n$ follows after some moment interpolations.  Note that (since $n/a\geq1$)
\begin{align*}
\sum^{I}_{i=1}\sum^{I}_{j=1}C^{2}_{ij}&\,S^{ij}_{n}(t)\leq 2  \max_{1\leq i,j\leq I}C^{2}_{ij} \sum^{\lfloor n/2\rfloor}_{a=1}\Big(\begin{matrix}n\\a\end{matrix}\Big)\frac{n^{\overline{\bar{s}}/2}}{a^{\overline{\bar{s}}/2+1}}{\bf m}_{2a}(t){\bf m}_{2(n-a)+\overline{\overline{\lambda}}}(t)\,,
\end{align*}
Furthermore, $2 < 2(n-a)+\overline{\overline{\lambda}} < 2n$ (recalling that $n\geq2$), therefore, the following interpolations hold
\begin{align*}
{\bf m}_{2(n-a)+\overline{\overline{\lambda}}}(t)&\leq {\bf m}_{2}^{1-\theta_1}\,{\bf m}^{\theta_1}_{2n + \lambda^\natural}(t)\,,\qquad \theta_1=\frac{2(n-a)+\overline{\overline{\lambda}}-2}{2n+\lambda^\natural-2}\,,\\
{\bf m}_{2a}(t)&\leq {\bf m}_{2}^{1-\theta_2}\,{\bf m}^{\theta_2}_{2n+\lambda^\natural}(t)\,,\qquad \theta_2=\frac{2a-2}{2n+\lambda^\natural-2}\,,
\end{align*}
which yields that (since $\overline{\overline{\lambda}}- \lambda^{\natural}<2$)
 $${\bf m}_{2a+\overline{\overline{\lambda}}}(t)\,{\bf m}_{2(n-a)}(t)\leq {\bf m}^{2-\theta_1-\theta_2}_{2}{\bf m}^{\theta_1+\theta_2}_{2n+\lambda^\natural}\,,\qquad \theta_1+\theta_2 = \frac{2n+\overline{\overline{\lambda}}-4}{2n+\lambda^\natural-2}<1\,.$$
Consequently,
\begin{align*}
\sum^{I}_{i=1}\sum^{I}_{j=1}C^{2}_{ij}\,S^{ij}_{n}(t)&\leq C({\bf m}_2)\,2^{\frac n2}\, n^{\frac{\overline{\overline{s}}}{2}}\,{\bf m}^{\theta}_{2n+\lambda^\natural}(t)\,,\qquad \theta=\frac{2n+\overline{\overline{\lambda}}-4}{2n+\lambda^\natural-2}<1\,,\\\
 &\leq \Big(C\,2^{\frac n2} \, n^{\frac{\overline{\overline{s}}}{2}-\frac{\theta\bar{s}}{2}}\Big)^{\frac{1}{1-\theta}} + C_1\;n^{\frac{\bar{s}}{2}}\;  {\bf m}_{2n+{\lambda}^\natural}(t)\,,
\end{align*}
where in the latter we used Young's inequality.  Consequently,
$$\sum^{I}_{i=1}\sum^{I}_{j=1}\partial_t{\bf m}^{ij}_{2n}(t)\leq - C_1\;n^{\frac{ \bar{s}}{2}}\;  {\bf m}_{2n+{\lambda}^\natural}(t)+\Big(C\,2^{\frac n2}\, n^{\frac{\overline{\overline{s}}}{2}-\frac{\theta \bar{s}}{2}}\Big)^{\frac{1}{1-\theta}}\,,$$
which proves item (ii) of the proposition.
\end{proof}
\begin{proof}[Proof of Theorem \ref{Mom}]
Use the interpolation
\begin{equation*}
{\bf m}_{2n}(t)\leq {\bf m}^{1-\theta}_{2}\,{\bf m}^{\theta}_{2n + \lambda^\natural}(t)\,,\qquad \theta=\frac{2n-2}{2n+\lambda^\natural-2}<1\,,
\end{equation*}
in Proposition \ref{Momlem} item ${\rm (ii)}$ to conclude that
\begin{align*}
{\bf m}'_{2n}(t)& \leq  - C_1\;n^{\frac{ \bar{s}}{2}}\; {\bf m}^{\frac{1-\theta}{\theta}}_{2}\;{\bf m}^{1+\frac{\lambda^\natural}{2n-2}}_{2n}(t)+\Big(C\,2^{\frac n2}\,n^{\frac{\overline{\overline{s}}}{2}-\frac{\theta
\bar{s}}{2}}\Big)^{\frac{1}{1-\theta}}\\
&=: - A_{n}\,{\bf m}^{1+\frac{\lambda^\natural}{2n-2}}_{2n}(t)+B_{n}\,,\qquad c_{n}=\frac{\lambda^\natural}{2n-2}\,.
\end{align*}
Thus, a comparison principle for order, see \cite[Lemma 18]{AG}, gives that the super solution $ {\bf m}^*_{2n}(t)$ controls the $2n^{th}$-moment
$${\bf m}_{2n}(t)\leq {\bf m}^*_{2n}(t) := \Big(\frac{B_n}{A_n}\Big)^\frac{1}{1+c_n}+\Big(\frac{1}{c_n\,A_{n}}\Big)^{\frac{1}{c_n}}\,t^{-\frac{2n-2}{\lambda^\natural}}\,.$$
Similarly, if $m_{2n}(0)<\infty$ one deduces the propagation bound
$${\bf m}_{2n}(t)\leq \max\Big\{ m_{2n}(0) , \Big(\frac{B_n}{A_n}\Big)^\frac{1}{1+c_n}\Big\}\,.$$
For a general $r>2$, consider the minimal integer $n\geq2$ such that $2< r \leq 2n$, then
\begin{equation*}
{\bf m}_{r}(t)\leq {\bf m}^{1-\tilde\theta}_{2}\,{\bf m}^{\tilde\theta}_{2n}(t)\,,\qquad \tilde\theta=\frac{r-2}{2n-2}\leq1\,.
\end{equation*}
Thus, we conclude that
\begin{equation*}
{\bf m}_r(t) \leq {\bf m}^{1-\tilde\theta}_{2}\bigg( \Big(\frac{B_n}{A_n}\Big)^\frac{\tilde\theta}{1+c_n}+\Big(\frac{1}{c_n\,A_{n}}\Big)^{\frac{\tilde\theta}{c_n}}\,t^{-\frac{r-2}{\lambda^\natural}}\bigg)\,.
\end{equation*}
This concludes the proof of the result.
\end{proof}
\subsection{Exponential moments}
Let us now focus in the generation and propagation of exponential moments result of Theorem \ref{theo1}. Our goal is to show that the partial sums
$$
\sum_{n=0}^p \frac{(\sigma t)^{\rho n}{\bf m}_{2n}[\F](t)}{ (n!)^{\alpha}}
$$
are bounded uniformly in time $t$ and summation truncation $p$ for some explicit $\rho>0$ and $\alpha>0$.  To this end we use a methodology reminiscent of \cite{NF, ACG}. The proof relies on generation and propagation of polynomial moments stated in Theorem \ref{Mom} and the following lemma.

\begin{lem}[Equivalence property]\label{lemexp1}
Let $\G=[g_i]_{1\leq i\leq I}\in L^1(\R^3)$, then we have the following equivalence:
\begin{itemize}
\item[$(i)$] (System) For $\sigma_0\in (0,\infty),\; \alpha\in [1,\infty)$ and $K\in[1,\infty)$, we have 
$$
\sup_{n\geq 0} \frac{\sigma_0^n\, {\bf m}_{2n}[\G]}{(n!)^\alpha}\leq K\; \Longrightarrow \;\sum_{i=1}^I\int_{\R^3}\exp\Big(\frac{\sigma_0^{1/\alpha}\langle v\rangle_i^{2/\alpha}}{2}\Big)g_i(v)dv\leq 2 \, {\bf m}_0^{1-\frac1\alpha}K^{\frac1\alpha}.
$$
\item[$(ii)$] (Component) Let $\rho\in(0,2]$, $\sigma_{0}\in (0,1]$ and $K\geq1$, there exists $\sigma_{1}:=\sigma_{1}(\rho, \sigma_{0}, K)$ such that 
$$
\int_{\R^3}\exp\big( \sigma_0\langle v \rangle_i^\rho \big)\,g_i(v)\,dv\leq \frac{K}{I}\;\Longrightarrow \; \sup_{n\geq 0}\frac{\sigma_1^n\,{\bf m}_{2n,i}[g_i]}{(n!)^{2/\rho}}\leq \frac{1}{I}\,.
$$
\end{itemize}
\end{lem}
\begin{proof}
We start with $(i)$ noticing that by Holder's inequality (recall that $\alpha\geq 1$),
$${\bf m}_{\frac{2n}{\alpha}}[\G]\leq {\bf m}_{2n}[\G]^{\frac{1}{\alpha}}{\bf m}_0^{1-\frac{1}{\alpha}}\leq {\bf m}_0^{1-\frac{1}{\alpha}}K^{\frac{1}{\alpha}}\sigma_0^{-\frac{n}{\alpha}}\,n!\,.$$
Then,
$$\sum_{i=1}^I\int_{\R^3}\exp\Big(\frac{\sigma_0^{\frac{1}{\alpha}}\langle v\rangle_i^{\frac{2}{\alpha}}}{2}\Big)g_i(v) dv= \sum_{n\geq 0}\frac{\sigma_0^\frac{n}{\alpha}{\bf m}_{\frac{2n}{\alpha}}[\G]}{2^n n!}\leq {\bf m}_0^{1-\frac{1}{\alpha}}K^{\frac{1}{\alpha}}\sum_{n\geq 0}2^{-n}=2\,{\bf m}_0^{1-\frac{1}{\alpha}}K^{\frac{1}{\alpha}}\,.
$$
Regarding to $(ii)$, note that
$$\sup_{n\geq 0} \frac{\sigma_{0}^n{\bf m}_{\rho n,i}[g_i]}{n!}\leq \sum_{n\geq 0}\frac{\sigma_{0}^n {\bf m}_{\rho n,i}[g_i]}{n!}=\int_{\R^3}\exp\big(\sigma_{0}\langle v\rangle_i^{\rho}\big)g_i(v) dv \leq \frac{K}{I}\,.$$
For $n\geq 1$ we set $k_{n}=\lceil \frac{2n}{\rho}\rceil \in [\frac{2n}{\rho}, \frac{2n}{\rho}+1)$.  Using that $\rho \, k_{n}\geq 2n$ we obtain that 
$${\bf m}_{2n,i}[g_i]\leq {\bf m}_{\rho k_{n},i}[g_i]\leq \frac{K k_{n}!}{I\sigma_{0}^{k_{n}}}\,,$$
and consequently,
$$\frac{\sigma_{1}^n {\bf m}_{2n,i}[g_i]}{(n!)^{2/\rho}}\leq \frac{\sigma_{1}^n K k_{n}!}{I\sigma_{0}^{k_{n}}(n!)^{2/\rho}}.$$
Invoking Stirling`s formula $ n!\sim \sqrt{2\pi n}(n/e)^n$ and since $k_n=\lceil\frac{2n}{\rho}\rceil$  we conclude that for some constant $A\in(0,\infty)$ depending only on $\rho$ it holds that 
$$\frac{k_{n}!}{(n!)^{2/\rho}}\leq A \bigg(\frac{4}{\rho}\bigg)^{\frac{2n}{\rho}},\qquad n\geq 1\,.$$
We refer to \cite{NF} for additional details.   Moreover, since $\sigma_{0}\in(0,1]$ then $\sigma_{0}^{k_{n}}\geq \sigma_{0}^{\frac{2n}{\rho}+1}$, and consequently 
$$
\frac{\sigma_{1}^n {\bf m}_{2n,i}[g_i]}{(n!)^{2/\rho}}\leq \frac{K\,A\,\sigma_{1}^n}{I\sigma_{0}^{\frac{2n}{\rho}+1}}\bigg(\frac{4}{\rho}\bigg)^{\frac{2n}{\rho}}<1,\qquad n\geq \frac{1}{I}\,,
$$
after choosing, for example,
$$
\sigma_{1}\leq \frac{1}{2}\inf_{n\geq1}\frac{I\sigma^{\frac{2}{\rho}+1}_{0}}{(K\,A)^{\frac{1}{n}}}\Big(\frac{\rho}{4}\Big)^{\frac{2}{\rho}}\,.
$$
This concludes the proof.
\end{proof}
\subsubsection{ Generation of exponential moments.}
First we introduce for $\sigma\in (0,1]$ and $\alpha>0$, to be chosen later sufficiently small and large respectively, $p\geq 2$ and $t>0$
$$
E_p[\F](t):=\sum_{n=0}^p \frac{(\sigma t)^{2n/{\lambda^\natural}}{\bf m}_{2n}[\F](t)}{ (n!)^{\alpha}}.
$$
Additionally, we define the objects, for $1\leq i,j\leq I$,  
$$E_{p}^{i}[f_i](t):=\sum_{n=0}^p \frac{(\sigma t)^{2n/{\lambda^\natural}}{\bf m}_{2n,i}[f_i](t)}{ (n!)^{\alpha}},\qquad F_{p}^{ij}[f_i](t):=\sum_{n=2}^p\frac{n^{s_{ij}/2}(\sigma t)^{2n/\lambda^\natural}{\bf m}_{2n+\lambda_{ij},i}[f_i](t)}{(n!)^{\alpha}},$$
$$G_{p}^{ij}[f_i,f_j](t):=\sum_{n=2}^p \frac{(\sigma t)^{2n/\lambda^\natural}\mathcal{S}_{n}^{ij}(t)}{(n!)^{\alpha}},\qquad H_{p}^i[f_i](t):=\sum_{n=0}^p \frac{n(\sigma t)^{2n/{\lambda^\natural-1}}{\bf m}_{2n,i}[f_i](t)}{ (n!)^{\alpha}}.$$
We start by time differentiation of $E_p[\F](t)$ and then using Proposition \ref{Momlem} item (i). Since ${\bf m}_{0}'={\bf m}_{2}'=0$, we deduce that
\begin{align}\label{add-moments}
\begin{split}
\frac{{\rm d}}{{\rm d}t} E_p[\F](t)&=\sum_{n=0}^p \bigg[\frac{(\sigma t)^{2n/{\lambda^\natural}}{\bf m}'_{2n}[\F](t)}{ (n!)^{\alpha}} + \frac{2\sigma}{\lambda^\natural}\frac{n(\sigma t)^{2n/{\lambda^\natural}-1}{\bf m}_{2n}[\F](t)}{ (n!)^{\alpha}}\bigg]\\
&=\sum_{n=0}^p\bigg[\frac{(\sigma t)^{2n/{\lambda^\natural}}\sum_{i,j=1}^I \partial_t{\bf m}^{ij}_{2n}(t)}{ (n!)^{\alpha}}+\frac{2\sigma}{\lambda^\natural} \frac{n(\sigma t)^{2n/{\lambda^\natural-1}}\sum^{I}_{i=1}{\bf m}_{2n,i}[f_i](t)}{(n!)^{\alpha}}\bigg]\\
&\leq \sum^I_{i,j=1} \bigg[ -C^1_{ij}\Big(F_{p}^{ij}[f_i](t)+F_p^{ji}[f_j](t)\Big)+C^2_{ij}G_{p}^{ij}[f_i,f_j](t)\bigg]+\frac{2\sigma}{\lambda^\natural}  \sum^{I}_{i=1}H_{p}^i[f_i](t). 
\end{split}
\end{align}
The following lemma is crucial as it expresses the convolution structure of the moments in the summation.  This is where the nonlinear part is estimated.
\begin{lem}[Moment Convolution]\label{Gen1}
Let $1\leq i,j\leq I$, $\lambda^\natural>0$, $s_{ij}\in(0,2)$, and $\F_0\in U(D_0,E_0)$.  Given $\epsilon>0$ there are explicit constants $B_{\epsilon}$ and $D_{\epsilon}$, depending on the parameters of the model, $E_0$, and $D_0$, such that for $\sigma\in(0,1]$, $\alpha\geq 1$ and $p\geq 2$, the following estimate holds
\begin{align*}
G^{ij}_{p}[f_i, f_j](t) \leq &B_{\epsilon}\,\sigma^{2/\lambda^\natural}\,\Big[ F^{ij}_p[f_i](t)+F^{ji}_p[f_j](t)\Big]\\
&+ \epsilon\, \Big[F^{ij}_p[f_i](t)E^{j}_p[f_j](t)+F^{ji}_p[f_j](t)E^{i}_p[f_i](t)\Big]+\sigma^{2/\lambda^\natural}\,D_{\epsilon},\qquad t\in(0,1]\,.
\end{align*}
\end{lem}
\begin{proof}
Using the definition of $\mathcal{S}_{n}^{ij}(t)$ it follows that
\begin{align*}
G_{p}^{ij}[f_i,f_j](t)=\sum_{n=2}^p (\sigma t)^{2n/\lambda^\natural}\sum^{\lfloor n/2 \rfloor}_{a=1}\Big(\begin{matrix}n\\a\end{matrix}\Big)\frac{n^{s_{ij}/2}}{a^{s_{ij}/2+1}(n!)^{\alpha}}&\Big({\bf m}_{2a,i}(t){\bf m}_{2(n-a)+\lambda_{ij},j}(t)\\
&\qquad\qquad +{\bf m}_{2a,j}(t){\bf m}_{2(n-a)+\lambda_{ij},i}(t)\Big).
\end{align*}
For $1\leq a\leq \lfloor n/2 \rfloor$ one has that $n^{s_{ij}/2}\leq 2(n-a)^{s_{ij}/2}$.  Therefore,
\begin{align*}
G^{ij}_{p}[f_i,f_j](t)&\leq 2\sum_{n=2}^p (\sigma t)^{2n/\lambda^\natural}\sum^{n-1}_{a=1}\Big(\begin{matrix}n\\a\end{matrix}\Big)^{\alpha}\frac{(n-a)^{s_{ij}/2}}{a^{s_{ij}/2+1}(n!)^{\alpha}}\Big({\bf m}_{2a,i}(t){\bf m}_{2(n-a)+\lambda_{ij},j}(t) \\
&\hspace{7cm}+{\bf m}_{2a,j}(t){\bf m}_{2(n-a)+\lambda_{ij},i}(t)\Big)\\
&= 2\sum_{n=2}^p \sum^{n-1}_{a=1}\frac{(\sigma t)^{2a/\lambda^\natural}{\bf m}_{2a,i}(t)}{(a!)^{\alpha}}\frac{(n-a)^{s_{ij}/2}(\sigma t)^{2(n-a)/\lambda^\natural}{\bf m}_{2(n-a)+\lambda_{ij},j}(t)}{a^{s_{ij}/2+1}((n-a)!)^{\alpha}}\\
&\qquad + \frac{(n-a)^{s_{ij}/2}(\sigma t)^{2(n-a)/\lambda^\natural}{\bf m}_{2(n-a)+\lambda_{ij},i}(t)}{((n-a)!)^{\alpha}}\frac{(\sigma t)^{2(n-a)/\lambda^\natural)}{\bf m}_{2a,j}(t)}{a^{s_{ij}/2+1}(a!)^{\alpha}}\\
&=:\mathcal{G}_{p,1}^{ij}[f_i,f_j](t)+\mathcal{G}_{p,2}^{ij}[f_i,f_j](t),
\end{align*}
with obvious definitions for each term.  Now, note that interchanging the summation it hols that
\begin{align*}
 \mathcal{G}_{p,1}^{ij}[f_i,f_j](t)&= 2\sum^{p-1}_{a=1}\sum_{l=a}^{p-a}\frac{l^{s_{ij}/2}(\sigma t)^{2l/\lambda^\natural}{\bf m}_{2l+\lambda_{ij},j}(t)}{(l!)^{\alpha}}\frac{(\sigma t)^{2a/\lambda^\natural}{\bf m}_{2a,i}(t)}{a^{s_{ij}/2+1}(a!)^{\alpha}}\\
 &\leq 2\bigg(\sum_{l=1}^{p}\frac{l^{s_{ij}/2}(\sigma t)^{2l/\lambda^\natural}{\bf m}_{2l+\lambda_{ij},j}(t)}{(l!)^{\alpha}}\bigg)\bigg(\sum_{a=1}^p\frac{(\sigma t)^{2a/\lambda^\natural}{\bf m}_{2a,i}(t)}{a^{s_{ij}/2+1}(a!)^{\alpha}}\bigg)\\
 & = \Big(F^{ji}_{p}[f_j](t)+(\sigma t)^{2/\lambda^\natural}{\bf m}_{2+\lambda_{ij},j}(t)\Big)I^{ij}_{p,1}[f_i](t), 
\end{align*}
where $$I^{ij}_{p,1}[f_i](t):=2\sum_{a=1}^p\frac{(\sigma t)^{2a/\lambda^\natural}{\bf m}_{2a,i}(t)}{a^{s_{ij}/2+1}(a!)^{\alpha}}.$$
Setting $N_{\epsilon}=\lceil 2\,{\epsilon}^{-1}\rceil$, it holds that for $a\geq N_{{\epsilon}}$ implies $2\,a^{-1}\leq \epsilon$.  Hence,
$$I^{ij}_{p,1}[f_j](t)\leq \epsilon\, E^{i}_{p}[f_i](t)+ J_{\epsilon}^{i}[f_i](t) \quad{\rm where}\quad J_{\epsilon}^{i}[f_i](t)=2\sum_{a=1}^{N_{\epsilon}}\frac{(\sigma t)^{2a/\lambda^\natural}{\bf m}_{2a,i}(t)}{(a!)^{\alpha}}.$$
Using the first result of Theorem \ref{Mom}, we deduce that  there are explicit constants $C$ and $C_{\epsilon}$ such that
$$(\sigma t)^{2/\lambda^\natural}{\bf m}_{2+\lambda_{ij},j}(t)\leq C(\sigma t)^{2/\lambda^\natural}(1+t^{-\frac{\lambda_{ij}}{\lambda^\natural}})\leq 2\,C \, \sigma^{2/\lambda^\natural},\qquad\sigma\in(0,1], \quad t\in(0,1]\,,$$
and
\begin{align*}
J_{\epsilon}^{i}[f_i](t)&\leq C_{\epsilon}\sum_{a=1}^{N_{\epsilon}}\frac{(\sigma t)^{2a/\lambda^\natural}(1+t^{-\frac{2a-2}{\lambda^\natural}})}{(a!)^{\alpha}}\leq C_{\epsilon}\sum_{a=1}^{\infty}\frac{2}{(a!)^{\alpha}} =: A_{\epsilon}\,\sigma^{2/\lambda^\natural}\,.
\end{align*}
Thus,
\begin{equation}\label{G2}
\mathcal{G}_{p,1}^{ij}[f_i,f_j](t)\leq \Big(F^{ji}_{p}[f_j](t)+2\,C\,\sigma^{2/\lambda^\natural}\Big)\Big(\epsilon\, E_{p}^{i}[f_i](t)+A_{\epsilon}\,\sigma^{2/\lambda^\natural}\Big).
\end{equation}
Similarly, we also have that
\begin{equation}\label{G4}
\mathcal{G}_{p,2}^{ij}[f_i,f_j](t)\leq \Big(F^{ij}_{p}[f_i](t)+2\,C\,\sigma^{2/\lambda^\natural}\Big)\Big(\epsilon\, E^{j}_{p}[f_j](t)+A_{\epsilon}\,\sigma^{2/\lambda^\natural}\Big).
\end{equation}
In this way, adding \eqref{G2} and \eqref{G4}, it holds that
\begin{align*}
G_{p}^{ij}[f_i,f_j](t)&\leq  \mathcal{G}_{p,1}^{ij}[f_i,f_j](t)+\mathcal{G}_{p,2}^{ij}[f_i,f_j](t)\\
&\leq \sigma^{2/\lambda^\natural} A_{\epsilon}\Big[ F^{ij}_p[f_i](t)+F^{ji}_{p}[f_j](t)\Big] +2\,\sigma^{2/\lambda^\natural}\epsilon\,C\,\Big[E^{i}_{p}[f_i](t)+E^{j}_{p}[f_j](t)\Big] \\
&\qquad+ {\epsilon} \Big[F^{ij}_p[f_i](t)E^{j}_{p}[f_j](t)+F^{ji}_{p}[f_j](t)E^{i}_{p}[f_i](t)\Big]+4\,\sigma^{4/\lambda^\natural}A_{\epsilon}\,C\,.
\end{align*}
Furthermore, for $t\in(0,1]$ it holds that
\begin{align*}
E^{i}_{p}[f_i](t)+E^{j}_{p}[f_j](t)&\leq {\bf m}_{0}(t)+\sigma \,{\bf m}_2(t)+F^{ij}_{p}[f_i](t)+F^{ji}_{p}[f_j](t)\,.
\end{align*} 
Consequently, we are led to
\begin{align*}
G^{ij}_{p}[f_i, f_j](t)&\leq B_{\epsilon}\,\sigma^{2/\lambda^\natural}\Big[ F^{ij}_p[f_i](t)+F^{ji}_{p}[f_j](t)\Big]\\
&\qquad+ \epsilon\, \Big[F^{ij}_p[f_i](t)E^{j}_{p}[f_j](t)+F^{ji}_{p}[f_j](t)E^{i}_{p}[f_i](t)\Big]+\sigma^{2/\lambda^\natural}\,D_{\epsilon},\qquad t\in(0,1]\,,
\end{align*}
where $$B_{\epsilon}:=2 A_{\epsilon}+2\,\epsilon\,C\,,\quad\text{and}\quad D_{\epsilon}:=4\,\sigma^{2/\lambda^\natural}A_{\epsilon}\,C+2\,\epsilon\,C\,\big({\bf m}_{0}+\sigma \,{\bf m}_2\big)\,.$$
This concludes the proof.
\end{proof}
\begin{lem}\label{Gen2}
Let $1\leq i,j\leq I$, $\lambda^\natural>0$, $s^\natural>0$, $\alpha =\max \big\{1, \frac{2-s^\natural}{\lambda^\natural}\big\}$, and $\F_0\in U(D_0,E_0)$.  There are explicit constants $K$ and $L$, depending on the parameters of the model, $E_0$, and $D_0$, such that for any $\sigma\in(0,1]$ and $p\geq 2$, it holds that
$$H^i_{p}[f_i](t)\leq K\, F^{i}_{p}[f_i](t)+L,\qquad t\in(0,1]\,,$$
where $F^i_p[f_i](t):=\sum_{n=2}^p  \frac{n^{\bar{\bar{s}}_{i}/2}(\sigma t)^{2n/\lambda^\natural}{\bf m}_{2n+\bar{\bar{\lambda}}_i,i}(t)}{(n!)^{\alpha}}.$
\end{lem}
\begin{proof}
First observe that for $\kappa \geq 1$, to be chosen later sufficiently large,
\begin{itemize}
\item[$\cdot$] If ${\bf m}_{2n,i}(t)\leq \Big(\frac{n^{1-\frac{\bar{\bar{s}}_{i}}{2}}}{\kappa\sigma t}\Big)^{\frac{2n-2}{\bar{\bar{\lambda}}_i}}$, then $\frac{n\,{\bf m}_{2n,i}(t)}{\sigma t}\leq \frac{n}{\sigma t}\Big(\frac{n^{1-\frac{\bar{\bar{s}}_{i}}{2}}}{\kappa \sigma t}\Big)^\frac{2n-2}{\bar{\bar{\lambda}}_i}$.
\item[$\cdot$] If ${\bf m}_{2n,i}(t)\geq \Big(\frac{n^{1-\frac{\bar{\bar{s}}_{i}}{2}}}{\kappa\sigma t}\Big)^{\frac{2n-2}{\bar{\bar{\lambda}}_i}}$, then $\frac{n\,{\bf m}_{2n,i}(t)}{\sigma t}\leq \kappa n^\frac{\bar{\bar{s}}_{i}}{2} {\bf m}_{2n,i}(t)^{1+\frac{\bar{\bar{\lambda}}_i}{2n-2}}$.
 \end{itemize}
 Thus, overall we get that
 $$\frac{n\,{\bf m}_{2n,i}(t)}{\sigma t}\leq \kappa\, n^{\bar{\bar{s}}_{i}/2}{\bf m}_{2n,i}(t)^{1+\frac{\bar{\bar{\lambda}}_i}{2n-2}}+\frac{n}{\sigma t}\Big(\frac{n^{1-\bar{\bar{s}}_{i}/2}}{\kappa \sigma t}\Big)^{\frac{2n-2}{\bar{\bar{\lambda}}_i}}.$$
Interpolation gives that
$$
{\bf m}_{2n,i}(t)^{1+\frac{\bar{\bar{\lambda}}_i}{2n-2}}\leq {\bf m}_{2n+\bar{\bar{\lambda}}_i,i}(t)\,{\bf m}_{2,i}^{\frac{\bar{\bar{\lambda}}_i}{2n-2}}(t)\leq  {\bf m}_{2n+\bar{\bar{\lambda}}_i,i}(t)\,\max\{1,{\bf m}_{2}\}\,,
$$
and consequently,
\begin{align*}
H^i_{p}[f_i](t)\leq \kappa\,{\bf m}_{2,i}(t)(\sigma t)^{2/\lambda^\natural}+\kappa\max\{1,{\bf m}_{2}\}\,F^{i}_{p}[f_i](t) + (\sigma t)^{2/\bar{\bar{\lambda}}_i-1}\sum_{n=1}^p \frac{n^{\frac{(2-\bar{\bar{s}}_{i})n}{\bar{\bar{\lambda}}_i}-\frac{(2-\bar{\bar{s}}_{i})}{\bar{\bar{\lambda}}_i}+1}}{(n!)^{\alpha} \kappa^{(2n-2)/\bar{\bar{\lambda}}_i}}\,. 
\end{align*}
Conservation of mass and energy implies that
$$
{\bf m}_{2,i}(t)(\sigma t)^{2/\lambda^\natural}\leq {\bf m}_{2}\, \sigma^{2/{\lambda^\natural}},\qquad t\in (0,1]\,.$$
Moreover, recall that $\alpha =\max \{1, \frac{2-s^\natural}{\lambda^\natural}\}\geq \frac{2-\bar{\bar{s}}_{i}}{\bar{\bar{\lambda}}_i}$ and $\kappa\geq1$, then
$$(\sigma t)^{2/\bar{\bar{\lambda}}_i-1}\sum_{n=1}^\infty \frac{n^{\frac{(2-\bar{\bar{s}}_{i})n}{\bar{\bar{\lambda}}_i}-\frac{2-\bar{\bar{s}}_{i}}{\bar{\bar{\lambda}}_i}+1}}{(n!)^{\alpha} \kappa^{(2n-2)/\bar{\bar{\lambda}}_i}}\leq \kappa^{2/\lambda^\natural}\sum_{n=1}^\infty \frac{n^{\alpha n+1}}{(n!)^{\alpha} \kappa^n}=:S\,,$$
where the series $S$ is convergent provided $\kappa>2e^{\alpha}$ \big(invoking the Stirling's formula $n!\sim \sqrt{2\pi n}(n/e)^n$\big). Therefore, for all $1\leq i\leq I$,
 $$H^i_{p}[f_i](t)\leq 3e^{\alpha}\max\{1,{\bf m}_{2}\}\,F^{i}_{p}[f_i](t)+3e^{\alpha}{\bf m}_{2}  + S,\qquad \sigma\in(0,1]\,,\quad t\in(0,1]\,,$$
which is the statement of the lemma for $K=3e^{\alpha}\max\{1,{\bf m}_{2}\}$ and $L=3e^{\alpha}{\bf m}_{2}  + S$.
\end{proof}
We can finally proceed with the proof of exponential moment generation.
\begin{proof}[Proof of the generation result of Theorem \ref{theo1}] We devise the proof in two steps.

\medskip
\noindent
{\bf Step 1.} We show first that for $\lambda^\natural>0$, $s^\natural>0$ and $\alpha=\max\big\{1,\frac{2-{s}^\natural}{\lambda^\natural}\big\}$, there exist $\sigma\in(0,1]$ and $T\in(0,1]$, depending only on the parameters of the model such that 
$$\sup_{t\in[0,T]} \sum_{n=0}^\infty \frac{(\sigma t)^{2n/\lambda^\natural}{\bf m}_{2n}(t)}{(n!)^\alpha}\leq 2\,{\bf m}_0\,.$$ 
Invoking Theorem \ref{Mom} it holds that for some constant $C_{p}\in(0,\infty)$,
$${\bf m}_{0}\leq E_{p}[\F](t)\leq {\bf m}_{0} +C_{p}\, t^{2/\lambda^\natural}\,,$$
then $\lim_{t\rightarrow 0} E_{p}[\F](t)={\bf m}_{0}$.

\medskip
\noindent
Now, using identity \eqref{add-moments} and lemmas \ref{Gen1} and \ref{Gen2}, we get that 
\begin{align*}
\frac{{\rm d}}{{\rm d}t} E_{p}&[\F](t)\leq \sum_{i,j=1}^I\Big[-C_{ij}^1\Big(F^{ij}_{p}[f_i](t)+F^{ji}_p[f_j](t)\Big)+C_{ij}^2G^{ij}_{p}[f_i, f_j](t)\Big]+\frac{2\sigma}{\lambda^\natural}\sum_{i=1}^I  H^{i}_{p}(t)\\
&\leq \sum_{i,j=1}^I\Big[-C_{ij}^1\Big(F^{ij}_{p}[f_i](t)+F^{ji}_p[f_j](t)\Big)+C_{ij}^2\Big( B_{\epsilon}\sigma^{2/\lambda^\natural}\big[ F^{ij}_p[f_i](t)+F^{ji}_{p}[f_j](t)\big]\\
&\qquad+ {\epsilon}\, \big[F^{ij}_p[f_i](t)E^{j}_{p}[f_j](t)+F^{ji}_{p}[f_j](t)E^{i}_{p}[f_i](t)\big]+D_{\epsilon}\Big)\Big]+\frac{2\sigma}{\lambda^\natural}\sum_{i=1}^I\Big(K F^{i}_{p}[f_i](t)+ L\Big)\\
&\leq \sum_{i,j=1}^I \Big[-C_{ij}^1\big[F^{ij}_{p}[f_i](t)+F^{ji}_p(t)[f_j](t)\big]+ R^{ij}_{{\epsilon}}\sigma^{2/\lambda^\natural}\big[ F^{ij}_p[f_i](t)+F^{ji}_{p}[f_j](t)\big]\\
&\qquad + \epsilon\, C^2_{ij} \Big(F^{ij}_p[f_i](t)+F^{ji}_{p}[f_j](t)\Big)\Big(E^{i}_{p}[f_i](t)+E^{j}_{p}[f_j](t)\Big)+\frac{2\sigma}{\lambda^\natural}K\,\sum_{i=1}^I F^{i}_{p}[f_i](t)+Q_{\epsilon}\\
&= I^1+I^2+Q_{\epsilon}\,,
\end{align*}
where $R^{ij}_{\epsilon}=C^2_{ij}\,B_{\epsilon}$, $Q_{\epsilon}=D_{\epsilon}\sum_{i,j=1}^IC^2_{ij}+\frac{2\sigma}{\lambda^\natural}IL$, and
\begin{align*}
I^1&:=\sum_{i,j=1}^I \Big[-C_{ij}^1\big[F^{ij}_{p}[f_i](t)+F^{ji}_p(t)[f_j](t)\big]+ R^{ij}_{{\epsilon}}\sigma^{2/\lambda^\natural}\big[ F^{ij}_p[f_i](t)+F^{ji}_{p}[f_j](t)\big]\\
&\hspace{4.5cm} + \epsilon\, C^2_{ij}\, \Big(F^{ij}_p[f_i](t)+F^{ji}_{p}[f_j](t)\Big)\Big(E^{i}_{p}[f_i](t)+E^{j}_{p}[f_j](t)\Big)\Big]\,,\\
I^2&:=\frac{2\sigma}{\lambda^\natural}K\sum_{i=1}^I F^{i}_{p}[f_i](t)\,.
\end{align*}
Choosing $$\epsilon:=\min_{1\leq i,j\leq I}\frac{C^1_{ij}}{8IC_{ij}^2{\bf m}_{0}}\,,\quad\text{and}\quad \sigma_0 := \min_{1\leq i,j\leq I}\Big(\frac{C^1_{ij}}{4R^{ij}_{\epsilon}}\Big)^{\lambda^\natural/2}\,,$$
we get for any $\sigma\leq\min\{1,\sigma_{0}\}$ 
\begin{align*}
I^1 &\leq -\frac{3\min_{1\leq i,j\leq I}C_{ij}^1}{4}\sum_{i,j=1}^I\Big(F^{ij}_p[f_i](t)+F^{ji}_{p}[f_j](t)\Big)\\
&\qquad +\frac{\min_{1\leq i,j\leq I}C_{ij}^1}{8I{\bf m}_{0}}\sum_{i,j=1}^I\Big(F^{ij}_p[f_i](t)+F^{ji}_{p}[f_j](t)\Big)\Big(E^{i}_{p}[f_i](t)+E^{j}_{p}[f_j](t)\Big).
\end{align*}
Similar, for $\sigma\leq \sigma_{1}:=\min_{1\leq i,j\leq I}\frac{C_{ij}^1\lambda^\natural}{8K}$, we obtain that
$$
I^2\leq \frac{\min_{1\leq i,j\leq I}C^1_{ij}}{4}\sum_{i=1}^IF^{i}_{p}[f_i](t)\leq \frac{\min_{1\leq i,j\leq I}C^1_{ij}}{4}\sum_{i,j=1}^IF^{ij}_{p}[f_i](t)\,,
$$
where the latter inequality is a direct consequence of the definitions of $F^{i}_{p}[f_i]$ and $F^{ij}_{p}[f_i]$.   Thus, overall we deduce that
\begin{align*}
\frac{{\rm d}}{{\rm d}t} & E_{p}[\F](t)\leq I^1+I^2+Q \leq -\frac{\min_{1\leq i,j\leq I}C_{ij}^1}{2}\sum_{i,j=1}^I\Big(F^{ij}_p[f_i](t)+F^{ji}_{p}[f_j](t)\Big)\\
&\qquad +\frac{\min_{1\leq i,j\leq I}C_{ij}^1}{8I{\bf m}_{0}}\sum_{i,j=1}^I\Big(F^{ij}_p[f_i](t)+F^{ji}_{p}[f_j](t)\Big)\Big(E^{i}_{p}[f_i](t)+E^{j}_{p}[f_j](t)\Big)+Q.
\end{align*}
Now, let us introduce the time
$$
T_p:=\sup\big\{0< t\leq 1\, \big| \,  E_p[\F](s)\leq 2\,{\bf m}_{0},\,  s\in (0,t]\big\}>0\,.
$$
This time is indeed positive by continuity of the polynomial moments (we pointed out previously that $\lim_{t\rightarrow 0} E_{p}[\F](t)={\bf m}_{0}$).  Thus, for all $t\in (0,T_p]$, we have that
\begin{align*}
\frac{{\rm d}}{{\rm d}t} E_{p}[\F](t) &\leq  -\frac{\min_{1\leq i,j\leq I}C_{ij}^1}{2}\sum_{i,j=1}^I\Big(F^{ij}_p[f_i](t)+F^{ji}_{p}[f_j](t)\Big)\\
&\hspace{-1cm}+\frac{\min_{1\leq i,j\leq I}C_{ij}^1}{8I{\bf m}_{0}}\sum_{i,j=1}^I\Big(F^{ij}_p[f_i](t)+F^{ji}_{p}[f_j](t)\Big)\sum_{i,j=1}^I\Big(E^{i}_{p}[f_i](t)+E^{j}_{p}[f_j](t)\Big)+Q\\
&\leq   -\frac{\min_{1\leq i,j\leq I}C_{ij}^1}{2}\sum_{i,j=1}^I\Big(F^{ij}_p[f_i](t)+F^{ji}_{p}[f_j](t)\Big) \\
&\hspace{+2cm}+\frac{\min_{1\leq i,j\leq I}C_{ij}^1}{2}\sum_{i,j=1}^I\Big(F^{ij}_p[f_i](t)+F^{ji}_{p}[f_j](t)\Big)+Q.\\
\end{align*}
Then, we comclude that for all $t\in (0,T_p]$ it holds that
$$\frac{{\rm d}}{{\rm d}t} E_{p}[\F](t)\leq Q\,,\qquad\text{and, thus}\qquad E_{p}[\F](t)\leq E_{p}[\F](0) + Q \,t.$$
Recalling that $E_p[\F](0)= {\bf m}_{0},$  we deduce directly from the definition of $T_p$ that
$$T_p\geq \min\{1, \frac{{\bf m}_{0}}{ Q}\}=:T\,,\quad\text{independent of}\;\; p\,.$$
Thus, for $\sigma \leq \min\{1,\sigma_0,\sigma_1\}$,
$$E[\F](t) = \lim_{p\rightarrow\infty}E_p[\F](t)\leq 2\,{\bf m}_0,\qquad t\in(0,T]\,.$$

\noindent
{\bf Step 2.} Let us proceed with the concluding argument.  To this end, fix the rate
$$
\rho=\min \Big\{\frac{2\lambda^\natural}{2-{s}^\natural},2\Big\}\,,
$$
and observe that $\alpha=\max\big\{1,\frac{2-{s}^\natural}{\lambda^\natural}\big\}=\frac{2}{\rho}$. 
Using the result of Step 1, we obtain that
$$
\frac{(\sigma t)^{2n/\lambda^\natural}\, {\bf m}_{2n}(\F)(t)}{(n!)^\alpha} \leq E[\F](t)=\sum_{n=0}^\infty \frac{(\sigma t)^{2n/\lambda^\natural}\, {\bf m}_{2n}(\F)(t)}{(n!)^\alpha}\leq 2{\bf m}_0\,,\quad t\in(0,T]\,.$$
Then, Lemma \ref{lemexp1} item (i) with $\sigma_0=(\sigma t)^{2/\lambda^\natural}$ and $K=2{\bf m}_0$ implies that
\begin{align*}
\sum_{i=1}^I\int_{\R^3}&\exp\Big[\frac{(\sigma t)^{\rho/\lambda^\natural}\langle v\rangle_i^\rho}{2}\Big]f_i(v)dv\\
&=\sum_{i=1}^I\int_{\R^3}\exp\Big[\frac{(\sigma t)^{2/{(\lambda^\natural\alpha)}}\langle v\rangle_i^{2/\alpha}}{2}\Big]f_i(v)dv \leq 2^{1+1/\alpha}{\bf m}_0 \leq 4{\bf m}_0,\qquad t\in(0,T]\,,
\end{align*}
 which is precisely the statement of the exponential moment generation.
 \end{proof}
\subsubsection{Propagation of exponential moment.} The argument to propagate exponential moments follows a similar path to that of propagation; we present the main steps for completeness.  Let us introduce for $\sigma\in (0,1]$, to be chosen sufficiently small, and $p\geq 2$,
$$
\mathcal{E}_p[\F](t):=\sum_{n=0}^p \frac{\sigma^{2n/{\lambda^\natural}}{\bf m}_{2n}[\F](t)}{ (n!)^{\alpha}},\qquad t\geq0\,.
$$
Analogously, for all $1\leq i,j\leq I$, we denote  
$$\mathcal{E}_{p}^{i}[f_i](t):=\sum_{n=0}^p \frac{\sigma^{2n/{\lambda^\natural}}{\bf m}_{2n,i}[f_i](t)}{ (n!)^{\alpha}},\qquad\mathcal{P}_{p}^{ij}[f_i](t):=\sum_{n=2}^p\frac{n^{s_{ij}/2}\sigma^{2n/\lambda^\natural}{\bf m}_{2n+\lambda_{ij},i}[f_i]}{(n!)^{\alpha}},$$
$$\text{and}\qquad \mathcal{G}_{p}^{ij}[f_i,f_j](t):=\sum_{n=2}^p \frac{\sigma^{2n/\lambda^\natural}\mathcal{S}_{n}^{ij}(t)}{(n!)^{\alpha}},\qquad t\geq0\,.$$
Note that the analogous to \eqref{add-moments} in this context is given by
\begin{align}\label{add-moments-prop}
\frac{{\rm d}}{{\rm d}t}\mathcal{E}_p[\F](t)\leq\sum_{i,j=1}^I\Big[ -C^1_{ij}\Big(\mathcal{P}_{p}^{ij}[f_i](t)+\mathcal{P}_{p}^{ji}[f_j](t)\Big)+C^2_{ij}\mathcal{G} _{p}^{ij}[f_i,f_j](t)\Big]. 
\end{align}
As before, we have the key estimate on the convolution structure of moments.
\begin{lem}[Moment Convolution]\label{Prop1}
Let $1\leq i,j\leq I$, $\lambda^\natural>0$, $s_{ij}\in(0,2)$, $\F_0\in U(D_0,E_0)$, and
$$\tilde{\sigma}_0=\sup_{n\geq 0} \frac{\sigma_{0}^n\, {\bf m}_{2n}(0)}{(n!)^\alpha}\leq 1,\qquad \text{for some}\;\; \sigma_0>0\,.$$
Given ${\epsilon}>0$ there are explicit constants $\mathcal{B}_{\epsilon}$ and $\mathcal{D}_{\epsilon}$, depending on the parameters of the model, $E_0$, $D_0$, and $\tilde{\sigma}_0$, such that for any $\sigma\in(0,1]$, $p\geq 2$, and $\alpha\geq 1$, the following estimate holds
\begin{align*}
\mathcal{G}^{ij}_{p}[f_i, f_j](t) & \leq \mathcal{B}_{\epsilon}\,\sigma^{2/\lambda^\natural}\Big[ \mathcal{P}^{ij}_p[f_i](t)+\mathcal{P}^{ji}_p[f_j](t)\Big]\\
&\quad + \epsilon\, \Big[\mathcal{P}^{ij}_p[f_i](t)\mathcal{E}^{j}_p[f_j](t)+\mathcal{P}^{ji}_p[f_j](t)\mathcal{E}^{i}_p[f_i](t)\Big]+\sigma^{2/\lambda^\natural}\mathcal{D}_{\epsilon},\qquad t\geq0\,.
\end{align*}
\end{lem}
\begin{proof}
Using a copycat argument in the proof of Lemma \ref{Gen1}, we have that
\begin{align*}
\mathcal{G}_{p}^{ij}[f_i,f_j](t)&\leq \Big(\mathcal{P}^{ji}_{p}[f_j](t)+\sigma^{2/\lambda^\natural}{\bf m}_{2+\lambda_{ij},j}(t)\Big)\Big(\epsilon\, \mathcal{E}^{i}_{p}[f_i](t)+ \mathcal{J}_{\epsilon}^{i}[f_i](t)\Big)\\
&\qquad+\Big(\mathcal{P}^{ij}_{p}[f_i](t)+\sigma^{2/\lambda^\natural}{\bf m}_{2+\lambda_{ij},i}(t)\Big)\Big(\epsilon \,\mathcal{E}^{j}_{p}[f_j](t)+ \mathcal{J}_{{\epsilon}}^{j}[f_j](t)\Big), 
\end{align*}
where,
setting $\mathcal{N}_{\epsilon}=\lceil 2\,{\epsilon}^{-1}\rceil,$ 
$$\mathcal{J}_{\epsilon}^{i}[f_i](t):=2\sum_{a=1}^{\mathcal{N}_{\epsilon}}\frac{\sigma^{2a/\lambda^\natural}{\bf m}_{2a,i}[f_i](t)}{(a!)^{\alpha}}\leq \mathcal{A}_{\epsilon}\,\sigma^{2/\lambda^\natural}\,.$$
For the later inequality in the right we have used the propagation of polynomial moments given in Theorem \ref{Mom}, which applies since the initial datum satisfies ${\bf m}_{2n}(0)\leq \frac{(n!)^\alpha}{\sigma^n_0}$, for $n\geq1$.  Then, we can simply take the constant
$$
\mathcal{A}_{\epsilon} := \mathcal{A}_{\epsilon}(\tilde{\sigma}_0) = 2\,\sum_{a=1}^{\infty}\frac{\sigma^{2a/\lambda^\natural}}{(a!)^{\alpha}}\,\sup_{t\geq0}{\bf m}_{2\mathcal{N}_\epsilon}[\F](t)<\infty\,.
$$
Thus,
\begin{align*}
\mathcal{G}_{p}^{ij}[f_i,f_j](t)&\leq \Big(\mathcal{P}^{ji}_{p}[f_j](t)+2\,C\,\sigma^{2/\lambda^\natural}\Big)\Big( \epsilon\, \mathcal{E}_{p}^{i}[f_i](t)+\mathcal{A}_{\epsilon}\,\sigma^{2/\lambda^\natural}\Big) \\
&\qquad+\Big(\mathcal{P}^{ij}_{p}[f_i](t)+2\,C\,\sigma^{2/\lambda^\natural}\Big)\Big( \epsilon\, \mathcal{E}^{j}_{p}[f_j](t)+\mathcal{A}_{\epsilon}\,\sigma^{2/\lambda^\natural}\Big)\,.
\end{align*}
And, since
\begin{align*}
\mathcal{E}^{i}_{p}[f_i](t)+\mathcal{E}^{j}_p[f_j](t)&\leq {\bf m}_{0}+\sigma{\bf m}_2+\mathcal{P}^{ij}_{p}[f_i](t)+\mathcal{P}^{ji}_p[f_j](t)\,,
\end{align*}
we conclude that
\begin{align*}
\mathcal{G}^{ij}_{p}[f_i, f_j](t)&\leq \mathcal{B}_{\epsilon}\sigma^{2/\lambda^\natural}\Big[ \mathcal{P}^{ij}_p[f_i](t)+\mathcal{P}^{ji}_p[f_j](t)\Big]\\
&\qquad + \epsilon \Big[\mathcal{P}^{ij}_p[f_i](t)\mathcal{E}^{j}_p[f_j](t)+\mathcal{P}^{ji}_{p}[f_j](t)\mathcal{E}^{i}_p[f_i](t)\Big]+\sigma^{2/\lambda^\natural}\mathcal{D}_{\epsilon},
\end{align*}
where 
$$\mathcal{B}_{\epsilon}= 2 \mathcal{A}_{\epsilon}+2\,\epsilon\,C \qquad\text{and}\qquad \mathcal{D}_{\epsilon}=4\sigma^{2/\lambda^\natural}\mathcal{A}_{\epsilon}C+2\,\epsilon\,C\,\big({\bf m}_{0}+\sigma {\bf m}_2\big)\,,$$
which is the statement of the result.
\end{proof}
 \begin{lem}\label{Prop2} 
Let $1\leq i,j\leq I$, $\lambda^\natural>0$, $s_{ij}\in(0,2)$ and $p>2$. For any  $\sigma\in(0,1]$, $\alpha\geq 1$ it holds that
 $$\mathcal{P}_{p}^{ij}[f_i](t)+\mathcal{P}_p^{ji}[f_j](t)\geq \frac{1}{\sigma^{\lambda^\natural/2}}\Big[\mathcal{E}^{i}_p[f_i](t)+\mathcal{E}^{j}_p[f_j](t) - 2 \, e \max\big\{{\bf m}_0[f_i], {\bf m}_0[f_j]\big\} \Big],\qquad t\geq0\,.$$
\end{lem}
\begin{proof}
Using that $x^{b}\geq x^a-1$ for any $b\geq a\geq0$ we can write 
\begin{align*}
\mathcal{P}_{p}^{ij}[f_i](t)+\mathcal{P}_p^{ji}[f_j](t)&=\sum_{n=2}^p \frac{n^{s_{ij}/2}\sigma^n }{(n!)^\alpha}\Big({\bf m}_{2n+\lambda_{ij},i}[f_i](t)+{\bf m}_{2n+\lambda_{ij},j}[f_j](t)\Big)\\
&\geq \frac{1}{\sigma^{\lambda^\natural/2}}\sum_{n=2}^p \bigg[\int_{\R^3}\frac{(\sigma \langle v\rangle_i^2)^{n+{\lambda}^\natural/2} f_i(t,v)dv}{(n!)^{\alpha}}+\int_{\R^3}\frac{(\sigma \langle v\rangle_j^2)^{n+{\lambda}^\natural/2} f_j(t,v)dv}{(n!)^{\alpha}}\bigg]\\
&\geq \frac{1}{\sigma^{{\lambda}^\natural/2}}\sum_{n=2}^p\bigg[\int_{\R^3}\frac{\big[(\sigma \langle v\rangle_i^2)^{n}-1\big] f_i(t,v)dv}{(n!)^{\alpha}}+\int_{\R^3}\frac{\big[(\sigma \langle v\rangle_j^2)^{n}-1\big] f_j(t,v)dv}{(n!)^{\alpha}}\bigg] \\
 &\geq  \frac{1}{\sigma^{{\lambda}^\natural/2}}\sum_{n=2}^p\frac{\sigma^{n}}{(n!)^{\alpha}} \Big({\bf m}_{2n}[f_i](t)-{\bf m}_0[f_i]+{\bf m}_{2n}[f_j](t)-{\bf m}_0[f_j]\Big)\,.
 \end{align*}
 Then, since $\sigma\in(0,1]$ and ${\alpha}\geq 1$, it holds that 
\begin{align*} 
\mathcal{P}_{p}^{ij}[f_i](t)+\mathcal{P}_{p}^{ji}[f_j](t)&\geq \frac{1}{\sigma^{\lambda^\natural/2}}\bigg[\mathcal{E}^{i}_p[f_i](t)+\mathcal{E}_{p}^{j}[f_j](t)-{\bf m}_0[f_i](t)-\sigma{\bf m}_2[f_i](t)\\
 &\quad -{\bf m}_0[f_j](t)-\sigma{\bf m}_2[f_j](t)-{\bf m}_0[f_i](t)\sum_{n=2}^p \frac{1}{(n!)^{\alpha}}-{\bf m}_0[f_j](t)\sum_{n=2}^p \frac{1}{(n!)^{\alpha}}\bigg]\\
&\geq \frac{1}{\sigma^{\lambda^\natural/2}}\Big[\mathcal{E}^{i}_p[f_i](t)+\mathcal{E}^{j}_p[f_j](t) - 2 \, e \max\big\{{\bf m}_0[f_i],{\bf m}_0[f_j]\big\}\Big].
\end{align*}
\end{proof}
Let us proceed with the proof the propagation of exponential moments.
\begin{proof}[Proof of the propagation result of Theorem \ref{theo1}]  As before, we present two steps.

\medskip
\noindent
{\bf Step 1.} First we prove that for any $\sigma_{0}>0$ and $\alpha\geq 1$ there exists $\sigma>0$ (sufficiently small) depending only on the parameters of the model, $D_0$, $E_0$, $\sigma_0$ and $\alpha$ such that  
$$\text{if}\quad\sup_{n\geq 0} \frac{\sigma_{0}^n\, {\bf m}_{2n}(0)}{(n!)^\alpha}\leq 1\;\Longrightarrow\; \sup_{t\geq 0} \sum_{n=0}^\infty\frac{\sigma^{2n/\lambda^\natural} \,{\bf m}_{2n}(t)}{(n!)^\alpha}\leq 3\,I\,e\,\big( {\bf m}_0+1 \big).$$
We fix $\alpha\geq 1$, $\sigma_{0}>0$, and assume that $\sup_{n\geq 0}\frac{\sigma_{0}^n \,{\bf m}_{2n}(0)}{(n!)^{\alpha}}\leq 1$.
Note that if $\sigma\in (0, \frac{\sigma_{0}}{2}]$ then $\mathcal{E}_p[\F](0)\leq 1+{\bf m}_{0}$, for $p\geq 2$. Indeed, 
$$
\mathcal{E}_{p}[\F](0)\leq {\bf m}_{0} + \sum_{n=1}^p\frac{(\sigma_0/2)^n\, {\bf m}_{2n}(0)}{(n!)^{\alpha}}\leq {\bf m}_{0}+\sum_{n\geq 1}2^{-n}={\bf m}_{0}+1\,.
$$
Using identity \eqref{add-moments-prop} and lemmas \ref{Prop1} and \ref{Prop2}, for any $\sigma\in(0,1]$ and $p\geq2$, it follows that
\begin{align*}
\frac{{\rm d}}{{\rm d}t}\mathcal{E}_p[\F](t)\leq \sum_{i,j=1}^I&\bigg[-C_{ij}^1 \Big(\mathcal{P}_{p}^{ij}[f_i](t)+\mathcal{P}_{p}^{ji}[f_j](t)\Big)+C^2_{ij}\,\mathcal{B}_{\epsilon}\,\sigma^{2/\lambda^\natural}\,\Big[ \mathcal{P}^{ij}_p[f_i](t)+\mathcal{P}^{ji}_p[f_j](t)\Big]\\
&\quad + \epsilon \, C^2_{ij}\,\Big[\mathcal{E}^{i}_p[f_i](t)\mathcal{P}^{ji}_p[f_j](t)+\mathcal{E}^{j}_p[f_j](t)\mathcal{P}^{ij}_p[f_i](t)\Big]+\mathcal{D}_{\epsilon}\bigg]\,,\qquad t>0\,.
\end{align*}
Choosing
$$
\epsilon=\min_{1\leq i,j\leq I}\frac{C_{ij}^1}{12\,I\,e\,C^2_{ij}\,({\bf m}_{0}+1)} \qquad\text{and} \qquad \sigma_{1} = \min_{1\leq i,j\leq I}\Big[\frac{C^1_{ij}}{2 \,C^2_{ij} \,\mathcal{B}_{\epsilon}}\Big]^{\lambda^\natural/2},$$
we conclude that for any choice of $\sigma\in(0,\sigma_{1}]$ 
\begin{align*}
\frac{{\rm d}}{{\rm d}t}\mathcal{E}_p[\F](t) &\leq -\frac{\min_{1\leq i,j\leq I}C_{ij}^1}{2} \sum_{i,j=1}^I\Big(\mathcal{P}_{p}^{ij}[f_i](t)+\mathcal{P}_{p}^{ji}[f_j](t)\Big)\\
&\quad +\frac{\min_{1\leq i,j\leq I}C_{ij}^1}{12\,I\,e\,({\bf m}_{0}+1)}\sum_{i,j=1}^I\Big(\mathcal{P}^{ji}_p[f_j](t)+\mathcal{P}^{ij}_p[f_i](t)\Big)\mathcal{E}_p[\F](t)+I^{2}\mathcal{D},\qquad t>0\,.
\end{align*}
Recall that for $\sigma\in(0,\sigma_{0}/2]$ we know that  $\mathcal{E}_p[\F](0)\leq {\bf m}_{0}+1$.  Then, by continuity of the polynomial moments, the time
$$
T_p :=\sup \big\{ t>0\,\big|\, \mathcal{E}_p[\F](s)\leq 3\,I\,e\,({\bf m}_{0}+1)\,,\, s\in(0, t]\big \}
$$
is positive.  Thus, invoking Lemma \ref{Prop2}, and provided $\sigma\leq  \min\{ \frac{\sigma_{0}}{2},\, \sigma_{1}\}$, it holds that  
\begin{align*}
&\frac{{\rm d}}{{\rm d}t}{\mathcal{E}}_p[\F](t)\leq -\frac{\min_{1\leq i,j\leq I}C_{ij}^1}{4}\sum_{i,j=1}^I\Big(\mathcal{P}_{p}^{ij}[f_i](t)+\mathcal{P}_{p}^{ji}[f_j](t)\Big)+
I^{2}\mathcal{D}\\
&\leq -\frac{\min_{1\leq i,j\leq I}C_{ij}^1}{ 4\sigma^{\lambda^\natural/2}}\sum_{i,j=1}^I\Big[\mathcal{E}^{i}_p[f_i](t)+\mathcal{E}^{j}_p[f_j](t) -2\,e\,\max\big\{{\bf m}_0[f_i],{\bf m}_0[f_j]\big\} \Big]+I^{2}\mathcal{D}\,, \quad t\in(0,T_p]\,,
\end{align*}
and consequently,
\begin{align*}
\frac{{\rm d}}{{\rm d}t}\mathcal{E}_p[\F](t) \leq -\frac{\min_{1\leq i,j\leq I}C_{ij}^1}{ 2\sigma^{\lambda^\natural/2}}\Big[\mathcal{E}_p[\F](t)-2\,I\, e\,{\bf m}_0 \Big]+ I^{2}\mathcal{D}\,,\qquad t\in(0,T_p].
\end{align*}
Then,  
$$
\mathcal{E}_p[\F](t)\leq 2\,I \,e\,{\bf m}_0 +\frac{2 \sigma^{\lambda^\natural/2}\,I^2\,\mathcal{D} }{\min_{1\leq i,j\leq I}C^1_{ij}}\,,\qquad t\in [0,T_p]\,.
$$
Choosing $\sigma\leq \min\{ \frac{\sigma_{0}}{2},\, \sigma_{1}\}$ sufficiently small so that
 $$ \frac{2 \sigma^{\lambda^\natural/2}\,I^{2}\mathcal{D} }{\min_{1\leq i,j\leq I}C^1_{ij}}\leq I \,e\,{\bf m}_0\,,\qquad\text{that is}\qquad \sigma\leq \bigg(\frac{e\,{\bf m}_0\,\min_{1\leq i,j\leq I}C^1_{ij}}{2\,I\,\mathcal{D}}\bigg)^{2/\lambda^\natural}=:\sigma_2\,,$$ 
we conclude that, for all $p\geq2$, it holds that
$$\mathcal{E}_{p}[\F](t)\leq 3\,I\,e\,{\bf m}_{0}\, \qquad t\in [0,T_p]\,.$$
Then, by continuity of the polynomial moments and the definition of $T_{p}$, we must have  $T_p=\infty$ for all $p\geq2$.  We conclude sending $p\rightarrow\infty$ that, for $\sigma\leq\min\{1,\sigma_0/2,\sigma_1,\sigma_2\}$,
$$
\sum_{n=0}^\infty \frac{\sigma^{2n/{\lambda^\natural}}{\bf m}_{2n}[\F](t)}{ (n!)^{\alpha}}\leq 3\,I\,e\,\big( {\bf m}_{0}+1\big)\,, \qquad t\geq 0\,.
$$
{\bf Step 2.} Let us conclude.  Fix $\rho\in(0,2]$, $\sigma_{0}\in(0,1]$, and $A>1$.   Assume that  $$\sum^{I}_{i=1}\displaystyle\int_{\R^3}\exp\big(\sigma_{0} \langle v\rangle_i^{\rho}\big)\,f_i(0,v)dv\leq A\,,$$
and set $\alpha=\frac{2}{\rho}\geq 1$.  Using Lemma \ref{lemexp1} item (ii) it follows that for some $\tilde\sigma:=\tilde\sigma(\rho, \sigma_{0}, A)$
$$
\sup_{n\geq 0}\frac{\tilde\sigma^n {\bf m}_{2n,i}  (0) }{(n!)^{2/\rho}}\leq \frac{1}{I}\,,\qquad\text{and then},\qquad \sup_{n\geq 0}\frac{\tilde\sigma^n {\bf m}_{2n}  (0) }{(n!)^{2/\rho}}\leq 1\,.$$
Thus, we apply Step 1 to conclude that there exists $\sigma\in(0,\tilde\sigma/2]$ (computed is Step 1) such that
$$
\sup_{n\geq0}\frac{\sigma^{2n/\lambda^\natural} {\bf m}_{2n}(t)}{(n!)^{2/\rho}}\leq\sum_{n=0}^\infty\frac{\sigma^{2n/\lambda^\natural} {\bf m}_{2n}(t)}{(n!)^{2/\rho}}\leq 3\,I\,e\,\big( {\bf m}_{0}+1\big)\,,\qquad t\geq0\,.
$$
We deduce from Lemma \ref{lemexp1} item (i) (applied with $\sigma_0= \sigma^{2/\lambda^\natural}$) that, for $t\geq0$, 
\begin{align*}
\sum_{i=1}^I\int_{\R^3}\exp\Big( \frac{\sigma^{\rho/\lambda^\natural} \langle v\rangle_i^\rho}{2}\Big)f_i(t,v)dv&\leq 2\,{\bf m}_0^{1-\rho/2}\Big(3\,I\,e\,({\bf m}_0+1)\Big)^{\rho/2} \leq 6\,I\,e\,({\bf m}_0+1)\,,
\end{align*}
which is the statement of the theorem.
\end{proof}
\section{Lebesgue integrability generation and propagation}
After some technical generalization of a coercivity estimate for the collision operator originated in \cite{ADVW} for Maxwell Molecules and refined in \cite{AMUXY} for other potentials, it is possible to implement an energy estimate method, see \cite{DM,RA}, to prove generation and propagation of higher Lebesgue integrability.
\subsection{A uniform coercive estimate}
Recall that the dynamic of the gas mixture satisfies $\F(t)\in U(D_0,E_0)$.  Let us make an important observation related to such space $U(D_0,E_0)$.  Set $B(R)= \{v \in \mathbb{R}^3 \,|\, |v| \leq R\}$ and $B_{v_0}(R,r)= \{v\in B(R) \,|\, |v-v_0|\geq
r\}$ for $R>0$, $r>0$, and $v_0\in \mathbb{R}^3$.  It follows from the compactness of $U(D_0,E_0)$ in $L^{1}$ that there
exist positive constants $\tilde R$ and $\tilde r$ depending only on $D_0$ and $E_0$ (independent of $v_0$) such that
\begin{equation}\label{U0}
\G\in U(D_0,E_0)\Rightarrow\; \chi_{B_{v_0}(\tilde R,\tilde r)}\G\in U(D_0/2,E_0)\,,\qquad \forall\; v_0\in\mathbb{R}^{3}\,,
\end{equation}
where $\chi_A$ denotes a characteristic function of the set $A\subset \mathbb{R}^3$.  Of course one can take $\tilde{R}>1>\tilde{r}$ with no loss of generality.

\begin{prop}\label{Coercive}
Fix $1\leq i\leq I$ and assume that $ s_{ij}\in (0,2)$ and $\lambda_{ij}\in(0,2]$ for all $1\leq j\leq I$.  Assume also that $D_0,\,E_0 > 0$, and $\G=[g_j]_{1\leq j \leq I} \in U(D_0, E_0)$.  Then there exist positive constants $c_{i}$, $C_{i}$ depending only on the parameters of the model, $D_0$ and $E_0$, such that
$$-\sum_{j=1}^I\big(Q_{ij}(f,g_j), f\big)_{L^2}\geq c_{i}\big\|\langle v \rangle_i^{\frac{\overline{\lambda}_{i}}{2}}f\big\|^2_{H^{\frac{\overline{\overline{s}}_{i}}{2}}}- C_{i}\|\langle \cdot \rangle^{\frac{\bar{\bar{\lambda}}_{i}}{2}}_i\,f\|^2_{L^2}\,.$$
%
%
\end{prop}
\begin{proof}
Compute with a suitable $f$
\begin{align*}
\sum_{j=1}^I\big(Q_{ij}(f,g_j), f\big)_{L^2}
&= \sum_{j=1}^I\int_{\mathbb{R}^3} Q_{ij}(f,g_j) \,f \, dv\\
&= \sum_{j=1}^I \int_{\mathbb{R}^3\times \mathbb{R}^3\times \mathbb{S}^2} g_j(v_*)f(v)\big( f(v')-f(v) \big) B_{ij}(v,v_*,\sigma)d\sigma dv_* dv\\
 &= \sum_{j=1}^I \int_{\mathbb{R}^3\times \mathbb{R}^3\times \mathbb{S}^2} g_j(v_*)f(v)\big( f(v')-f(v) \big) |v-v_*|^{\lambda_{ij}}b_{ij}(\widehat{u},\sigma)d\sigma dv_* dv\,.
 \end{align*}
Using the identity
 $$f(v) \big( f(v')-f(v) \big)=-\frac{1}{2}\big(f(v') - f(v)\big)^2+\frac{1}{2}\big(f(v')^2-f(v)^2\big)$$
we get that
 \begin{align*}
 \sum_{j=1}^I\big(Q_{ij}(f,g_j), f\big)_{L^2}&= -\frac{1}{2}\sum_{j=1}^I  \int_{\mathbb{R}^3\times \mathbb{R}^3\times \mathbb{S}^2}|v-v_*|^{\lambda_{ij}} g_j(v_*)\big(f(v') - f(v)\big)^2b_{ij}(\widehat{u},\sigma)d\sigma dv_* dv\\
 &\qquad+\frac{1}{2} \sum_{j=1}^I \int_{\mathbb{R}^3\times \mathbb{R}^3\times \mathbb{S}^2}|v-v_*|^{\lambda_{ij}} g_j(v_*)\big(f(v')^2 - f(v)^2\big)b_{ij}(\widehat{u},\sigma)d\sigma dv_* dv\\
 &= -\frac{1}{2}I_{i,1}(g_j,f)+\frac{1}{2}I_{i,2}(g_j,f)
\end{align*}
where the terms are given by
$$I_{i,1}(g_j,f):= \sum_{j=1}^I \int_{\mathbb{R}^3\times \mathbb{R}^3\times \mathbb{S}^2}|v-v_*|^{\lambda_{ij}} g_j(v_*)\big(f(v')-f(v)\big)^2b_{ij}(\widehat{u},\sigma)d\sigma dv_* dv\,,$$
and 
$$I_{i,2}(g_j,f):= \sum_{j=1}^I \int_{\mathbb{R}^3\times \mathbb{R}^3\times \mathbb{S}^2}|v-v_*|^{\lambda_{ij}} g_j(v_*)\big(f(v')^2-f(v)^2\big)b_{ij}(\widehat{u},\sigma)d\sigma dv_* dv\,.$$
{\bf Step 1.} This step is based on the scalar case given in \cite{AMUXY}.  Let us estimate $I_1(g_j,f)$ in a region of high velocities $|v_*| \leq \tilde R \ll |v|$.  Choose $\tilde R$ and $\tilde r$ such that (\ref{U0}) holds. Let $\phi_{\tilde R}$ be a non-negative smooth function not greater than one, which is $1$ for $\{ |v|\geq 4\tilde R \}$ and $0$ for $\{ |v|\leq 2\tilde R\}$. In view of 
$$\frac{\langle v\rangle_i}{4\sqrt{m_i}}\leq |v-v_*|\leq \frac{2\,\langle v\rangle_i}{\sqrt{m_i}}\,,\quad \text{in the support of}\quad\chi_{B(\tilde R)}(v_*)\,\phi_{\tilde R}(v)$$
we have that
\begin{align*}
4^{\lambda_{ij}}\,&|v-v_*|^{\lambda_{ij}}\,g_j(v_*)\,\big( f(v')-f(v) \big)^2\\
&\geq \min_{1\leq j\leq I} m_{i}^{-\frac{\lambda_{ij}}{2}} \; \big(g_j\chi_{B(\tilde R)}\big)(v_*)\big(\langle v\rangle_i^\frac{\lambda_{ij}}{2}\,\phi_{\tilde R}(v) \big)^2\, \big(f(v')-f(v)\big)^2\\
& \geq \min_{1\leq j\leq I} m_{i}^{-\frac{\lambda_{ij}}{2}} \;\big( g_j\chi_{B(\tilde R)} \big)(v_*)\bigg[\,\frac{1}{2}\,\Big( \big( \langle \cdot\rangle_i^\frac{\lambda_{ij}}{2}\phi_{\tilde R} \, f \big)(v') -  \big(\langle \cdot\rangle_i^\frac{\lambda_{ij}}{2}\phi_{\tilde R}\, f \big )(v)\Big)^2\\
&\hspace{5.5cm} - \Big( \big( \langle \cdot\rangle_i^\frac{\lambda_{ij}}{2}\phi_{\tilde R} \big)(v')- \big(\langle \cdot\rangle_i^\frac{\lambda_{ij}}{2}\phi_{\tilde R} \big)(v)\Big)^2\,f(v')^2\bigg]\,.
\end{align*}
Observe that 
\begin{align}\label{comparisons}
\begin{split}
|v' - v| &\leq 2 |v - v_*|\sin\frac{\theta}{2}\,,\qquad \text{and}\\
\frac{|v - v_*|}{\sqrt2}\leq |v' - v_*| &\leq 2|v-v_*+\tau(v' - v)| \leq 6|v-v_*|\quad \text{for}\quad \cos\theta\geq0\,,
\end{split}
\end{align}
where the latter follow after using that $u^+ \cdot u^- =0$ and $|u|^{2}=|u^+|^2 + |u^-|^{2}$ in the definition of $v'$.  Using the mean value theorem there exists $\tau \in (0,1)$ such that the inequalities hold
\begin{align*}
\Big|\big( \langle \cdot\rangle_i^\frac{\lambda_{ij}}{2}\phi_{\tilde R} \big)(v') - \big( \langle \cdot\rangle_i^\frac{\lambda_{ij}}{2}&\phi_{\tilde R} \big)(v)\Big|\lesssim \langle v+\tau (v'-v)\rangle^{\frac{\lambda_{ij}}{2}-1}_i|v-v_*|\sin\frac{\theta}{2}\\
&\lesssim \frac{1}{\sqrt{m_i}} \langle v_*\rangle^{1-\frac{\lambda_{ij}}{2}}_i\langle v'-v_*\rangle^{\frac{\lambda_{ij}}{2}}_i\sin\frac{\theta}{2}\leq \frac{1}{\sqrt{m_i}} \langle v_*\rangle_i\langle v'\rangle^{\frac{\lambda_{ij}}{2}}_i\sin\frac{\theta}{2}.
\end{align*}
In the second inequality we used the uniform equivalences in \eqref{comparisons}.  Since $$\langle v_*\rangle_i\leq \Big(\max_{1\leq j\leq I}\sqrt{\frac{ m_i}{ m_j}}\Big)\,\langle v
_*\rangle _j\,$$
it implies that
$$\Big| \big( \langle \cdot\rangle_i^\frac{\lambda_{ij}}{2}\phi_{\tilde R} \big)(v') - \big(\langle \cdot\rangle_i^\frac{\lambda_{ij}}{2}\phi_{\tilde R} \big)(v)\Big|\lesssim \Big(\max_{1\leq j\leq I}m^{-\frac12}_{j}\Big) \langle v_*\rangle_j\langle v'\rangle^{\frac{\lambda_{ij}}{2}}_i\sin\frac{\theta}{2}.$$
One concludes that 
\begin{align*}
I_{i,1}&(g_j,f) \\
&\geq\min_{1\leq j\leq I} m_{i}^{-\frac{\lambda_{ij}}{2}}\sum_{j=1}^I\int_{\mathbb{R}^3\times \mathbb{R}^3\times 
[0,\frac{\pi}{2}]}\big(g_j\chi_{B(\tilde R)}\big)(v_*)\bigg[\frac{1}{2}\Big( \big( \langle \cdot\rangle_i^\frac{\lambda_{ij}}{2}\phi_{\tilde R} f \big)(v') - \big( \langle \cdot\rangle_i^\frac{\lambda_{ij}}{2}\phi_{\tilde R} f \big)(v)\Big)^2\\
&\qquad\qquad -\bigg| \Big(\max_{1\leq j\leq I}m^{-\frac12}_j\Big)\langle v_*\rangle_j\langle v'\rangle^{\frac{\lambda_{ij}}{2}}_i\,\sin\frac{\theta}{2}\,f(v')\bigg|^2\bigg]b_{ij}(\cos\theta)d\theta dv_*dv\,.
\end{align*}
The second term in the right side, due to the second order cancellation, is controlled by the expression
\begin{equation*}
C\min_{1\leq j\leq I} m_{i}^{-\frac{\lambda_{ij}}{2}} \; \max_{1\leq j\leq I}m^{-1}_j \; \max_{1\leq j\leq I}\Big\{\|\theta^2\,b_{ij}\|_{L^1}\Big\}\,E_0\,\|\langle\cdot\rangle^{\frac{\bar{\bar{\lambda}}_{i}}{2}}_i\,f\|^2_{L^2}\,.
\end{equation*}
As for the first term, it holds, thanks to Lemma \ref{Kij}, 
\begin{align*}
\frac{1}{2}&\min_{1\leq j\leq I} m_{i}^{-\frac{\lambda_{ij}}{2}}\sum_{j=1}^I \int_{\mathbb{R}^3\times \mathbb{R}^3\times 
[0, \frac{\pi}{2}]}\big(g_j\chi_{B(\tilde R)}\big)(v_*) \\
&\hspace{3cm}\times\Big( \big(\langle \cdot\rangle_i^\frac{\lambda_{ij}}{2}\phi_{\tilde R}\, f \big)(v') - \big( \langle \cdot\rangle_i^\frac{\lambda_{ij}}{2}\phi_{\tilde R}\, f \big)(v)\Big)^2b_{ij}(\cos\theta)d\theta dv_*dv\\
&\geq c_0\min_{1\leq j\leq I} m_{i}^{-\frac{\lambda_{ij}}{2}}\min_{1\leq j\leq I}K^{ij}(s,D_0,E_0)\int_{|\xi|\geq 1}\Big|\,|\xi|^{\overline{\overline{s}}_{i}}\mathcal{F}\big(\langle v\rangle_i^\frac{\overline{\lambda}_{i}}{2}\phi_{\tilde R}\, f\, \big)(\xi)\,\Big|^2d\xi\\
&\geq c_0\min_{1\leq j\leq I} m_{i}^{-\frac{\lambda_{ij}}{2}}\min_{1\leq j\leq I}K^{ij}(s,D_0,E_0) \big\|\langle v\rangle_i^\frac{\overline{\lambda}_{i}}{2}\,\phi_{\tilde R}\,f\big\|^2_{H^{\frac{\overline{\overline{s}}_{i}}{2}}}\,,
\end{align*}
where $c_0=\frac{1}{32\pi^3}$.   Thus, overall
\begin{multline}\label{I1-1}
I_{i,1}(g_j, f)\geq c_0\min_{1\leq j\leq I} m_{i}^{-\frac{\lambda_{ij}}{2}}\min_{1\leq j\leq I}K^{ij}(s,D_0,E_0)\,\big\|\langle v\rangle_i^\frac{\overline{\lambda}_{i}}{2}\,\phi_{\tilde R}\,f\big\|^2_{H^{\frac{\overline{\overline{s}}_{i}}{2}}}\\
 - C\min_{1\leq j\leq I} m_{i}^{-\frac{\lambda_{ij}}{2}} \; \max_{1\leq j\leq I}m^{-1}_j \; \max_{1\leq j\leq I}\Big\{\|\theta^2\,b_{ij}\|_{L^1}\Big\}\,E_0\,\|\langle\cdot\rangle^{\frac{\bar{\bar{\lambda}}_{i}}{2}}_i\,f\|^2_{L^2}\,.
\end{multline}
\noindent
Now we estimate $I_1$ in a region where $v_*$ and $v$ have comparable sizes, yet, separated.  For the set $B(4\tilde{R})$ we take the finite covering 
$$B(4\tilde{R})\subset \bigcup_{v_k\in B(4\tilde R)}A_k,\qquad A_k=\Big\{v\in \mathbb{R}^3;\; |v-v_k|\leq \frac{\tilde r}{4}\Big\}.$$ 
For each $A_k$ choose a non-negative smooth function $\phi_{A_k}$ which is $1$ on $A_k$ and $0$ on $\big\{|v-v_k|\geq \frac{\tilde r}{2} \big\}$.  Note that 
$$\frac{\tilde r}{2}\leq |v-v_*|\leq 6\tilde R \quad \text{on the support of} \quad \chi_{B_{v_k}(\tilde R, \tilde r)}(v_*)\phi_{A_k}(v)\,.$$
Then, recalling that $m_i<1$ and $\tilde R\geq1$,
\begin{align*}
&|v-v_*|^{\lambda_{ij}}g_j(v_*)\big( f(v') - f(v)\big)^2\gtrsim  \tilde r^{\lambda_{ij}}\big( g_j\,\chi_{B_{v_k}(\tilde R,\tilde r)} \big)(v_*)\phi^2_{A_k}\big( f(v') - f(v) \big)^2\\
&\quad\gtrsim \bigg(\min_{1\leq j \leq I}\frac{\tilde r^{\lambda_{ij}}}{\tilde R^{\lambda_{ij}}}\bigg) \big( g_j\,\chi_{B_{v_k}(\tilde R,\tilde r)}\big)(v_*)\bigg[\frac{1}{2}\Big(\big( \langle \cdot\rangle_i^\frac{\lambda_{ij}}{2}\phi_{A_k} f \big)(v') - \big( \langle \cdot\rangle_i^\frac{\lambda_{ij}}{2}\phi_{A_k} f \big)(v)\Big)^2\\
&\hspace{5.5cm} -\Big( \big( \langle \cdot\rangle_i^\frac{\lambda_{ij}}{2}\phi_{A_k} \big)(v')- \big(\langle \cdot\rangle_i^\frac{\lambda_{ij}}{2}\phi_{A_k}\big)(v) \Big)^2f(v')^2\bigg]\,.
\end{align*}
Again, using the mean value theorem it holds for $|v_*|\leq \tilde{R}$ that
$$
\Big|\big(\langle \cdot\rangle_i^\frac{\lambda_{ij}}{2}\phi_{\tilde R} \big)(v') - \big(\langle \cdot\rangle_i^\frac{\lambda_{ij}}{2}\phi_{\tilde R}\big)(v)\Big| \lesssim \frac{\tilde R}{\sqrt{m_i}}\langle v'\rangle_i\sin\frac{\theta}{2}\,.
$$
Consequently, we obtain in a similar fashion as before that
\begin{align}
\begin{split}\label{I1-2}
I_{i,1}(g_j,f)\gtrsim  \min_{1\leq j\leq I}\frac{\tilde r^{\lambda}}{\tilde R^{\lambda_{ij}}}\bigg[\min_{1\leq j\leq I}&K^{ij}(s_{ij},D_0,E_0)\,\big\|\langle v\rangle_i^\frac{\overline{\lambda}_{i}}{2}\phi_{A_k} f\big\|^2_{H^{\frac{\overline{\overline{s}}_{i}}{2}}} \\
&-C\Big( \max_{1\leq j\leq I} \tilde{R}^{2\lambda_{ij}}\|\theta^2b_{ij}\|_{L^1}\Big)\,E_0\,\|\langle \cdot\rangle^{\frac{\bar{\bar{\lambda}}_{i}}{2}}_i\, f\|^2_{L^2}\bigg]\,.
\end{split}
\end{align}
Now observe that $\phi^{2}_{\tilde{R}} + \sum^{N}_{k=1}\phi^{2}_{A_k}\geq 1$ where $N\sim(\tilde R/\tilde r)^3$ is the number of sets in the covering.  Consequently,
\begin{align}\label{coercive-principal-part}
(1+N)I_1(g_j,f) \geq I(g_j,f) + N\,I_1(g_j,f) &\geq  \tilde{c}_i\big\|\langle v\rangle_i^\frac{\overline{\lambda}_{i}}{2}\phi_{\tilde{R}} f\big\|^2_{H^{\frac{\overline{\overline{s}}_{i}}{2}}} - \tilde C_{i}\|\langle\cdot\rangle^{\frac{\bar{\bar{\lambda}}_{i}}{2}}_i f\|^2_{L^2} \nonumber \\ 
&\qquad +c'_i\sum^{N}_{k=1}\big\|\langle v\rangle_i^\frac{\overline{\lambda}_{i}}{2}\phi_{A_k} f\big\|^2_{H^{\frac{\overline{\overline{s}}_{i}}{2}}} - N\,C'_{i}\|\langle\cdot\rangle^{\frac{\bar{\bar{\lambda}}_{i}}{2}}_i f\|^2_{L^2}\\
& \geq\min\{\tilde c_i,c'_i\}\big\|\langle v\rangle_i^\frac{\overline{\lambda}_{i}}{2} f\big\|^2_{H^{\frac{\overline{\overline{s}}_{i}}{2}}} - \big(\tilde C_{i} + N\,C'_i\big)\|\langle\cdot\rangle^{\frac{\bar{\bar{\lambda}}_{i}}{2}}_i f\|^2_{L^2}\,,\nonumber
\end{align}
where we used \eqref{I1-1} for the first term and \eqref{I1-2} for the $N$ remaining.

\smallskip
\noindent
{\bf Step 2.} Let us estimate $I_2(g_j,f)$.  To this purpose use the Cancellation Lemma \ref{Canc1}, specifically Remark \ref{cancelation-remark}, that gives
\begin{align*}
I_2(g_j,f)&= \sum_{j=1}^I \int_{\mathbb{R}^3\times \mathbb{R}^3\times \S^2}|v-v_*|^{\lambda_{ij}} g_j(v_*)\big(f(v')^2-f(v)^2\big)b_{ij}(\widehat{u},\sigma)d\sigma dv_* dv\\
&\leq |\S| \sum_{j=1}^I \int_{\mathbb{R}^3\times \mathbb{R}^3}\int_0^\frac{\pi}{2}|v-v_*|^{\lambda_{ij}} g_j(v_*)f(v)^2\\
&\hspace{3cm}\times \Big[\frac{1}{\beta(\cos\theta)^{3+\lambda_{ij}}}-1\Big]b_{ij}(\cos\theta)\sin\theta\,d\theta dv_* dv\\
& = \sum_{j=1}^I\int_{\mathbb{R}^3\times \mathbb{R}^3}\int_0^\frac{\pi}{2}|v-v_*|^{\lambda_{ij}} g_j(v_*)f^2_i(v)b_{i}(\cos\theta)d\theta dv_* dv
\end{align*}
where we defined 
$$b_i(\cos\theta):= |\S| \Big[\frac{1}{\beta(\cos\theta)^{3+\lambda_{ij}}}-1\Big] \sin\theta\,b_{ij}(\cos\theta)\in L^{1}\big([0,\tfrac\pi2]\big)\,.$$
In addition, recall that 
$$|v-v_*|^{\lambda_{ij}}\leq  (m_im_j)^{-\frac{\lambda_{ij}}{2}}\,\langle v\rangle_i^{\lambda_{ij}}\langle v_*\rangle_j^{\lambda_{ij}},$$
therefore, since $\lambda_{ij}\leq 2$,
\begin{align*}
I_2(g_j,f)&\leq \max_{1\leq j\leq I}(m_im_j)^{-\frac{\lambda_{ij}}{2}}\,\|b_{i}\|_{L^1}\sum_{j=1}^I\|\langle \cdot \rangle_j^{\lambda_{ij}}g_j\|_{L^1}\|\langle \cdot \rangle_i^{\frac{\overline{\overline{\lambda}}_{i}}{2}}f_i\|^2_{L^2}\\
&\leq \max_{1\leq j\leq I}(m_im_j)^{-\frac{\lambda_{ij}}{2}}\|b_{i}\|_{L^1}\,E_0\,\|\langle \cdot \rangle_i^{\frac{\overline{\overline{\lambda}}_{i}}{2}}f_i\|^2_{L^2}=:\tilde C'_{i}\,\|\langle \cdot \rangle_i^{\frac{\overline{\overline{\lambda}}_{i}}{2}}f_i\|^2_{L^2}.
\end{align*}
This proves the proposition with $c_{i}:=\frac{\min\{\tilde c_i,c'_i\}}{N+1}$ and $C_i:=\frac{\tilde C_{i}+N\,C'_{i}}{N+1} + \tilde C'_i$.
\end{proof}
The following corollary of Proposition \ref{Coercive} extends the $L^{2}$ frame to $L^{p}$.
\begin{cor}\label{LemLp}
Assume that $\G=[g_j]_{1\leq j \leq I}\in U(D_0,E_0)$, $\lambda_{ij}\in(0,2]$ and $s_{ij}\in(0,2)$.  Then, for any $p\in(1,\infty)$ it holds that
$$\sum_{j=1}^I\int_{\mathbb{R}^3}Q_{ij}(f,g_j)(v)f(v)^{p-1}dv\leq -c_i\big\|\langle v\rangle^{\frac{\overline{\lambda}_{i}}{2}}_i f^{\frac p2}\big\|^2_{H^{\frac{\overline{\overline{s}}_{i}}{2}}}+C_{i}\big\|\langle \cdot\rangle_i^{\frac{\overline{\overline{\lambda}}_{i}}{p}}f\big\|^p_{L^p}$$
for positive constants $c_i$ and $C_i$ depending on the parameters of the model, $D_0$, $E_0$, and $p$.
\end{cor}
\begin{proof}
This arguments is based on \cite[Lemma 1]{RA}.  Using the weak formulation, it holds for any fixed $1\leq i \leq I$ that
\begin{align*}
\sum_{j=1}^I\int_{\mathbb{R}^3}&Q_{ij}(f,g_j)(v)f(v)^{p-1}dv \\
&= \sum_{j=1}^I\int_{\mathbb{R}^{3}\times \R^3\times  \mathcal{S}^{2}} g_j(v_*)f(v)\big(f(v')^{p-1}-f(v)^{p-1}\big)|v-v_*|^{\lambda_{ij}}b_{ij}(\widehat{u}\cdot\sigma)d\sigma dv_*dv\,.
\end{align*}
Using that, with notation $p'=\frac{p}{p-1}$,
\begin{align*}
f(v)\big[ f(v')^{p-1}-f(v)^{p-1}\big]&= f(v)^p\Big[\Big(\frac{f(v')^{\frac{p}{2}}}{f(v)^{\frac{p}{2}}}\Big)^{\frac{2}{p'}}-1\Big]\\
&\leq \frac{1}{p'} f(v)^p\Big[\frac{f(v')^p}{f(v)^p}-1\Big]-\frac{1}{\max\{p,p'\}}f(v)^p\Big[\frac{f(v')^{\frac{p}{2}}}{f(v)^{\frac{p}{2}}}-1\Big]^2\\
&\leq \frac{1}{p'} \Big[f(v')^p-f(v)^p\Big]-\frac{1}{\max\{p,p'\}}\Big[f(v')^{\frac{p}{2}}-f(v)^{\frac{p}{2}}\Big]^2\,,
\end{align*}
we conclude that
\begin{equation}\label{J1J2}
\sum_{j=1}^I\int_{\mathbb{R}^3}Q_{ij}(f,g_j)(v)f(v)^{p-1}dv\leq J_{1,i}(g_j,f) - J_{2,i}(g_j,f)\,,
\end{equation}
where
$$J_{1,i}(g_j,f) :=\frac{1}{p'}\sum_{j=1}^I\int_{\mathbb{R}^{3}\times \R^3\times \mathcal{S}^{2}} g_j(v_*)\Big[f(v')^p-f(v)^p\Big]|v-v_*|^{\lambda_{ij}}b_{ij}(\widehat{u}\cdot\sigma)d\sigma dv_*dv,$$
and
$$J_{2,i}(g_j,f):= \frac{1}{\max\{p,p'\}}\sum_{j=1}^I\int_{\mathbb{R}^{3}\times \R^3\times \mathcal{S}^{2}} g_j(v_*)\Big[ f(v')^{\frac{p}{2}} - f(v)^{\frac{p}{2}} \Big]^2|v-v_*|^{\lambda_{ij}}b_{ij}(\widehat{u}\cdot\sigma)d\sigma dv_*dv.$$
For the term $J_1(g_j, f )$ we use the Cancellation Lemma \ref{Canc1} to obtain that
\begin{align*}
p'\,J_{1,i}(g_j,f)&= |\S|\,\sum_{j=1}^I\int_{\mathbb{R}^{3}\times \R^3}\int_0^\frac{\pi}{2}|v-v_*|^{\lambda_{ij}} g_j(v_*)f(v)^p\\
&\qquad\times\Big[\frac{1}{\beta(\cos\theta)^{3+\lambda_{ij}}}-1\Big] \sin\theta\,b_{ij}(\cos\theta)d\theta dv_*dv\\
&= \sum_{j=1}^I\int_{\mathbb{R}^{3}\times\R^3}\int_0^\frac{\pi}{2}|v-v_*|^{\lambda_{ij}} g_j(v_*)f(v)^p\,b_{i}(\cos\theta)d\theta dv_*dv\,,
\end{align*}
where recall that the scattering $b_{i}$ was defined as
$$
b_{i}(\cos\theta) = |\S|\,\sin\theta \, \Big[\frac{1}{\beta(\cos\theta)^{3+\lambda_{ij}}}-1\Big] b_{ij}(\cos\theta)\in L^{1}\big([0,\tfrac\pi2]\big)\,.
$$
Consequently,
\begin{equation}\label{J1}
p' \,J_{1,i}(g_j,f) \leq \max_{1\leq j\leq I}(m_im_j)^{-\frac{\lambda_{ij}}{2}} \|\overline{b}_{i}\|_{L^1}\sum_{j=1}^I\|\langle\cdot\rangle_j^{\overline{\overline{\lambda}}_{i}}g_j\|_{L^1}\|\langle\cdot\rangle_i^{\frac{\overline{\overline{\lambda}}_{i}}{p}}f\|^p_{L^p}\leq C'_{i}\,\|\langle\cdot\rangle_i^{\frac{\overline{\overline{\lambda}}_{i}}{p}}f\|^p_{L^p}\,.
\end{equation}
Given that $\lambda_{ij}\leq2$, we can take $$C'_{i}=\max_{1\leq j\leq I}(m_im_j)^{-\frac{\lambda_{ij}}{2}} \| b_{i}\|_{L^1} E_0\,.$$
Focusing in the term $J_{2,i}(g_j,f)$ now, we use the coercive estimate \eqref{coercive-principal-part}
\begin{equation}\label{J2}
\max\big\{ p, p' \big\}J_{2,i}(g_j,f)\geq \tilde c_{i}\big\|\langle v\rangle^{\overline{\lambda}_{i}/2}_i f^{\frac p2}\big\|^2_{H^{\frac{\overline{\overline{s}}_{i}}{2}}} - \tilde C_{i} \|\langle\cdot\rangle_i^{\frac{\overline{\overline{\lambda}}_{i}}{p}}f\|^p_{L^p}\,.
\end{equation}
The result follows from \eqref{J1J2}, \eqref{J1}, \eqref{J2}  with $c_i = \frac{\tilde c_{i}}{\max\{p,p'\}}$ and $C_{i} = \frac{C'_{i}}{p'} + \frac{\tilde{C}_{i}}{\max\{p,p'\}}$.
\end{proof}
\subsection{Generation and propagation of $L^p$ integrability}
In this section, we study the $L^p$ integrability generation and propagation property for solutions of the monatomic Boltzmann system in the case $1<p<\infty$.  
\begin{proof}[Proof of Theorem \ref{theo2}]
Take $1\leq i\leq I$ and multiplying the $i^{th}$ row of the Boltzmann equation \eqref{B-E1} by $p\,f_i$.   Integrating on  $v\in\mathbb{R}^3$ and using Corollary \ref{LemLp}, we obtain that
 \begin{equation}\label{EqLp}
\begin{split}
\frac{d}{dt}\|f_i\|^p_{L^p}&\leq \sum_{j=1}^I \int_{\mathbb{R}^3}Q_{ij}(f_i,f_j)(v)f_i(v)^{p-1}dv\leq -c_i\big\|\langle v\rangle^{\frac{\bar{\lambda}_{i}}{2}}_i f^{\frac p2}_i\big\|^2_{H^{\frac{\overline{\overline{s}}_{i}}{2}}}+C_{i}\|\langle \cdot\rangle_i^{\frac{\overline{\overline{\lambda}}_{i}}{p}}f_i\|^p_{L^p}
\end{split}
\end{equation}
with the corresponding constants $c_{i}$ and $C_{i}$ given in the statement of the corollary.  We use below the weighted interpolation a couple of times
\begin{equation*}
\| \langle \cdot \rangle^{a}\,g \|_{L^r} \leq \| \langle \cdot \rangle^{a_1}\,g \|^{\theta}_{L^{r_1}}\| \langle \cdot \rangle^{a_2}\,g \|^{1-\theta}_{L^{r_2}}
\end{equation*}
with
\begin{equation*}
\frac{1}{r} = \frac{\theta}{r_1} + \frac{1-\theta}{r_2}\,,\qquad a=\theta a_1 + (1-\theta)a_2\,,\qquad \theta\in(0,1).
\end{equation*}
Now, by Sobolev inequality it follows that
\begin{align*}
\sqrt{c_{s}}\,\big\| \langle v\rangle^{\frac{\overline{\lambda}_{i}}{p}}_i f_i \big\|^{\frac p2}_{L^{\frac{p\,q_s}{2}}}=\sqrt{c_{s}}\,\big\| \langle v\rangle^{\frac{\overline{\lambda}_{i}}{2}}_i f^{\frac p2}_i \big\|_{L^{q_s}} \leq \big\|\langle v\rangle^{\frac{\overline{\lambda}_{i}}{2}}_i f^{\frac p2}_i\big\|_{H^{\frac{\overline{\overline{s}}_{i}}{2}}}\,,\qquad q_s :=\frac{6}{3 - \overline{\overline{s}}_{i}}>2\,. 
\end{align*}
In addition, using the weighted interpolation we conclude that for $p\in(1,\infty)$,
\begin{equation}\label{EqLp2}
\begin{split}
\big\|\langle \cdot \rangle_i^\frac{\overline{\overline{\lambda}}_{i}}{p}f_i \big\|_{L^p} \leq \| \langle \cdot \rangle_i^{\frac{\alpha_i}{p}} f_i  \|^{1-\theta_s}_{L^1}\big\| \langle v\rangle^{\frac{\overline{\lambda}_{i}}{p}}_i f_i \big\|^{\theta_s}_{L^{\frac{p\,q_s}{2}}}
\end{split}
\end{equation}
where $\frac1p =(1-\theta_s) + \frac{2}{pq_s}\theta_s$, namely, $\theta_s:= \frac{pq_s-q_s}{pq_s - 2}\in(0,1)$, and $\alpha_{i}\leq\bar{\bar\lambda}_i\frac{pq_s-2}{q_s-2}$.  Then, using Young's inequality it holds that
$$\frac{d}{dt}\|f_i\|^p_{L^p}+ \tilde c_i\,\big\| \langle v\rangle^{\frac{\overline{\lambda}_{i}}{p}}_i f_i \big\|^{p}_{L^{\frac{p q_s}{2}}} \leq \tilde C_{i}\| \langle \cdot \rangle_i^{\frac{\alpha_i}{p}} f_i  \|^{p}_{L^1}\,.$$
Using the same interpolation, with no weights, it holds then for the same $\theta_s$ that
$$\frac{d}{dt}\|f_i\|^p_{L^p}+\tilde c_{i}\frac{ \|f_i\|^{\frac p\theta_s}_{L^p}}{\|f_i\|^{p\frac{1-\theta_s}{\theta_s}}_{L^1}}\leq \tilde C_{i}\| \langle \cdot \rangle_i^{\frac{\alpha_i}{p}} f_i  \|^p_{L^1}\leq \tilde C_i\,C^{p}_{1}\Big(1+t^{-\frac{\alpha_i/p-2}{\lambda^\natural}}\Big)^p\,,\qquad t>0\,.$$
For the latter inequality in the right side we used Theorem \ref{Mom}.  Thus, introducing $X:=X(t)=\|f_i(t)\|^p_{L^p}$ one has, after invoking conservation of mass, that
\begin{equation}\label{ODE}
\frac{d X}{dt}+ a\, X^\frac{1}{\theta_s} \leq B_{t_\star}:= C_i\,C^p_{1}\Big(1+t_\star^{-\frac{\alpha_i/p-2}{\lambda^\natural}}\Big)^p,\qquad t > t_\star\,,
\end{equation}
where $a :=c_i\, \textbf{m}^{-p\frac{1-\theta_s}{\theta_s}}_0$.   Therefore, using \cite[Lemma 18]{AG} in the interval $t > t_\star$ it follows that
$$\|f_i(t)\|^{p}_{L^p} = X(t) \leq \Big(\frac{B_{t_\star}}{a}\Big)^{\theta_s} + \Big(\frac{\theta_s}{(1-\theta_s)\,a}\Big)^{\frac{\theta_s}{1-\theta_s}}\big( t-t_\star \big)^{-\frac{\theta_s}{1-\theta_s}}=:K_{t,t_\star},\,\qquad t>t_\star\,.$$
This proves the estimate \eqref{theo2-e1} when particularised to $t_\star=\frac{t_0}{2}$ with $C_{t_0}=K^{\frac{1}{p}}_{t_0,t_0/2}$.
\medskip
\noindent
Next, take an integer $n\geq2$ such that $2n\geq \frac{\alpha_i}{p}$ and assume $\| \langle \cdot \rangle_i^{2n} f_{i,0}  \|_{L^1}<\infty$.  Then, Theorem \ref{Mom} implies that $\sup_{t\geq0}\| \langle \cdot \rangle_i^{2n} f_{i}(t)\|_{L^1}\leq \max \big\{{\bf m}_{2n}(0), C_2 \big\}$.  Consequently, estimate \eqref{ODE} changes to
\begin{equation*}
\frac{d X}{dt}+ a\, X^\frac{1}{\theta_s} \leq B:= \tilde{C}_{i}\,\max \big\{{\bf m}_{2n}(0), C_2 \big\}^p\,,\qquad t>0\,,
\end{equation*}
which leads to
$$X(t)\leq\max\Big\{\|f_{i,0}\|_{L^p}^p, \Big(\frac{B}{a}\Big)^{\theta_s} \Big\}\,.$$
This proves estimate \eqref{theo2-e2} since $\frac{\alpha_i}{p}\leq \frac{\bar{\bar\lambda}_i\,q_s}{q_s-2} = \frac{3\,\bar{\bar\lambda}_i}{\bar{\bar{s}}_i}$\,.
\end{proof}
\begin{proof}[Proof of Corollary \ref{CorLp}]
In one hand using estimate (\ref{EqLp}) with $p=2$, we get that
$$\frac{d}{dt}\|f_i\|_{L^2}^2+c_{i}\|\langle v\rangle_i^{\frac{\overline{\lambda}_i}{2}}f_i\|_{H^{\frac{\overline{\overline{s}}_{i}}{2}}}^2\leq C_{i}\|\langle\cdot\rangle_i^{\frac{\overline{\overline{\lambda}}_{i}}{2}}f_i\|_{L^2}^2.$$
In the other hand, the interpolation \eqref{EqLp2} with $p=2$ give us
$$
\big\|\langle \cdot \rangle_i^\frac{\overline{\overline{\lambda}}_{i}}{2}f_i \big\|_{L^2} \leq \| \langle \cdot \rangle_i^{\frac{\alpha_i}{2}} f_i  \|^{1-\theta_s}_{L^1}\big\| \langle v\rangle^{\frac{\overline{\lambda}_{i}}{2}}_i f_i \big\|^{\theta_s}_{L^{q_s}} \lesssim \| \langle \cdot \rangle_i^{\frac{\alpha_i}{2}} f_i  \|^{1-\theta_s}_{L^1}\big\|\langle v\rangle^{\frac{\overline{\lambda}_{i}}{2}}_i f_i\big\|^{\theta_s}_{H^{\frac{\overline{\overline{s}}_{i}}{2}}}\,,
$$
with $q_s,\,\theta_{s},\,\alpha_i$ previously defined.  Consequently,
\begin{align*}
\frac{d}{dt}\|f_i\|_{L^2}^2+\tilde c_{i}\|\langle v\rangle_i^{\frac{\overline{\lambda}_i}{2}}f_i\|_{H^{\frac{\overline{\overline{s}}_{i}}{2}}}^2&\leq \tilde C_{i}\| \langle \cdot \rangle_i^{\frac{\alpha_i}{2}} f_i  \|^{2}_{L^1}\,,
\end{align*}
which yields after time integration in the interval $(t_0,t)$
$$\tilde c_{i}\int^{t}_{t_0}\|\langle v\rangle_i^{\frac{\overline{\lambda}_i}{2}}f_i(\tau)\|_{H^{\frac{\overline{\overline{s}}_{i}}{2}}}^2\,d\tau \leq \|f_i(t_0)\|^2_{L^2} + \tilde C_{i}\int^{t}_{t_0}\|\langle \cdot \rangle_i^{\frac{\alpha_i}{2}}\,f_{i}(\tau)\|^{2}_{L^1}d\tau \leq \tilde C_{t_0}(1+t)\,.$$
In the last inequality we used theorems \ref{Mom} and \ref{theo2} so that $\tilde C_{t_0}\lesssim t^{-\beta}_{0}+1$ for some $\beta>0$.  This is the first statement of the corollary with $C_{t_0}=\frac{\tilde{C}_{t_0}}{\tilde{c_i}}$.
\smallskip
\noindent
The second statement is also clear from this last estimate using the propagation part of theorems \ref{Mom} and \ref{theo2}.
\end{proof}
\subsection{$L^\infty$ Theory}
Let us study in this section the particular case of $p=\infty$.  We start with a coercive estimate for the levels of levels of each component.  We set in the sequel
\begin{equation*}
f^+_{K}(v):=\big( f(v) - K \big)1_{\{f_{K}\geq 0\}}\,.
\end{equation*}
\begin{lem}\label{Linfty}
Take $\mathbb{G}\in U(D_0,E_0)$, $f$ sufficiently smooth, $\lambda_{ij}\in(0,2]$ and $s_{ij}\in (0,2)$. Then,
\begin{align*}
\sum_{j=1}^I \int_{\mathbb{R}^3}Q_{ij}&(f, g_j)(v)f^+_{K}(v)dv\\
&\leq - c_{i} \|\langle v\rangle_i^{\frac{\overline{\lambda}_{i}}{2}}f^+_{K}\|^2_{H^{\frac{\overline{\overline{s}}_{i}}{2}}}+\tilde{C}_{i}\Big(\|\langle\cdot\rangle_i^\frac{\overline{\overline{\lambda}}_{i}}{2}f^+_{K}\|_{L^2}^2 + K\,\|\langle\cdot\rangle_i^{\overline{\overline{\lambda}}_i}f^+_{K}\|_{L^1}\Big)\,,\qquad K\geq0\,,
\end{align*}
for positive constant
\begin{equation*} 
\tilde{C}_{i} \lesssim C_{i} + \max_{1\leq j\leq I}\Big[\big(m_im_j\big)^{-\frac{\lambda_{ij}}2}\Big]\|b_{i}\|_{L^1}E_0\,,
\end{equation*}
where the constants $c_{i}$ and $C_{i}$ are defined in Proposition \ref{Coercive}.
\end{lem}
\begin{proof}
Set $K\geq 0$, define $f_{K}(v):= f(v)-K$, and recall that $f^+_{k}(v)= f_{K}(v)1_{\{f_{K}\geq 0\}}$. Note that
\begin{align*}
f(v)\big[f^+_{K}(v')-f_{K}^+(v)\big]&=f_{K}(v)\big[f^+_{K}(v')-f^+_{K}(v)\big]+K\big[f^+_{K}(v')-f^+_{K}(v)\big]\\
& =f_{K}(v)\big(1_{\{f_{K}\geq 0\}}+1_{\{f_{K}< 0\}}\big)\big[f^+_{K}(v')-f^+_{K}(v)\big]+K\big[f^+_{K}(v')-f^+_{K}(v)\big]\\
&\leq f^+_{K}(v)\big[f^+_{K}(v')-f^+_{K}(v)\big]+K\big[f^+_{K}(v')-f^+_{K}(v)\big]\,,
\end{align*}
where we used in the last step that
$$
f_{K}(v)1_{\{f_{K}<0\}}\big[f^+_{K}(v')-f^+_{K}(v)\big]=f_{K}(v)1_{\{f_{K}<0\}}f^+_{K}(v')\leq 0\,.
$$
Therefore, it holds that
\begin{align*}
f(v)\big[f^+_{K}(v')-f_{K}^+(v)\big]&=f^+_{K}(v)\big[f^+_{K}(v')-f^+_{K}(v)\big]+K\big[f^+_{K}(v')-f^+_{K}(v)\big]\\
&\leq -\frac{1}{2} \big[f^+_{K}(v')-f^+_{K}(v)\big]^2
+\frac{1}{2}\big[ \big( f^+_{K}(v')\big)^2-\big(f^+_{K}(v)\big)^2\big]\\
&\hspace{2cm} +K \big[f^+_{K}(v')-f^+_{K}(v)\big]\,.
\end{align*}
And consequently,
\begin{align*}
 \sum_{j=1}^I\int_{\mathbb{R}^3}Q_{ij}(f,g_j)(v)f^+_{K}(v)dv&=\sum_{j=1}^I\int_{\mathbb{R}^3\times \mathbb{R}^3\times \mathcal{S}^2}g_j(v_*)f(v)\big[f^+_{K}(v')-f^+_{K}(v)\big]\\
&\qquad\qquad\times|v-v_*|^{\lambda_{ij}}b_{ij}(\widehat{u}\cdot\sigma)d\sigma dv_*dv\\
&\leq J_1+J_2+J_3,
\end{align*}
where
\begin{align*}
J_1&=-\frac{1}{2}\displaystyle\sum_{j=1}^I\int_{\mathbb{R}^3\times \mathbb{R}^3\times \mathcal{S}^2}g_j(v_*)\big[f^+_{K}(v')-f^+_{K}(v)\big]^2|v-v_*|^{\lambda_{ij}}b_{ij}(\widehat{u}\cdot\sigma)d\sigma dv_*dv\\
J_2&=\frac{1}{2}\displaystyle\sum_{j=1}^I\int_{\mathbb{R}^3\times \mathbb{R}^3\times \mathcal{S}^2}g_j(v_*)\big[\big(f^+_{K}(v')\big)^2-\big(f^+_{K}(v)\big)^2\big]|v-v_*|^{\lambda_{ij}}b_{ij}(\widehat{u}\cdot\sigma)d\sigma dv_*dv\\
J_3&=K\displaystyle\sum_{j=1}^I\int_{\mathbb{R}^3\times \mathbb{R}^3\times \mathcal{S}^2}g_j(v_*)\big[f^+_{K}(v')-f^+_{K}(v)\big]|v-v_*|^{\lambda_{ij}}b_{ij}(\widehat{u}\cdot\sigma)d\sigma dv_*dv\,.
\end{align*}
Using the coercive estimate given in the proof of Proposition \ref{Coercive} one deduces that
$$J_1\leq -c_{i} \big\|\langle v\rangle_i^{\frac{\overline{\lambda}_{i}}{2}}f^+_{K}\big\|^2_{H^{\frac{\overline{\overline{s}}_{i}}{2}}}+C_{i}\|\langle\cdot\rangle_i^\frac{\overline{\overline{\lambda}}_{i}}{2}f^+_{K}\|_{L^2}^2\,.$$
Furthermore, using the Cancellation Lemma \ref{Canc1} for the terms $J_1$ and $J_2$, it follows that
\begin{align*}
J_2&\lesssim\max_{1\leq j\leq I}\Big[\big(m_im_j\big)^{-\frac{\lambda_{ij}}2}\Big]\|b_{i}\|_{L^1}E_0\|\langle\cdot\rangle_i^\frac{\overline{\overline{\lambda}}_{i}}{2}f^+_{K}\|^2_{L^2}\qquad \text{and}\\ 
J_3&\lesssim K\max_{1\leq j\leq I}\Big[\big(m_im_j\big)^{-\frac{\lambda_{ij}}2}\Big]\|b_{i}\|_{L^1}E_0\|\langle\cdot\rangle_i^{\overline{\overline{\lambda}}_{i}}f^+_{K}\|_{L^1}.
\end{align*}
The result follows adding the estimates.
\end{proof}

\begin{lem}\label{inc-hom-lem}
Fix dimension $d\geq2$, $s\in(0,2)$ and $1-\frac sd\leq\alpha\leq1$.  Then,
\begin{equation*}
\| \varphi \|^{q\alpha}_{L^{q}} \leq \| \varphi \|^{2(\frac1\theta-1)}_{L^{2}}\| \varphi \|^{2}_{H^{\frac s2}}\,, 
\end{equation*}
where
\begin{equation*}
q:=q(s,\alpha)= 2\Big(1+\frac{s}{d\alpha}\Big)>2\,,\qquad \theta:=\theta(s,\alpha) = \frac{d}{s+d\alpha}\in(0,1]\,.
\end{equation*}
\end{lem}
\begin{proof}
Note that for $2\leq q\leq p_*=\frac{2}{1-s/d}$ Lebesgue's interpolation and Sobolev inequality gives that
\begin{equation*}
\| \varphi \|^{q\alpha}_{L^{q}}\leq \| \varphi \|^{(1-\theta)q\alpha}_{2}\| \varphi \|^{\theta q\alpha}_{L^{p_*}}\leq \| \varphi \|^{(1-\theta)q\alpha}_{2}\| \varphi \|^{\theta q\alpha}_{H^{\frac s2}}\,,\qquad \theta = \frac ds\Big(1 - \frac2q\Big)\,.
\end{equation*}
Choosing $\theta q \alpha =2$ we arrive to
\begin{equation*}
q=2\Big(1+\frac{s}{d\alpha}\Big)\in(2,p_*]\,,\quad\text{and consequently}\quad \theta=\frac{d}{s+d\alpha}>0\,.
\end{equation*}
The proof follows by writing $(1-\theta)q\alpha = 2(\frac1\theta -1)$. 
\end{proof}
\begin{proof}[Proof of Theorem \ref{theo3}]
We follow the spirit of the argument used in the proof of \cite[Theorem 2]{RA} after fixing $1\leq i \leq I$.  For solution of the Boltzmann system we apply Lemma \ref{Linfty} with  $\mathbb{G} = \mathbb{F}$ and $f=f_i$ to conclude that
\begin{equation}\label{E1}
\frac{1}{2}\frac{d}{dt} \|f^+_{K}\|^2_{L^2}+c_{i}\|\langle v \rangle_i^{\frac{\overline{\lambda}_{i}}{2}}f^+_{K}\|^2_{H^{\frac{\overline{\overline{s}}_{i}}{2}}}\leq \tilde{C}_{i}\Big(\|\langle\cdot\rangle_i^\frac{\overline{\overline{\lambda}}_{i}}{2}f^+_{K}\|^2_{L^2}+K \|\langle\cdot\rangle_i^{\overline{\overline{\lambda}}_{i}}f^+_{K}\|_{L^1}\Big).
\end{equation}
Let us introduce the levels and times ($K>0$ and $t_*>0$) 
$$K_k:=K\Big(1-\frac{1}{2^{k}}\Big),\qquad t_k=t_*\Big(1-\frac{1}{2^{k+1}}\Big),\qquad k=0,1,2,\cdots\,.$$
Here $K>0$ will be chosen later sufficiently large.  Define also the energy functional (with notation $f_k = f^+_{K_k}$)
$$W_k:= \frac{1}{2}\sup_{t\in[t_k,T]}\|f_{k}(t)\|^2_{L^2}+c_i\int_{t_k}^T\|\langle\cdot\rangle_i^\frac{\overline{\lambda}_{i}}{2}f_{k}(\tau)\|^2_{H^{\frac{\overline{\overline{s}}_{i}}{2}}}d\tau,\qquad T>t_*>0\,.$$
Integrating estimate \eqref{E1} we deduce that for $t_{k-1}\leq \tau\leq t_k$
$$W_k\leq \frac{1}{2}\|f_{k}(\tau)\|^2_{L^2}+\tilde{C}_{i}\int_{t_{k-1}}^T\Big(\|\langle \cdot \rangle_i^\frac{\overline{\overline{\lambda}}_{i}}{2}f_{k}\|_{L^2}^2+ K\|\langle \cdot \rangle_i^{\overline{\overline{\lambda}}_{i}}f_{k}\|_{L^1}\Big)\,,$$
and taking the mean over $\tau\in [t_{k-1},t_k]$ (noticing that $t_k-t_{k-1}=\frac{t_*}{2^{k+1}}$), it follows that
\begin{align}\label{E3}
\begin{split}
W_k &\leq \Big(\frac{2^k}{t_*}+\tilde{C}_{i}\Big)\int_{t_{k-1}}^T\Big(\|\langle\cdot \rangle_i^\frac{\overline{\overline{\lambda}}_{i}}{2}f_{k}\|_{L^2}^2+K\|\langle\cdot \rangle_i^{\overline{\overline{\lambda}}_{i}}f_{k}\|_{L^1}\Big)\\
&\leq 2^k\Big(\frac{2^k}{t_*}+\tilde{C}_{i}\Big)\int_{t_{k-1}}^T\|\langle\cdot \rangle_i^\frac{\overline{\overline{\lambda}}_{i}}{2}f_{k-1}1_{f_k\geq0}\|_{L^2}^2\,.
\end{split}
\end{align}
where in the last step we used the key observation that in the set $\{f_{k}\geq 0\}$ one has that $f_{k-1}\geq 2^{-k}K$, thus 
$$
K\|\langle\cdot \rangle_i^{\overline{\overline{\lambda}}_{i}}f_{k}\|_{L^1} \leq 2^k\|\langle\cdot \rangle_i^\frac{\overline{\overline{\lambda}}_{i}}{2}f_{k-1}1_{f_k\geq0}\|_{L^2}^2\,.
$$
In fact, keep in mind that for any $\beta>0$
\begin{equation}\label{inc-hom}
1_{\{f_k\geq0\}}\leq \Big(\frac{2^k}{K}f_{k-1}\Big)^\beta\,.
\end{equation}
Now,  interpolation give us that
\begin{equation*}
\|\langle\cdot \rangle_i^\frac{\overline{\overline{\lambda}}_{i}}{2}f_{k-1}1_{f_k\geq0}\|_{L^2}\leq \|\langle\cdot \rangle_i^{\frac{\xi-1}{\xi-2}\overline{\bar \lambda}_i}f\|^{1-\theta}_{L^1}\|f_{k-1}1_{f_k\geq0}\|^{\theta}_{L^\xi}\,,\qquad \xi>2,\quad \theta =\frac{\xi}{2(\xi-1)}\,.
\end{equation*}
Moments are controlled by Theorem \ref{Mom} so that
\begin{equation*}
\|\langle\cdot \rangle_i^{\frac{\xi-1}{\xi-2}\overline{\bar \lambda}_i}f\|_{L^1} \leq C\Big(1+ t_*^{-\frac{\frac{\xi-1}{\xi-2}\overline{\bar \lambda}_i - 2}{\lambda^\natural}}\Big)\,,
\end{equation*}
for a constant $C:=C(E_0,\xi)>0$.   Also, using \eqref{inc-hom} with $\beta=q-\xi>0$ with $q>2$ given in Lemma \ref{inc-hom-lem}, that is $q=2\Big(1+\frac{\overline{\bar{s}}_i}{3\alpha}\Big)$, it holds that
\begin{equation*}
\|f_{k-1}1_{f_k\geq0}\|_{L^\xi} \leq \frac{2^{k(\frac{q}{\xi} - 1)}}{K^{\frac{q}{\xi} - 1}}\|f_{k-1}\|^{\frac{q}{\xi}}_{L^q}\,.
\end{equation*}
In this way, we are led from \eqref{E3} and some algebra to the estimate
\begin{equation*}
W_k \leq \frac{2^{2k\frac{q-1}{\xi-1}}}{K^{\frac{q-\xi}{\xi-1}}}\,\tilde{C}\,\Big( 1 + t^{-\alpha_0}_* \Big)\int_{t_{k-1}}^T\|f_{k-1}\|^{\frac{q}{\xi-1}}_{L^q}\,,\qquad 2<\xi<q\,,
\end{equation*} 
for a constant $\tilde{C}$ depending on $(E_0,\xi)$ (so the limit $\xi\rightarrow 2$ is not allowed) and
$\alpha_0 := \frac{\overline{\bar{\lambda}}_i - \frac{2(\xi-2)}{\xi-1}}{\lambda^\natural}$.  Invoking Lemma \ref{inc-hom-lem} with $s=\overline{\bar{s}}_{i}$ and $\alpha=\frac{1}{\xi-1}<1$ we conclude that for $k=1,2,\cdots$
\begin{align*}
W_k &\leq \frac{2^{2k\frac{q-1}{\xi-1}}}{K^{\frac{q-\xi}{\xi-1}}}\,\tilde{C}\,\Big( 1 + t^{-\alpha_0}_* \Big)\sup_{t\in[t_{k-1},T]}\| f_{k-1}\|^{2(\frac{\overline{\bar{s}}_{i}}{3} + \frac{1}{\xi-1} - 1)}_{L^{2}}\int_{t_{k-1}}^T\|f_{k-1}\|^{2}_{H^{\frac{\overline{\bar{s}}_{i}}{2}}}\\
&\leq \frac{2^{2k\frac{q-1}{\xi-1}}}{K^{\frac{q-\xi}{\xi-1}}}\,\tilde{C}\,c^{-1}_{i}\,\Big( 1 + t^{-\alpha_0}_* \Big)W^{\frac{\overline{\bar{s}}_{i}}{3} + \frac{1}{\xi-1}}_{k-1}\,,\qquad\qquad  2<\xi<\frac{2-\frac {\overline{\bar{s}}_{i}}{3}}{1-\frac {\overline{\bar{s}}_{i}}{3}}\,.
\end{align*} 
Set $a=\frac{\overline{\bar{s}}_{i}}{3} + \frac{1}{\xi-1}>1$, $Q=2^{\frac{2}{a-1}\frac{q-1}{\xi-1}}>1$, and choose
$$
K=\Big[ 4^{\frac{q-1}{\xi-1}}\,\tilde{C}\,c^{-1}_i\,W^{a-1}_0\,\Big( 1 + t^{-\alpha_0}_* \Big) \Big]^{\frac{\xi-1}{q-\xi}}\,.
$$
Then, $W^{*}_{k}=\frac{W_0}{Q^{k-1}}$ is a super solution of the aforementioned recursion inequality.  Thus,
$$
W_k \leq W^*_{k} \rightarrow 0\quad {\rm as}\quad k\rightarrow \infty,
$$
provided $W_0$ is finite.  The fact that $W_0$ is finite is clear invoking Theorem \ref{theo2} and Corollary \ref{CorLp}
$$
W_0=\frac{1}{2}\sup_{t\in [\frac{t_*}{2}, T]}\|f(t)\|^2_{L^2}+c_{i}\int_{\frac{t_*}{2}}^T\big\|\langle\cdot\rangle_i^\frac{\overline{\lambda}_{i}}{2}f(\tau)\|^2_{H^{\frac{\overline{\overline{s}}_i}{2}}}d\tau\leq C_{t_*}(1+T),\qquad T>t_*>0,
$$
where $C(t_*)\lesssim \big( t^{-\beta}_* + 1 \big)$ (for some $\beta>0$) is defined in such theorem and corollary.

\medskip
\noindent
As a consequence, since $K_{k}\rightarrow K$ and $t_{k}\rightarrow t_*$ as $k\rightarrow\infty$,
$$
\sup_{t\in [t_*, T]}\|f_{K}^+(t)\|_{L^2}=0\,,
$$
and thus, using the choice for $K$,
\begin{equation}\label{Linf-T}
f(t)\leq K\sim C_{t_*}(1+T^{(a-1)\frac{\xi-1}{q-\xi}})\,,\qquad T\geq t \geq t_{*}>0\,,
\end{equation}
with a constant  $C_{t_*}\lesssim (t^{-\beta}_*+ 1)$ (for some $\beta>0$) depending only on $(D_0,E_0,\xi)$.  In order to make estimate \eqref{Linf-T} independent of $T$ we recall that the Boltzmann system is time-invariant, so we can perform previous analysis in any time interval $(t_0,t_0+T]$.  In this case, $T>0$ plays the role of the time interval length which can be set, for example, equal $1$.  It is, of course, essential that the constants involved only depend on the conserved quantities $(D_0, E_0)$.

\medskip
\noindent
Regarding the propagation part of the result one uses a slight modification of the previous argument, we refer to \cite{RA} for the details.  We only mention that the condition $f_{i,0}\in L^{1}_{2n}$ for $2n\geq\frac{18\bar{\bar{\lambda}}_i}{\bar{\bar s}_i}$\footnote{This condition follows after specifically choosing $\xi = \frac{2-\frac {\overline{\bar{s}}_{i}}{2}}{1-\frac {\overline{\bar{s}}_{i}}{3}}$ in the argument.} is necessary to guarantee the finiteness of the constants involved near time zero.
\end{proof}

\section{Appendix}
\subsection{Interpolation Lemma}
\begin{lem}[Mixed interpolation]\label{App-interp}
For any $b\geq a\geq0$ and $\beta>0$
\begin{equation*}
{\bf m}_{a+\beta,i}[f] \, {\bf m}_{b,j}[g]\leq \theta\, {\bf m}_{b+\beta,i}[f]\, {\bf m}_{a,j}[g] + (1-\theta)\,{\bf m}_{b+\beta,j}[g]\, m_{a,i}[f]\,,
\end{equation*}
with $\theta = \frac{\beta}{b+\beta-a} \in (0,1]$.
\end{lem}
\begin{proof}
Standard interpolation gives that
\begin{align*}
{\bf m}_{a+\beta,i}[f] &\leq {\bf m}^{\theta}_{b+\beta,i}[f]\,{\bf m}^{1-\theta}_{a,i}[f]\,, \qquad\text{and}\\
{\bf m}_{b,j}[g] &\leq {\bf m}^{1-\theta}_{b+\beta,j}[g]\,{\bf m}^{\theta}_{a,j}[g]\,,\qquad \theta=\frac{\beta}{b+\beta-a}\,.
\end{align*}
The result follows from here associating the common powers and applying Young's inequality.
\end{proof}

\subsection{Cancellation Lemma}
We show in this section the cancellation Lemma of Boltzmann equation for monatomic gas mixtures.  All results are proven assuming a integrable approximation of the scattering kernel, so that the collision operator can be separated, and then arguing by density.  For the classical Boltzmann equation the reader can consult \cite{ADVW} where technical details are filled. 

\smallskip
\noindent
In all the results below we assume the scattering is forward, that is, $B_{ij}$ has support in the set $\{\widehat{u}\cdot\sigma\geq0\}$.
\begin{lem}[Cancellation lemma]\label{P-1}
For a.e. $v_*\in \mathbb{R}^3$ and $1\leq i,j\leq I$ we have that
\begin{equation}\label{Canc1}
\int_{\mathbb{R}^3\times \mathcal{S}^{2}}B_{ij}( u,\widehat{u}\cdot\sigma)\big(f(v')-f(v) \big)dvd\sigma = (f*S_{ij})(v_*),\qquad u = v - v_*\,,
\end{equation}
where
$$
S_{ij}(u) = |\S|\int_0^\frac{\pi}{2} \bigg[\frac{1}{\beta(\cos\theta)^{3}}B_{ij}\Big(\frac{|u|}{\beta(\cos\theta)}, \cos(\theta)\Big) - B_{ij}\big(|u|, \cos(\theta)\big)\bigg]\sin\theta\,d\theta.
$$
Here $\beta(x) =  \sqrt{\alpha^2 + (1-\alpha)^2 + 2\alpha(1-\alpha)x}\in(0,1]$.
\end{lem}
\begin{proof}
For each $\sigma$ and $v_*$ fixed we preform the change of variables $v\rightarrow v'_{ij}$.  Recall that 
$$v'= v_*+ \frac{m_i}{m_i+m_j}(v-v_*)+\frac{m_j}{m_i+m_j}|v-v_*|\sigma\,.$$
This change of variables is well defined on the set $\{\cos\theta\geq0\}$.  Indeed, it follows by a direct calculation that the Jacobian is given by
\begin{align*}
\Big|\frac{dv'}{dv}\Big|&= \alpha^2\big( \alpha + (1-\alpha) \,\widehat{u}\cdot\sigma\big)\,,\qquad \alpha=\frac{m_i}{m_i+m_j}\,,
\end{align*}
where $\widehat{u}=\frac{v-v_*}{|v-v_*|}$.  One can also find the relations between magnitude and scattering angles
\begin{align*}
&\frac{|v' - v_{*}|}{|v-v_*|} = \sqrt{\alpha^2 + (1-\alpha)^2 + 2\alpha(1-\alpha)\widehat{u}\cdot\sigma }=:\beta(\widehat{u}\cdot\sigma)\in(0,1]\,,\\
&\qquad\text{and}\qquad\widehat{u'}\cdot\sigma = \frac{\alpha\,\widehat{u}\cdot\sigma+1-\alpha}{ \sqrt{\alpha^2 + (1-\alpha)^2 + 2\alpha(1-\alpha)\widehat{u}\cdot\sigma} } =: \phi(\widehat{u}\cdot\sigma)\,,
\end{align*}
where $\widehat{u'}=\frac{v'-v_*}{|v'-v_*|}$.  Since $\Big|\frac{dv'}{dv}\Big|\geq \alpha^3>0$ in the set $\{\widehat{u}\cdot\sigma=\cos\theta\geq0\}$, the inverse transformation $v'\rightarrow \psi_\sigma(v')=v$ is, then, well defined.  Similarly, note that $\phi$ is also invertible and its derivative is given by
\begin{equation*}
\phi'(x)= \frac{\alpha^2\big(\alpha + (1-\alpha)x\big)}{ \big(\alpha^2 + (1-\alpha)^2 + 2\alpha(1-\alpha)x\big)^{\frac32} }\geq\alpha^3>0\,,\qquad x\in[0,1]\,.
\end{equation*}
Applying this change of variable to the term $\displaystyle \int B_{ij}\,f'$ in the left-hand of \eqref{Canc1} we find that
\begin{align}
\begin{split}\label{change-v-v'}
&\int_{\mathbb{R}^3\times \mathcal{S}^{2}} B_{ij}(v-v_*,\widehat{u}\cdot\sigma) f(v')dvd\sigma= \int_{\mathbb{R}^3\times \mathcal{S}^{2}} B_{ij}\big(\psi_\sigma(v')-v_*,\phi^{-1}(\widehat{u'}\cdot\sigma)\big) f(v')\Big| \frac{dv}{dv'}\Big|dv'd\sigma\\
&=\int_{\mathbb{R}^3} f(v)\int_{\phi^{-1}(\widehat{u}\cdot\sigma)\geq0}\frac{1}{\alpha^2\big( \alpha + (1-\alpha) \,\phi^{-1}(\widehat{u}\cdot\sigma)\big)} B_{ij}\Big(\frac{|u|}{\beta(\phi^{-1}(\widehat{u}\cdot\sigma))}, \phi^{-1}(\widehat{u}\cdot\sigma)\Big)d\sigma dv\,.
\end{split}
\end{align}
We renamed $v'$ to $v$ before interchanging integrals in the last step.  The inner integral is further expanded with polar coordinates performing the change of variables $x=\phi^{-1}(\widehat{u}\cdot\sigma)$.  In this way,
\begin{align*}
\int_{\phi^{-1}(\widehat{u}\cdot\sigma)\geq0}&\frac{1}{\alpha^2\big( \alpha + (1-\alpha) \,\phi^{-1}(\widehat{u}\cdot\sigma)\big)} B_{ij}\Big(\frac{|u|}{\beta(\phi^{-1}(\widehat{u}\cdot\sigma))}, \phi^{-1}(\widehat{u}\cdot\sigma)\Big)d\sigma\\
&=|\S|\int_{x\geq0}\frac{\phi'(x)}{\alpha^2\big( \alpha + (1-\alpha)x\big)} B_{ij}\Big(\frac{|u|}{\beta(x)}, x\Big)dx\\
&= |\S|\int_0^\frac{\pi}{2} \frac{\sin\theta}{\beta(\cos\theta)^3}B_{ij}\Big(\frac{|u|}{\beta(\cos\theta)}, \cos(\theta)\Big)d\theta\,.
\end{align*}
In the last step we used that
$$
\frac{\phi'(x)}{\alpha^2\big( \alpha + (1-\alpha)x\big)} = \frac{1}{\beta(x)^3}\,.
$$
Therefore, estimate \eqref{Canc1} follows with
\begin{align*}
S_{ij}(u) = |\S|\int_0^\frac{\pi}{2} \bigg[ \frac{1}{\beta(\cos\theta)^{3}}B_{ij}\big(\frac{|u|}{\beta(\cos\theta)}, \cos(\theta)\big) - B_{ij}\big(|u|, \cos(\theta)\big)\bigg]\sin\theta\,d\theta.
\end{align*}
\end{proof} 
\begin{rem}\label{cancelation-remark}
Note that for cross sections of the form $B_{ij}(|u|, \cos\theta)=|u|^{\lambda_{ij}}\cos\theta\,,$
one has the explicit form
$$
S_{ij}(u)=|\S|\,|u|^{\lambda_{ij}}\,\int_0^{\pi/2}\Big[\frac{1}{\beta(\cos\theta)^{3+\lambda_{ij}}}-1\Big]\sin\theta\,b_{ij}(\cos\theta)\,d\theta\geq0\,.
$$
\end{rem}
\subsection{Coercive estimate of the Dirichlet form}
In this subsection we explore a coercive estimate for Dirichlet form of the collision operator for monatomic gas mixtures.  The argument is inspired in the classical case developed in \cite{ADVW} and is based in a series of lemmata.

\smallskip
\noindent
Recall the weak formulation for a suitable functions $f$, $g$, and $\varphi(v)$,
\begin{equation}\label{Q+}
\int_{\mathbb{R}^3}Q_{ij}^+(f,g)\varphi(v)dv=\int_{\mathbb{R}^{3}\times \R^3\times \mathcal{S}^{2}}B_{ij}(v-v_*,\sigma)g(v_*)f(v)\varphi(v'_{ij})dvdv_*d\sigma.
\end{equation}
\begin{lem}[Bobylev identity]\label{P0}
Let $f\in L^2(\R^3)$ and $g\in L^1(\R^3)$.  For
\begin{equation*}
B_{ij}(v-v_*,\sigma) = b_{ij}\Big(\frac{v-v_*}{|v-v_*|}\cdot\sigma\Big)\,,
\end{equation*}
it holds that
$$\mathcal{F}\big( Q^+_{ij}(f,g)\big)(\xi)=\int_{\mathcal{S}^{2}}\mathcal{F}(g)(\xi^-_{ij})\mathcal{F}(f)(\xi^+_{ij})b_{ij}\Big(\frac{\xi}{|\xi|}\cdot\sigma\Big)d\sigma\,,\qquad 1\leq i,j\leq I\,,$$
where $\xi = \xi^+_{ij} + \xi^-_{ij}$.  More explicitly,
$$
\xi^+_{ij}:=\frac{m_i\, \xi}{m_i + m_i} + \frac{m_j \,|\xi| \,\sigma}{m_i+m_j} \qquad\text{and}\qquad \xi^-_{ij}:=\frac{m_j\, \xi}{m_i + m_i} - \frac{m_j \,|\xi| \,\sigma}{m_i+m_j}\,.
$$
\end{lem}
\begin{proof}
Plugging $\varphi(v)=e^{-iv\cdot\xi}$ in the weak formulation \eqref{Q+}, we get that
\begin{align*}
\mathcal{F}\big(Q_{ij}^+(f,g)\big)(\xi)&=\int_{\mathbb{R}^{3}\times\R^3\times S^{2}}g(v_*)f(v)B_{ij}(v-v_*,\sigma)e^{-iv'_{ij}\cdot\xi}dvdv_*d\sigma\\
&=\int_{\mathbb{R}^{3}\times \R^3\times S^{2}}g(v_*)f(v)B_{ij}(v-v_*,\sigma)e^{-i\frac{m_iv+m_jv_*}{m_i+m_j}\cdot\xi} e^{-\frac{m_j}{m_i+m_j}|v-v_*|\sigma\cdot\xi}dvdv_*d\sigma\,.
\end{align*}
Note that, a key remark by Bobylev,
$$\int_{\mathcal{S}^{2}}b_{ij}\Big(\frac{v-v_*}{|v-v_*|}\cdot\sigma\Big)e^{-i\frac{m_j}{m_i+m_j}|v-v_*|\sigma\cdot\xi}d\sigma=\int_{\mathcal{S}^{2}}b_{ij}\Big(\frac{\xi}{|\xi|}\cdot\sigma\Big)e^{-i\frac{m_j}{m_i+m_j}|\xi|\sigma\cdot(v-v_*)}d\sigma\,. $$
Thus,
\begin{align*}
\mathcal{F}&\big(Q_{ij}^+(f, g)\big)(\xi)\\
&=\int_{\mathbb{R}^{3}\times\R^3\times \mathcal{S}^{2}}g(v_*)f(v)b_{ij}\Big(\frac{\xi}{|\xi|}\cdot\sigma\Big)e^{-i\xi\cdot\frac{m_iv+m_jv_*}{m_i+m_j}}e^{-i\frac{m_j}{m_i+m_j}|\xi|\sigma\cdot(v-v_*)}dvdv_*d\sigma\\
&=\int_{\mathcal{S}^{2}}\bigg(\int_{\mathbb{R}^{3}}f(v)e^{-iv\cdot\xi^+_{ij}}dv\bigg)\bigg(\int_{\R^3}g(v_*)e^{-iv_*\cdot\xi^-_{ij}}dv_*\bigg)b_{ij}\Big(\frac{\xi}{|\xi|}\cdot\sigma\Big)d\sigma\,.
\end{align*}
The result follows from here.
\end{proof}
\begin{lem}\label{P1}
Let $g\in L^1(\R^3)$ and $f \in L^2(\R^3)$.  Then,
\begin{align*}
\int_{\mathbb{R}^{2N}}&\int_{\mathcal{S}^2}b_{ij}(\widehat{u}\cdot\sigma)g(v_*)\big( f(v'_{ij})-f(v) \big)^2dvdv_*d\sigma\\
&=\frac{1}{(2\pi)^3}\int_{\R^3}\int_{\mathcal{S}^2} b_{ij}\Big(\frac{\xi}{|\xi|}\cdot\sigma\Big)\Big[\mathcal{F}(g)(0)\,|\mathcal{F}(f)(\xi)|^2+\mathcal{F}(g)(0)\,|\mathcal{F}(f)(\xi^+_{ij})|^2\\
&\qquad\qquad-\mathcal{F}(g)(\xi^-_{ij})\,\mathcal{F}(f)(\xi^+_{ij})\,\overline{\mathcal{F}(f)}(\xi) -\overline{\mathcal{F}(g)}(\xi^-_{ij})\, \overline{\mathcal{F}(f)}(\xi^+_{ij})\,\mathcal{F}(f)(\xi)\Big]d\xi d\sigma,
\end{align*}
with the same definitions of $\xi^\pm_{ij}$ of Lemma \ref{P0}.
\end{lem}
\begin{proof}
Expanding the quadratic term gives three terms, namely,
$$\big( f(v'_{ij})-f(v) \big)^2 = f(v'_{ij})^2+f(v)^2-2f(v'_{ij})f(v).$$
We begin with the middle term. By the pre-post collisional change of variables and Parseval's identity,
\begin{align*}
\int_{\R^3\times \R^3\times \S^2} b_{ij}(\widehat{u}\cdot\sigma)g(v_*)f(v)f(v'_{ij})dvdv_*d\sigma&=\int_{\R^3} Q^+_{ij}(f,g)\,f\,dv\\
&= \frac{1}{(2\pi)^3}\int_{\R^3} \mathcal{F}\big[Q^+_{ij}(f,g)\big]\overline{\mathcal{F}(f)}d\xi.
\end{align*}
Furthermore, using Bobylev's identity Lemma \ref{P0} in the right-side it holds that 
\begin{align*}
\int_{\R^3\times \R^3\times \S^2} b_{ij}(\widehat{u}\cdot\sigma)g(v_*)&f(v)f(v'_{ij})dvdv_*d\sigma\\
&=\frac{1}{(2\pi)^3}\int_{\R^3\times \S^2} b_{ij}\big(\widehat{\xi}\cdot\sigma\big)\mathcal{F}(g)(\xi^-_{ij})\mathcal{F}(f)(\xi^+_{ij})\overline{\mathcal{F}(f)}(\xi)d\xi d\sigma.
\end{align*}
The right-side is equal to its complex conjugate since the left-side is real valued.  Now, we note that $\displaystyle\int_{\mathcal{S}^{2}}b_{ij}(\widehat{u}\cdot\sigma)d\sigma$ does not depend on the unit vector $\widehat{u}\in\mathcal{S}^2$, consequently it follows that
\begin{align*}
\int_{\R^3\times \R^3\times \S^2} b_{ij}(\widehat{u}\cdot\sigma)g(v_*)f(v)^2dvdv_*d\sigma&=\int_{\S^2} b_{ij}(\widehat{u}\cdot\sigma)d\sigma \int_{\R^3} g(v_*)dv_*\int_{\R^3} f(v)^2dv\\
&=\frac{1}{(2\pi)^3}\int_{\S^2} b_{ij}\big(\widehat{\xi}\cdot\sigma\big)d\sigma\mathcal{F}(g)(0)\int_{\R^3}|\mathcal{F}(f)(\xi)|^2 d\xi,
\end{align*}
where we have applied the usual Plancherel identity.  Finally, for the term involving $f(v')^2$ we first make the change of variables $(v,v_*)\rightarrow (v-v_*, v_*)$, and then $v\rightarrow  v^{+}_{ij}$ (and rename as $v$) as in Cancellation lemma to obtain that
\begin{align*}
\int_{\R^3\times \R^3\times \S^2}&  b_{ij}(\widehat{u}\cdot\sigma) g(v_*)f(v'_{ij})^2\,d\sigma\,dv\,dv_*=\int_{\R^3\times \S^2\times \R^3} b_{ij}(\widehat{v}\cdot\sigma)g(v_*)\big| f(v^{+}_{ij}+v_*)\big|^2 dvd\sigma dv_*\\
&= \int_{\R^3} g(v_*) \int_{\S^2\times \R^3}\frac{b_{ij}\big( \phi^{-1}(\widehat{v}\cdot\sigma) \big)}{\alpha^2\big( \alpha + (1-\alpha) \,\phi^{-1}(\widehat{v}\cdot\sigma)\big)}\big|\tau_{-v_*}f(v)\big|^2dvd\sigma dv_*\\
&=\frac{1}{(2\pi)^3}\int_{\R^3} g(v_*)dv_*\,\int_{\R^3\times \S^2} \frac{b_{ij}\big( \phi^{-1}(\widehat{\xi}\cdot\sigma) \big)}{\alpha^2\big( \alpha + (1-\alpha) \,\phi^{-1}(\widehat{\xi}\cdot\sigma)\big)}|\mathcal{F}(f)(\xi)|^2d\xi d\sigma\\
&=\frac{1}{(2\pi)^3}\mathcal{F}(g)(0)\int_{\R^3\times \S^2} b_{ij}(\widehat{\xi}\cdot\sigma)\big|\mathcal{F}(f)(\xi^+_{ij})\big|^2d\xi d\sigma,
\end{align*}
which achieve the proof of Lemma \ref{P1} after adding the three respective terms.
\end{proof}
\begin{cor}\label{cor-coercive}
Let $f\in L^2(\R^3)$ and $g\in L^1(\R^3)$ with $g\geq0$.  Then, 
\begin{align*}
&\int_{\mathbb{R}^{3}\times\R^3}\int_{S^{2}}b_{ij}(\widehat{u}\cdot\sigma)g(v_*)\big( f(v'_{ij})-f(v) \big)^2dvdv_*d\sigma\\
&\qquad \geq\frac{1}{(2\pi)^3}\int_{\mathbb{R}^3}|\mathcal{F}(f)(\xi)|^2\,\int_{S^{2}} b_{ij}\Big(\frac{\xi}{|\xi|}\cdot\sigma\Big)\Big[\mathcal{F}(g)(0)-\big|\mathcal{F}(g)(\xi^-_{ij})\big|\Big]d\sigma d\xi \,.
\end{align*}
\end{cor}
\begin{proof}
In Lemma \ref{P1} use Cauchy-Schwarz inequality to control the latter two terms as
\begin{align*}
\mathcal{F}(g)(\xi^-_{ij})\,\mathcal{F}(f)(\xi^+_{ij})\,\overline{\mathcal{F}(f)}(\xi) +\overline{\mathcal{F}(g)}(\xi^-_{ij})\, \overline{\mathcal{F}(f)}(\xi^+_{ij})\,\mathcal{F}(f)(\xi)\\
\leq \big|\mathcal{F}(g)(\xi^-_{ij})\big|\Big(\big|\mathcal{F}(f)(\xi)\big|^{2}+\big|\mathcal{F}(f)(\xi^+_{ij})\big|^{2}\Big)\,.
\end{align*}
The result follows after observing that for $g\geq0$ one has that $\mathcal{F}(g)(0) - \big|\mathcal{F}(g)(\xi^-_{ij})\big|\geq0$.
\end{proof}
\begin{lem}\label{Kij}
Suppose that $b_{ij}$ satisfies assumption (\ref{B1}) and (\ref{B2}).  Then, there exists a constant $K^{ij}>0$ such that
$$
\int_{\mathcal{S}^{2}}b_{ij}\Big(\frac{\xi}{|\xi|}\cdot\sigma\Big)\Big[\mathcal{F}(g)(0)-|\mathcal{F}(g)(\xi^-_{ij})|\Big]d\sigma\geq  K^{ij}\big(|\xi|^2 \wedge |\xi|^{s_{ij}} \big)\,.
$$
The constant can be taken as $K^{ij} := K^{ij}(g)=\Big(\frac{m_j}{m_i+m_j}\Big)^2\frac{1}{2-s_{ij}}\,\kappa^1_{ij} \, C_g\,|\S|$ with the constant $C_{g}>0$ depending only on $ \|g\|_{L^1_1}$ and $\|g\|_{L \log L}$.
\end{lem}
\begin{proof}
Using \cite[Lemma 3]{ADVW}, there exists a constant $C_g$ depending on $\|g\|_{L^1_1}$ and $\|g\|_{L \log L}$ such that
$$
\mathcal{F}(g)(0)-|\mathcal{F}(g)(\xi^-_{ij})|\geq C_g\,\big(|\xi^-_{ij}|^2\wedge 1\big)\geq0\,.
$$
Then,
$$
\int_{\mathcal{S}^{2}}b_{ij}\Big(\frac{\xi}{|\xi|}\cdot\sigma\Big)\big(\mathcal{F}(g_j)(0)-|\mathcal{F}(g_j)(\xi^-_{ij})|\big)d\sigma\geq C_g\int_{\mathcal{S}^{2}}b_{ij}\Big(\frac{\xi}{|\xi|}\cdot\sigma\Big)\big(|\xi^-_{ij}|^2\wedge 1\big)d\sigma\,.
$$
We note that 
$$
|\xi^-_{ij}|^2=2 \Big(\frac{m_j}{m_i+m_j}\Big)^2|\xi|^2\Big(1-\frac{\xi}{|\xi|}\cdot\sigma\Big)\,.
$$
Thus, using spherical coordinates and recalling assumption \eqref{B2} it follows that
\begin{align*}
\int_{\mathcal{S}^{2}}b_{ij}\Big(\frac{\xi}{|\xi|}\cdot\sigma\Big)\big( |\xi^-_{ij}|^2\wedge 1\big)d\sigma&= \Big(\frac{m_j}{m_i+m_j}\Big)^2|\S|\int_0^{\frac{\pi}{2}}\sin\theta\, b_{ij}(\cos\theta)\,\big[|\xi|^2(1-\cos\theta)\wedge 1\big]d\theta\\
&\geq \Big(\frac{m_j}{m_i+m_j}\Big)^2\kappa^{1}_{ij} |\S|\int_0^{\frac\pi2}\big(|\xi|^2\theta^2\wedge1\big)\frac{d\theta}{\theta^{1+s_{ij}}}\,.
\end{align*}
Using the change of variables $\tilde\theta = |\xi|\theta$, this latter integral can be estimated as
$$
|\xi|^{s_{ij}}\int_0^{\frac\pi2|\xi|} \big(\tilde{\theta}^2\wedge1\big)\frac{d\tilde\theta}{\tilde{\theta}^{1+s_{ij}}}\geq|\xi|^{s_{ij}}\int_0^{\frac\pi2} \big(\tilde{\theta}^2\wedge1\big)\frac{d\tilde{\theta}}{\tilde{\theta}^{1+s_{ij}}}\geq\frac{|\xi|^{s_{ij}}}{2-s_{ij}}\,,\quad \text{for}\quad |\xi|\geq1\,.
$$
Meanwhile, when $|\xi|\leq1$ the integral is estimated as
$$
|\xi|^{s_{ij}}\int_0^{\frac\pi2|\xi|} \big(\tilde{\theta}^2\wedge1\big)\frac{d\tilde\theta}{\tilde{\theta}^{1+s_{ij}}}\geq|\xi|^{s_{ij}}\int_0^{|\xi|} \big(\tilde{\theta}^2\wedge1\big)\frac{d\tilde{\theta}}{\tilde{\theta}^{1+s_{ij}}} = \frac{|\xi|^{2}}{2-s_{ij}}\,.
$$
This concludes the proof with the constant
$$K_{ij} = \Big(\frac{m_j}{m_i+m_j}\Big)^2\frac{\kappa^1_{ij}}{2-s_{ij}}\, C_g\,|\S|\,.$$
\end{proof} 
Corollary \ref{cor-coercive} and Lemma \ref{Kij} readily prove the main coercive estimate.
\begin{prop}[Coercivity estimate]\label{Propcoercive}
Let $f \in L^2(\R^3)$ and $g\in(L^1_{1}\cap L\log L)(\R^3)$ with $g\geq0$.  Then,
\begin{align*}
\int_{\mathbb{R}^{2N}}&\int_{\mathcal{S}^2}b_{ij}(\widehat{u}\cdot\sigma)g(v_*)\big( f(v'_{ij})-f(v) \big)^2dvdv_*d\sigma \geq \frac{K_{ij}}{(2\pi)^3}\int_{\mathbb{R}^3}|\mathcal{F}(f)(\xi)|^2\,\big(|\xi|^2 \wedge |\xi|^{s_{ij}} \big)\,d\sigma d\xi \,,
\end{align*}
with constant $K_{ij}$ given in Lemma \ref{Kij}.  And, as a consequence, it holds that
 \begin{align*}
\int_{\mathbb{R}^{2N}}&\int_{\mathcal{S}^2}b_{ij}(\widehat{u}\cdot\sigma)g(v_*)\big( f(v'_{ij})-f(v) \big)^2dvdv_*d\sigma \geq \frac{K_{ij}}{(2\pi)^3}\Big(2^{-\frac{s_{ij}}{2}}\| f \|_{ H^{\frac{s_{ij}}{2}}} - \|f\|_{L^{2}}\Big) \,.
\end{align*}
\end{prop}

\end{document}